\documentclass[11pt,a4paper]{article}
\usepackage{amsmath,amssymb,amsthm,mathrsfs,mathtools}
\usepackage[T1]{fontenc}
\usepackage{lmodern}
\usepackage{microtype}
\usepackage{tikz}
\usetikzlibrary{arrows.meta,positioning,decorations.pathreplacing,calc,patterns,shapes.geometric}
\usepackage{pgfplots}
\pgfplotsset{compat=1.18}
\usepackage{caption}
\usepackage[shortlabels]{enumitem}
\usepackage{geometry}
\usepackage{cite}
\usepackage[colorlinks=true,allcolors=blue!45!black,linktoc=all]{hyperref}
\usepackage{authblk}

\allowdisplaybreaks[4]
\numberwithin{equation}{section}

\newtheorem{theorem}{Theorem}[section]
\newtheorem{lemma}[theorem]{Lemma}
\newtheorem*{weightedchain}{Weighted chain rule}
\newtheorem{proposition}[theorem]{Proposition}
\newtheorem{corollary}[theorem]{Corollary}
\theoremstyle{definition}

\newtheorem{remark}[theorem]{Remark}
\newtheorem{conjecture}[theorem]{Conjecture}

\DeclarePairedDelimiter{\norm}{\lVert}{\rVert}
\DeclarePairedDelimiter{\abs}{\lvert}{\rvert}
\newcommand{\R}{\mathbb{R}}
\newcommand{\T}{\mathbb{T}}
\newcommand{\PP}{\mathbb{P}}
\newcommand{\bu}{\mathbf{u}}
\newcommand{\bv}{\mathbf{v}}
\newcommand{\bw}{\mathbf{w}}
\newcommand{\bn}{\mathbf{n}}
\newcommand{\be}{\mathbf{e}}
\newcommand{\bx}{\mathbf{x}}
\newcommand{\dx}{\,d\mathbf{x}}
\newcommand{\Div}{\operatorname{div}}

\newcommand{\dist}{\operatorname{dist}}
\newcommand{\rs}{\rho^{\natural}}
\newcommand{\Rr}{\mathcal{R}(\rho_0)}
\newcommand{\Kr}{\mathcal{K}(\rho_0)}
\newcommand{\Hy}{\mathscr{H}}
\newcommand{\KE}{\mathrm{KE}}
\newcommand{\APE}{\mathrm{APE}}

\title{\bfseries Relaxation to hydrostatic equilibria for two-dimensional inhomogeneous incompressible viscous flows under a conservative force}

\author[a]{Song Jiang\thanks{jiang@iapcm.ac.cn}}
\author[b]{Quan Wang\thanks{Corresponding author: xihujunzi@scu.edu.cn}}
\affil[a]{\small LCP, Institute of Applied Physics and Computational Mathematics, Huayuan Road 6, Beijing 100088, China}
\affil[b]{\small College of Mathematics, Sichuan University, Chengdu, Sichuan 610065, China}
\date{}

\begin{document}

\maketitle

\begin{abstract}
We study two-dimensional inhomogeneous incompressible Navier--Stokes flow under conservative forcing. For initial density bounded above and away from zero and solenoidal no-slip $H^2$ initial velocity, the velocity decays and weak density limits are hydrostatic. If the potential has no level sets of positive measure, strong convergence to the decreasing rearrangement is equivalent to vanishing available potential energy. For two-phase data we compute the exact energy cost of variance loss in the weak closure of the rearrangement class. The minimizing ramp changes when it reaches the channel walls.

In the nonlinear no-slip channel, a bi-Lipschitz decreasing sorted profile, a uniform $H^s$ bound and quantitative control of the sorting error of the horizontal mean yield $L^2$ density decay of order $t^{-s/4}$. At a wall of constant density, a nonconstant normal derivative forces growth of the $H^s$ norm for $s>3/2$ if the density approaches a layered state. If the potential has no level sets of positive measure, arbitrarily small Rayleigh--Taylor perturbations can release a uniformly positive fraction of the available potential energy.

For the stable linearization with variable inertia, we compute the planar Weyl coefficient and the joint channel spectral distribution for continuous stratification. At large horizontal frequency, the rescaled spectrum has harmonic-oscillator corrections at nondegenerate interior maxima and an Airy correction of order $|k|^{-2/3}$ at a unique wall maximum with nonzero inward slope. At a unique maximum of either type we determine the leading inertial correction to the eigenvalues. For initial normalized density fluctuation in $H^4$, the sharp $L^2$ decay rate is $t^{-1/8}$ for a nonzero value trace and $t^{-3/8}$ when the value trace vanishes and the normal-derivative trace is nonzero. A datum-dependent positive fraction of the squared density error lies in a strip of width $O(t^{-1/4})$ at an affected wall.


\end{abstract}

\clearpage
\tableofcontents
\clearpage

\section{Introduction}

\subsection{The problem and related results}

We consider the large-time behavior of two-dimensional incompressible viscous flows with variable density under a conservative body force. The motion is governed by the inhomogeneous incompressible Navier--Stokes equations (IINS). In a bounded domain $\Omega\subset\R^2$, we consider the no-slip initial-boundary value problem
\begin{align}\label{main}
\begin{cases}
\rho(\partial_t \bu+\bu\cdot\nabla\bu)-\nu\Delta\bu+\nabla\Pi=-\rho\nabla f, & \text{in }(0,\infty)\times\Omega,\\
\partial_t\rho+\bu\cdot\nabla\rho=0,\qquad \Div\bu=0, & \text{in }(0,\infty)\times\Omega,\\
\bu=\mathbf{0} \ \text{ on }(0,\infty)\times\partial\Omega,\qquad (\rho,\bu)|_{t=0}=(\rho_0,\bu_0), &
\end{cases}
\end{align}
Here $\rho$ is the density, $\bu$ the velocity, $\Pi$ the pressure, and $\nu>0$ the viscosity. The potential $f$ is independent of time. For gravity we take $f=gx_2$ with $g>0$ \cite{Lions1996}.

The density is transported without diffusion, while viscosity dissipates the total mechanical energy through the velocity gradient. Every steady solution has zero velocity and satisfies the hydrostatic balance $\nabla\Pi=-\rho\nabla f$ (Lemma~\ref{lem:steady}). Under gravity, connected horizontal sections force a hydrostatic density to be horizontally layered. Such densities may have different orderings, including Rayleigh--Taylor (RT) unstable configurations. Velocity decay therefore leaves a separate selection problem for the density. When the potential has no level sets of positive measure, we measure the excess potential energy relative to the decreasing rearrangement with respect to the potential, following Lorenz \cite{Lorenz1955}.

Global finite-energy weak solutions for inhomogeneous incompressible flows were constructed by Antontsev and Kazhikhov \cite{antontsev1973,kazhikov1974} and Simon \cite{simon1978,simon1990}, with the latter allowing vacuum. Fern\'andez-Cara and Guill\'en \cite{Cara1992} treated unbounded domains. We refer to Lions \cite{Lions1996} for a systematic account. The strong theory for strictly positive densities goes back to Ladyzhenskaya and Solonnikov \cite{ladyzhenskaya1978}, with subsequent work by Okamoto, Padula and Salvi \cite{okamoto1984,padula1982,padula1990,salvi1991}. For strong solutions with vacuum, Choe and Kim \cite{Choe2003} imposed an initial compatibility condition that was later removed in two dimensions by Li \cite{Li2017}. See Kim \cite{Kim1987,Kim2006} for further weak-solution and continuation results. Danchin \cite{Danchin2003,Danchin2004} developed a critical-space theory, and Danchin and Mucha \cite{DM2012} used Lagrangian coordinates for uniqueness at low density regularity. Further results for bounded density include \cite{PZZ2013,Danchin2019,Hao2025}.

The unforced setting has an extensive theory of global solvability and velocity decay \cite{Craig2013,Abidi2010,Abidi2015,Gui2011,Huang2015,Gancedo2018,He2021,Lu2017,Zhang2022,Danchin,Adogbo2025,Qian2025}. For forced three-dimensional flows, Zhang, Zhao and Zhang \cite{Zhang2014} studied global regularity with a potential force under conditions on the data or viscosity. Yu \cite{Yu2018} considered general prescribed forces under a small-mass condition. Our earlier preprint \cite{JiangWang2025Preprint} treated hydrostatic equilibria, long-time velocity decay and convergence to linear profiles in bounded domains. The present analysis allows bounded positive densities, establishes finite-time continuous dependence in Eulerian coordinates, and quantifies density selection and the role of the walls.

For heavy-over-light configurations, Hwang and Guo \cite{Hwang2003} and Jiang, Jiang and Ni \cite{Jiang2013} proved nonlinear RT instability in inviscid and viscous settings, respectively. Jiang and Jiang \cite{Jiang2014} studied instability and stability for three-dimensional bounded domains. Li, Tan and Wang \cite{Li2025} obtained instability for general conservative forces when a hydrostatic density increases in the direction of $\nabla f$ near a point. See also \cite{Mao2024,Xing2024,Fan2025} for related geophysical models and \cite{Gebhard2021} for inviscid mixing solutions. The finite-time escape supplied by the instability theory is the starting point here. We use the subsequent viscous dissipation to quantify how far the potential energy and density remain from their initial equilibrium values.

The nondiffusive Boussinesq system provides a useful comparison for relaxation. Doering, Wu, Zhao and Zheng \cite{Doering2018} proved velocity decay and asymptotic hydrostatic balance in $H^{-1}$. Stability of stratified configurations has also been studied for partially dissipative or damped systems \cite{Tao2020,Castro2019,Lai2021,Adhikari2022}, and convergence on domains with slip conditions has been considered in \cite{BCN2023,JW2025,JW20252}. Ayd{\i}n and Jayanti \cite{AJ2025} studied regularity and asymptotic behavior on bounded domains with no-slip conditions. For a no-slip periodic channel, Park \cite[Theorem~1.3]{Park2024} proved algebraic convergence to the decreasing vertical rearrangement for small density and velocity perturbations in $H^4\cap H^2_0$ in the Boussinesq system. Dalibard, Guillod and Leblond \cite{DGL2025} established nonlinear relaxation near an affine stratification and derived boundary layer profiles for the Stokes--transport system. Examples of small-scale formation appear in \cite{KY2023,KPY2022}. For the full variable-inertia equations, Dong, Li and Sun \cite[Theorem~1.1]{DLS2026} proved small-perturbation stability and explicit decay in the three-dimensional no-slip strip $\R^2\times(0,1)$. Their assumptions include $H^4$ regularity, horizontal $L^1$ integrability of the initial velocity and the density perturbation's Laplacian with values in vertical $L^2$, and boundary compatibility. In the two-dimensional periodic channel, we derive nonlinear rates under a uniform-in-time density bound and a condition on vertical ordering. The small-data result for the full variable-inertia system in Theorem~\ref{thm:wall-rate} assumes uniform $H^4$ control and also quantifies convergence of the flow maps.

\begin{enumerate}[(i)]
\item \emph{Energy selection and the cost of mixing.}
For two-phase densities we compute the exact minimum excess potential energy at fixed variance loss over the weak closure of the rearrangement class. The minimizing ramp and the energy law change when the ramp contacts a wall. This minimum bounds the variance loss of every weak density limit (Theorem~\ref{thm:mix-frontier}).

\item \emph{Slow spectrum.}
We compute the planar Weyl coefficient and the joint counting law for the two channel mode indices. For each fixed vertical branch, nondegenerate interior maxima give curvature-dependent harmonic-oscillator corrections, and a unique wall maximum with nonzero inward slope gives an Airy correction. At a unique maximum we also determine the leading contribution of inertia (Theorems~\ref{thm:planar-spectrum} and~\ref{thm:channel-spectrum}).

\item \emph{Wall traces and decay.}
In the variable-inertia linearization, the first two incompatible traces give sharp $L^2$ density rates $t^{-1/8}$ and $t^{-3/8}$, and the lower bounds already hold on strips of width $O(t^{-1/4})$ (Theorem~\ref{thm:channel-traces}). For the nonlinear equations, uniform Sobolev control and a condition on the horizontal mean give conditional sorting rates, while the wall trace of the normal derivative yields a corresponding obstruction to uniform Sobolev control (Theorems~\ref{thm:rate} and~\ref{thm:wall-rate} and Proposition~\ref{thm:wall-obstruction}).
\end{enumerate}

The hydrostatic-limit and total-dissipation criteria (Theorems~\ref{thm:hydro} and~\ref{thm:lorenz}) underlie the energy and convergence results. Invariant flows can retain an unstable density, whereas a family of arbitrarily small Rayleigh--Taylor perturbations undergoes a definite potential-energy loss (Theorems~\ref{thm:nonmono} and~\ref{thm:RT}). Section~\ref{sec:linear} includes periodic Boussinesq and stress-free channel comparisons.

For comparison with the literature, the perturbation formulation is equivalent to \eqref{main}. With $\varrho=\rho-\rho_s$ and $P=\Pi-p_s$, where $\nabla p_s=-\rho_s\nabla f$, one obtains
\begin{align}\label{main-pert}
\begin{aligned}
&(\varrho+\rho_s)(\partial_t\bu+\bu\cdot\nabla\bu)=\nu\Delta\bu-\nabla P-\varrho\nabla f,\\
&\partial_t\varrho+\bu\cdot\nabla(\varrho+\rho_s)=0,\qquad \Div\bu=0 .
\end{aligned}
\end{align}
The estimates for the full density can therefore be applied with any hydrostatic background $\rho_s$, stable or unstable.

\subsection{Assumptions and notation}

Unless otherwise specified, $\norm{\cdot}$ denotes the $L^2(\Omega)$ norm, and spatial integrals are taken over $\Omega$. We use the convention $\nabla^\perp=(\partial_2,-\partial_1)$. The standing assumptions are as follows.
\begin{enumerate}[label=\textbf{(A\arabic*)}]
\item\label{A1} $\Omega\subset\R^2$ is a bounded domain with $C^{2,1}$ boundary, and $f\in W^{2,\infty}(\Omega)$. We set $m\coloneqq\operatorname*{ess\,inf} f$ and $M\coloneqq\operatorname*{ess\,sup} f$.
\item\label{A2} $\rho_0$ is measurable with $0<\alpha_1\le\rho_0\le\alpha_2<\infty$ a.e.
\item\label{A3} $\bu_0\in H^1_0(\Omega)\cap H^2(\Omega)$ and $\Div\bu_0=0$.
\item\label{A4} (No flat level sets.) $|\{\bx\in\Omega: f(\bx)=c\}|=0$ for every $c\in\R$; e.g.\ $\nabla f\neq0$ a.e.
\end{enumerate}
For results stated in a periodic channel, the channel version of \ref{A1} uses $\Omega=\T_L\times(0,h)$, with $\T_L=\R/L\mathbb Z$ and no-slip walls at $x_2=0,h$. In this geometry the Sobolev spaces are periodic in $x_1$, and the zero traces indicated by $H^j_0$ are imposed at the walls. $C$ denotes a constant depending only on $\alpha_1,\alpha_2,\nu,\Omega,\norm{f}_{W^{2,\infty}}$; $C_0$ may in addition depend on $\norm{\bu_0}_{H^2}$.

\paragraph{Energies.}
$E(t)\coloneqq\tfrac12\int\rho|\bu|^2+\int\rho f$ and $\KE_0\coloneqq\tfrac12\int\rho_0|\bu_0|^2$.

\paragraph{Hydrostatic densities.} Let $L^2(\Omega)^2=H\oplus G$ be the Helmholtz--Leray decomposition, where $H=\{\bv\in L^2:\Div\bv=0,\ \bv\cdot\bn|_{\partial\Omega}=0\}$ and $G=\{\nabla q: q\in H^1(\Omega)\}$, and let $\PP$ be the projection onto $H$ \cite{Temam1977,Galdi2011}. We set
\[
\Hy\coloneqq\{\sigma\in L^\infty:\ \PP(\sigma\nabla f)=0\}=\{\sigma\in L^\infty:\ \sigma\nabla f=\nabla q,\ q\in H^1(\Omega)\}.
\]
If $\sigma\nabla f=\nabla q$ with $\sigma\in\Hy$, then $(\rho,\bu,\Pi)=(\sigma,\mathbf 0,-q)$ is a steady solution.

\paragraph{Rearrangements.} For measurable $g$ let $\mu_g(c)\coloneqq|\{g>c\}|$. The rearrangement class of $\rho_0$ is $\Rr\coloneqq\{\sigma:\mu_\sigma=\mu_{\rho_0}\}$, and $\Kr$ denotes its weak-$*$ closure in $L^\infty$. On bounded sets this topology coincides with the weak $L^2$ topology and is metrizable. Let $\rho_0^*(s)=\inf\{c:\mu_{\rho_0}(c)\le s\}$, $s\in(0,|\Omega|)$, be the nonincreasing rearrangement. At the endpoints of its interval we use the one-sided limits, with the same convention for other bounded rearrangements. Under \ref{A4} the function $\Lambda(c)\coloneqq|\{f<c\}|$ is continuous. Following Lorenz \cite{Lorenz1955}, we define the \emph{sorted (reference) state} and the \emph{available potential energy}
\begin{equation}\label{def-sorted}
\rs\coloneqq\rho_0^*\circ\Lambda\circ f,\qquad E^\natural\coloneqq\int\rs f,\qquad \APE(t)\coloneqq\int(\rho(t)-\rs)f,\qquad\APE_0\coloneqq\APE(0).
\end{equation}
For $f=gx_2$ with $g>0$ in $\Omega=\T_L\times(0,h)$, $\rs=\rho_0^*(Lx_2)$ is the decreasing vertical rearrangement of \cite{Park2024,DGL2025}.

\subsection{Statement of the main results}

\subsubsection*{A. Well-posedness, decay and hydrostatic limits}
\begin{theorem}\label{thm:gwp}
Assume \ref{A1}--\ref{A3}. Then \eqref{main} has a unique global solution $(\rho,\bu,\nabla\Pi)$ with the following properties.
\begin{enumerate}[(1)]
\item $\rho\in C([0,\infty);L^p(\Omega))$ for all $1\le p<\infty$, and $\alpha_1\le\rho\le\alpha_2$. Moreover $\rho(t)=\rho_0\circ X_t^{-1}$, where $X_t$ is a measure-preserving bi-Lipschitz homeomorphism of $\overline\Omega$. In particular $\rho(t)\in\Rr$ for all $t\ge0$.
 \item $\bu\in C([0,\infty);H^2)$ and $\partial_t\bu,\nabla\Pi\in C([0,\infty);L^2)$. Moreover,
\begin{equation}\label{gwp-1}
\sup_{t\ge0}\Big(\norm{\bu}_{H^2}^2+\norm{\partial_t\bu}^2+\norm{\nabla\Pi}^2\Big)
+\int_0^\infty\Big(\norm{\nabla\bu}^2+\norm{\partial_t\bu}_{H^1}^2+\norm{\bu}_{W^{1,p}}^p\Big)\,dt\le C_0
\end{equation}
for every $p\in[2,\infty)$ (with $C_0$ depending also on $p$).
\item For $p\in(2,\infty)$ and $T>0$, $\int_0^T\norm{\bu}_{W^{2,p}}^2dt\le C_0(1+T)$. The accumulated strain $\Lambda(T)=\int_0^T\norm{\nabla\bu(t)}_{L^\infty}\,dt$ satisfies, for each $\beta>3/4$,
\[
 \Lambda(T)\le C_\beta(1+T)^\beta,qquad
 \Lambda(T)=o(T^\beta)\quad\text{as }T\to\infty.
\]
If $\rho_0\in W^{1,q}(\Omega)$ with $q\in[1,\infty]$, then
$\norm{\nabla\rho(t)}_{L^q}\le e^{\Lambda(t)}\norm{\nabla\rho_0}_{L^q}$.
\end{enumerate}
\end{theorem}

\begin{proposition}\label{prop:euler-dependence}
For two solutions of Theorem~\ref{thm:gwp} with the same force and with initial velocities in a bounded subset of $H^2$, set $U=\bu_1-\bu_2$ and $d=\rho_1-\rho_2$. On every finite interval $[0,T]$,
\begin{equation}\label{euler-dependence}
 \sup_{0\le t\le T}\big(\norm{U(t)}^2+\norm{d(t)}_{W^{-1,4}}^2\big)
 +\int_0^T\norm{\nabla U}^2\,dt
 \le C_T\big(\norm{U(0)}^2+\norm{d(0)}_{W^{-1,4}}^2\big),
\end{equation}
where $W^{-1,4}=(W^{1,4/3}_0)'$. The constant is uniform for the bounds \ref{A1}--\ref{A3} on this set.
\end{proposition}

\begin{theorem}\label{thm:decay}
Under the assumptions of Theorem~\ref{thm:gwp}, the following hold.
\begin{enumerate}[(1)]
\item $E$ is nonincreasing and $E(t)+\nu\int_0^t\norm{\nabla\bu}^2=E(0)$. The limit $E_\infty\coloneqq\lim_{t\to\infty}E(t)=\lim_{t\to\infty}\int\rho(t)f$ exists.
\item As $t\to\infty$: $\norm{\bu(t)}_{W^{1,r}}\to0$ for every $1\le r<\infty$, $\norm{\bu(t)}_{L^\infty}\to0$, and $\norm{\partial_t\bu(t)}\to0$.
\item Let $\bw(t)\coloneqq\PP(\rho(t)\nabla f)$. Then $\Delta\bu(t)$, $\bw(t)$ and $\nabla\Pi(t)+\rho(t)\nabla f$ converge to $0$ weakly in $L^2$ and strongly in $H^{-1}$. Moreover
\[
\norm{\nu\PP\Delta\bu-\bw}\to0,\qquad\norm{\nu(\mathbb I-\PP)\Delta\bu-\nabla\Pi-(\mathbb I-\PP)(\rho\nabla f)}\to0 .
\]
\item The following are equivalent:
(a) $\norm{\bu(t)}_{H^2}\to0$;
(b) $\norm{\Delta\bu(t)}\to0$;
(c) $\norm{\bw(t)}\to0$;
(d) $\norm{\nabla\Pi(t)+\rho(t)\nabla f}\to0$.
\item (Integrability of the baroclinic torque.) Equip $V\coloneqq H^1_0\cap H$ with the gradient norm and denote its dual by $V'$. Let $H^{-2}\coloneqq(H^2_0(\Omega))'$, and $J(t)\coloneqq\nabla^\perp f\cdot\nabla\rho(t)=\Div(\rho\nabla^\perp f)$. Then
\begin{equation}\label{baroclinic}
\int_0^\infty\Big(\norm{\bw}_{V'}^2+\norm{J}_{H^{-2}}^2+\norm{\nabla\Pi+\rho\nabla f}_{H^{-1}}^2\Big)\,dt\le C_0 .
\end{equation}
In particular, for every $\varepsilon>0$ and $T>0$,
\[
\abs[\big]{\{t\ge0:\norm{J(t)}_{H^{-2}}>\varepsilon\}}\le C_0\,\varepsilon^{-2},\qquad \min_{t\in[T,2T]}\norm{J(t)}_{H^{-2}}=o\big(T^{-1/2}\big).
\]
\item (Dissipation from a velocity peak.) Let $C_1>0$ bound $\sup_{t\ge0}\norm{\partial_t\bu(t)}$, and let $C_P$ be the Poincar\'e constant of $H^1_0(\Omega)$. For every $T\ge0$ and $\tau\ge0$, with $m_T=\norm{\bu(T)}$, one has
\begin{equation}\label{burst-local}
 E(T)-E(T+\tau)\ge \frac{\nu}{3C_P^2 C_1}
 \big[m_T^3-(m_T-C_1\tau)_+^3\big].
\end{equation}
Including the dissipation before $T$ gives
\begin{equation}\label{burst-two-sided}
 E(0)-E_\infty\ge\frac{\nu}{3C_P^2 C_1}
    \big[2m_T^3-(m_T-C_1 T)_+^3\big].
\end{equation}
In particular, with $M_u=\sup_{t\ge0}\norm{\bu(t)}$,
\begin{equation}\label{burst}
 E(0)-E_\infty\ge\frac{\nu}{3C_P^2 C_1}
    \big(2M_u^3-\norm{\bu_0}^3\big).
\end{equation}
For any weak density limit $\rho_\infty$, the same estimates bound its potential energy because $\int\rho_\infty f=E_\infty$.
\end{enumerate}
\end{theorem}

\begin{theorem}\label{thm:hydro}
Under the assumptions of Theorem~\ref{thm:gwp}, let
\[
\mathcal S\coloneqq\Big\{\sigma\in\Hy\cap\Kr:\ \int_\Omega\sigma f\dx=E_\infty\Big\}.
\]
\begin{enumerate}[(1)]
\item $\mathcal S$ is nonempty and weak-$*$ compact. Every sequence $t_n\to\infty$ has a subsequence along which $\rho(t_n)\overset{*}{\rightharpoonup}\rho_\infty\in\mathcal S$ and $\nabla\Pi(t_n)\rightharpoonup\nabla\Pi_\infty=-\rho_\infty\nabla f$ in $L^2$. Thus every limit point is a hydrostatic equilibrium.
\item $\dist_{H^{-1}}(\rho(t),\mathcal S)\to0$ as $t\to\infty$.
\item (Mixing dichotomy.) If $\rho(t_n)\rightharpoonup\rho_\infty$, then $\rho(t_n)\to\rho_\infty$ strongly in $L^2$ if and only if $\norm{\rho_\infty}=\norm{\rho_0}$, if and only if $\rho_\infty\in\Rr$.
\item If $\{\rho(t)\}_{t\ge0}$ is precompact in $L^2$, then $\dist_{L^2}(\rho(t),\mathcal S\cap\Rr)\to0$ and the equivalent statements (a)--(d) of Theorem~\ref{thm:decay}(4) hold.
\item If $f=gx_2$ ($g\neq0$) and every horizontal section $\{x_1:(x_1,c)\in\Omega\}$ is an interval (e.g.\ $\Omega$ convex), then every $\sigma\in\Hy$ depends only on $x_2$. Hence all limit points of $\rho(t)$ are horizontally layered.
\item (Lipschitz-flow criterion.) If $\int_0^\infty\norm{\nabla\bu(t)}_{L^\infty}dt<\infty$, then $X_t\to X_\infty$ uniformly, where $X_\infty$ is a measure-preserving bi-Lipschitz homeomorphism of $\overline\Omega$. Moreover $\rho(t)\to\rho_\infty\coloneqq\rho_0\circ X_\infty^{-1}$ in every $L^p$, $p<\infty$, with $\rho_\infty\in\mathcal S\cap\Rr$, and (a)--(d) of Theorem~\ref{thm:decay}(4) hold. If in addition $\rho_0\in W^{1,q}$, then $\sup_t\norm{\nabla\rho(t)}_{L^q}<\infty$. Conversely, if some weak limit point of $\rho(t)$ is not a rearrangement of $\rho_0$ (mixing), then $\int_0^\infty\norm{\nabla\bu}_{L^\infty}=\infty$.
\item (Structure of the limit set.) The set of actual density limit points is compact and connected in $H^{-1}$. The set $\mathcal S$ is convex. If $\mathcal S=\{\sigma\}$, the density converges to $\sigma$ in $H^{-1}$ and weakly in $L^2$, and converges in $L^2$ precisely when $\sigma\in\Rr$.
\end{enumerate}
\end{theorem}

\subsubsection*{B. Selection of the limit and stability}
\begin{theorem}\label{thm:lorenz}
Under \ref{A1}--\ref{A4}, $\rs\in\Rr\cap\Hy$ is a nonincreasing function of $f$, and $\APE_0\ge0$.
\begin{enumerate}[(1)]
\item For every $t\ge0$,
\begin{equation}\label{dissip-bound}
\int\rho(t)f\ge\int\rs f,\qquad
\nu\int_0^\infty\norm{\nabla\bu}^2dt=E(0)-E_\infty\le\KE_0+\APE_0 .
\end{equation}
\item The following are equivalent:
(a) $\rho(t)\to\rs$ in $L^2$;
(b) $\rs$ is a weak-$*$ limit point of $\rho(t)$;
(c) $E_\infty=\int\rs f$;
(d) $\nu\int_0^\infty\norm{\nabla\bu}^2\,dt=\KE_0+\APE_0$.
If one of them holds, then $\norm{\bu(t)}_{H^2}+\norm{\nabla\Pi(t)+\rho(t)\nabla f}\to0$.
\item (Linear profiles.) Let $\gamma>0$, $\beta\in\R$ and $\sigma_{\gamma,\beta}\coloneqq\beta-\gamma f$. For all $t\ge0$,
\begin{equation}\label{lin-identity}
\norm{\rho(t)-\sigma_{\gamma,\beta}}^2=\norm{\rho_0-\sigma_{\gamma,\beta}}^2-2\gamma\int(\rho_0-\rho(t))f
\ \ge\ \norm{\rs-\sigma_{\gamma,\beta}}^2 .
\end{equation}
Hence $\lim_{t\to\infty}\norm{\rho(t)-\sigma_{\gamma,\beta}}^2=\norm{\rho_0-\sigma_{\gamma,\beta}}^2-2\gamma\big(\int\rho_0f-E_\infty\big)$. Moreover, $\rho(t)\to\sigma_{\gamma,\beta}$ in $L^2$ if and only if $\rho_0$ is equimeasurable with $\sigma_{\gamma,\beta}$ \emph{and} the total dissipation equals $\KE_0+\APE_0$. If the distribution of $\rho_0$ has an atom, then no profile $\beta-\gamma f$ with $\gamma>0$ can be a strong $L^2$ limit.
\item (Quantitative relaxation.) Write $\rs=\phi(f)$ with $\phi$ nonincreasing (Lemma~\ref{lem:sorted}), and recall $\APE(t)\ge0$ from (1). If $\phi$ is $\delta_1$-Lipschitz on $[m,M]$, then for all $t\ge0$
\begin{equation}\label{quant-lorenz}
\begin{gathered}
\frac1{2\delta_1}\norm{\rho(t)-\rs}^2\le\APE(t)\le E(t)-E^\natural,\\
\norm{\rho(t)-\rs}^2+2\delta_1\Big(\frac12\norm{\sqrt\rho\,\bu(t)}^2+\nu\int_0^t\norm{\nabla\bu}^2\,ds\Big)\le2\delta_1(\KE_0+\APE_0).
\end{gathered}
\end{equation}
In particular every weak limit point $\rho_\infty$ satisfies
\[
\norm{\rho_\infty-\rs}^2\le2\delta_1(E_\infty-E^\natural)\le2\delta_1\Big(\KE_0+\APE_0-\frac{\nu}{3C_P^2C_1}\big(2M_u^3-\norm{\bu_0}^3\big)\Big),
\]
by \eqref{burst}: the upper bound on the limiting distance decreases as $\sup_t\norm{\bu(t)}$ increases. If moreover $-\phi'\ge\delta_0>0$ a.e., then $\APE(t)\le\tfrac1{2\delta_0}\norm{\rho(t)-\rs}^2$, so $\APE$ is equivalent to the squared $L^2$ distance to $\rs$, and
\[
\norm{\rho(t)-\rs}^2\le\frac{\delta_1}{\delta_0}\norm{\rho_0-\rs}^2+2\delta_1\KE_0\qquad(t\ge0).
\]
\end{enumerate}
\end{theorem}

\begin{theorem}\label{thm:rough-sorting}
Assume \ref{A2} and \ref{A4}. Let $\Delta_\rho=\operatorname*{ess\,sup}\rho_0-\operatorname*{ess\,inf}\rho_0>0$, and define
\[
 \omega_f(\ell)=\sup_{c\in\R}|\{\bx\in\Omega:c<f(\bx)<c+\ell\}|,\qquad
 \mathcal D(\sigma)=\int_\Omega(\sigma-\rs)f\,\dx .
\]
Then every $\sigma\in\Kr$ satisfies, for every $\ell>0$,
\begin{equation}\label{rough-sort-modulus}
 \norm{\sigma-\rs}_{L^1}
 \le 2\Delta_\rho\,\omega_f(\ell)+\ell^{-1}\mathcal D(\sigma).
\end{equation}
If $\omega_f(\ell)\le L_f\ell^\beta$ for some $\beta\in(0,1]$, set $q_\beta=1+1/\beta$. For every $1\le p<\infty$,
\[
 \norm{\sigma-\rs}_{L^p}^{p q_\beta}
 \le 2^{q_\beta-1}\frac{\beta+1}{\beta}
 L_f^{1/\beta}\Delta_\rho^{p q_\beta-1}\mathcal D(\sigma).
\]
For $\beta=1$, the estimate holds with the explicit constant
\begin{equation}\label{rough-sort-sharp}
 \norm{\sigma-\rs}_{L^p}^{2p}
 \le 4 L_f\Delta_\rho^{\,2p-1}\mathcal D(\sigma).
\end{equation}
In the gravitational periodic channel one may take $L_f=L/g$. For a two-valued density with an interior interface, horizontal layers interchanged across that interface give equality in \eqref{rough-sort-sharp}.
\end{theorem}

\begin{corollary}\label{cor:rough-ape}
Under \ref{A1}--\ref{A4} and $\omega_f(\ell)\le L_f\ell$, one has
\[
 \norm{\rho(t)-\rs}_{L^p}^{2p}\le 4L_f\Delta_\rho^{\,2p-1}\APE(t),
 \qquad
 \norm{\rho_\infty-\rs}_{L^p}^{2p}
 \le 4L_f\Delta_\rho^{\,2p-1}(E_\infty-E^\natural)
\]
for every weak limit point $\rho_\infty$ and every $1\le p<\infty$. In the setting of Proposition~\ref{thm:wall-obstruction}(1) with $q=\rs$, if $\Lambda_\infty=\lim_{t\to\infty}\Lambda(t)<\infty$, then
\[
 E_\infty-E^\natural\ge c e^{-2\Lambda_\infty}>0 .
\]
\end{corollary}

\begin{theorem}\label{thm:mix-frontier}
Let $\Omega=\T_L\times(0,h)$, $f=gx_2$ with $g>0$, and assume that $\rho_0$ takes two distinct values $a<b$, each on a set of positive measure. Write $\Delta=b-a$ and let $Lz_0=|\{\rho_0=b\}|$. Then $\Kr$ consists exactly of $\sigma=a+\Delta q$ with $0\le q\le1$ and $\int_\Omega q=Lz_0$. Set
\[
 q_*=\mathbf 1_{\{x_2<z_0\}},\qquad
 \mathcal M(q)=\int_\Omega q(1-q),\qquad
 D(q)=g\Delta\int_\Omega x_2(q-q_*) .
\]
The quadratic mixing deficit is
$\Delta^2\mathcal M(q)=\norm{\rho_0}^2-\norm{\sigma}^2$. Every admissible profile satisfies
\begin{equation}\label{mix-parabola}
 D(q)\ge\frac{3g\Delta}{2L}\mathcal M(q)^2 .
\end{equation}
The coefficient is sharp. For $0<\mathcal M(q)\le L\min(z_0,h-z_0)/3$, equality is attained only by the linear transition of half-width $3\mathcal M(q)/L$ centered at $z_0$, clipped to $[0,1]$.

More generally, for each prescribed mixing deficit strictly between zero and its maximum $Lz_0(1-z_0/h)$, the unique minimizer of $D$ is
\begin{equation}\label{mix-ramp}
 q_\ell(x_2)=\min\left(1,\max\left(0,\frac{c_\ell-x_2}{2\ell}\right)\right).
\end{equation}
The parameter $c_\ell$ enforces the mass constraint and $\ell>0$ is uniquely fixed by the mixing deficit. At zero mixing the minimizer is $q_*$, and at maximal mixing it is the constant $z_0/h$. At each intermediate mixing deficit, every admissible $q$ with $\mathcal M(q)=\mathcal M(q_\ell)$ satisfies
\begin{equation}\label{mix-stability}
 D(q)-D(q_\ell)\ge g\Delta\ell\norm{q-q_\ell}^2.
\end{equation}
To give the minimum explicitly, write $s_0=\min(z_0,h-z_0)$, $m_q=\mathcal M(q)/L$, and $m_{\max}=s_0(1-s_0/h)$. The minimum $F(m_q)$ of $D/(g\Delta L)$ at fixed $m_q$ is
\begin{equation}\label{mix-exact}
\begin{cases}
 \dfrac32 m_q^2,
   &0\le m_q\le s_0/3,\\[4pt]
 \dfrac{4s_0^3}{9(s_0-m_q)}-\dfrac{s_0^2}{2},
   &s_0/3\le m_q\le s_0-\dfrac{4s_0^2}{3h},\\[5pt]
 \dfrac{s_0(h-s_0)}{2}-\dfrac{h^{3/2}}{\sqrt{12}}\sqrt{m_{\max}-m_q},
   &s_0-\dfrac{4s_0^2}{3h}\le m_q\le m_{\max}.
\end{cases}
\end{equation}
On $(0,m_{\max})$ the function $F$ is continuously differentiable and strictly convex, with $F'(m_q)=\ell$ at the profile realizing $m_q$. \autoref{fig:mix-frontier}(a) shows the minimum energy at fixed variance loss. The junctions mark first contact with a wall and the onset of a transition spanning the channel. When $z_0=h/2$, both contacts occur at the same parameter and the middle branch reduces to one point. Panel (b) shows how these contacts change the optimizing profile. For every weak limit $\rho_\infty=a+\Delta q_\infty$ of \eqref{main},
\[
 E_\infty-E^\natural\ge g\Delta L F\big(\mathcal M(q_\infty)/L\big).
\]
In particular,
\begin{equation}\label{mix-actual}
 \big(\norm{\rho_0}^2-\norm{\rho_\infty}^2\big)^2
 \le \frac{2L\Delta^3}{3g}\,(E_\infty-E^\natural).
\end{equation}
\end{theorem}

\begin{figure}[tbp]
\centering
\begin{tikzpicture}
\begin{axis}[
  name=frontier,width=0.50\textwidth,height=5.4cm,
  xmin=0,xmax=0.225,ymin=0,ymax=0.114,
  xlabel={$m_q/h$},ylabel={$F(m_q)/h^2$},
  xtick={0,0.1,0.18,0.21},xticklabels={$0$,$0.10$,$0.18$,$0.21$},
  ytick={0,0.03,0.06,0.09},scaled y ticks=false,
  yticklabels={$0$,$0.03$,$0.06$,$0.09$},tick label style={font=\scriptsize},
  label style={font=\small},title={(a) Minimum residual energy},
  title style={font=\small},
]
\addplot[gray!65,very thin,densely dotted] coordinates {(0.10,0) (0.10,0.015)};
\addplot[gray!65,very thin,densely dotted] coordinates {(0.18,0) (0.18,0.055)};
\addplot[blue!75!black,very thick,domain=0:0.10,samples=40] {1.5*x^2};
\addplot[teal!70!black,very thick,dashed,domain=0.10:0.18,samples=40]
  {0.012/(0.3-x)-0.045};
\addplot[red!70!black,very thick,dash dot,domain=0.1732050808:0,samples=50]
  ({0.21-x^2},{0.105-x/sqrt(12)});
\addplot[only marks,mark=*,mark size=1.25pt,black]
  coordinates {(0.10,0.015) (0.18,0.055) (0.21,0.105)};
\end{axis}
\begin{axis}[
  at={(frontier.east)},anchor=west,xshift=1.55cm,
  width=0.43\textwidth,height=5.4cm,xmin=0,xmax=1,ymin=0,ymax=1.08,
  xlabel={$x_2/h$},ylabel={$q_\ell(x_2)$},xtick={0,0.5,1},ytick={0,0.5,1},
  tick label style={font=\scriptsize},label style={font=\small},
  title={(b) Minimizing profiles},title style={font=\small},
]
\addplot[blue!75!black,very thick]
  coordinates {(0,1) (0.15,1) (0.45,0) (1,0)};
\addplot[teal!70!black,very thick,dashed]
  coordinates {(0,{sqrt(0.6)}) ({sqrt(0.6)},0) (1,0)};
\addplot[red!70!black,very thick,dash dot]
  coordinates {(0,{1.22/2.4}) (1,{0.22/2.4})};
\end{axis}
\end{tikzpicture}
\caption{The exact two-phase energy bound \eqref{mix-exact} and the minimizing
profiles for $z_0/h=0.30$. In (a), every admissible pair lies on or above the
curve. The bottom wall first meets the linear transition at $m_q/h=0.10$,
and the top wall meets it at $m_q/h=0.18$. Panel (b) shows one profile from
each regime, with $\ell/h=0.15$, $0.50$ and $1.20$, respectively.}
\label{fig:mix-frontier}
\end{figure}

\begin{theorem}\label{thm:casimir}
Assume \ref{A1}--\ref{A3}. Let $\rho_s=\varphi(f)$ take values in $[\alpha_1,\alpha_2]$, where $\varphi\in C^1([m,M])$ satisfies $0<\delta_0\le-\varphi'\le\delta_1$. Then $\rho_s\in\Hy$, and for all $t\ge0$,
\begin{equation}\label{casimir}
\frac12\norm{\sqrt\rho\,\bu(t)}^2+\frac1{2\delta_1}\norm{\rho(t)-\rho_s}^2+\nu\int_0^t\norm{\nabla\bu}^2
\le\frac12\norm{\sqrt{\rho_0}\,\bu_0}^2+\frac1{2\delta_0}\norm{\rho_0-\rho_s}^2 .
\end{equation}
Consequently $(\mathbf 0,\rho_s)$ is nonlinearly Lyapunov stable in $L^2\times L^2$ (and in $L^2\times L^p$ for all $1\le p<\infty$), with no smallness condition. Every weak limit point $\rho_\infty\in\mathcal S$ satisfies $\norm{\rho_\infty-\rho_s}^2\le\delta_1\norm{\sqrt{\rho_0}\bu_0}^2+\tfrac{\delta_1}{\delta_0}\norm{\rho_0-\rho_s}^2$.
\end{theorem}

\begin{remark}
For the stable profiles covered by Theorem~\ref{thm:casimir}, the perturbation remains small in the energy topology. Even arbitrarily close initial densities can have rearrangement classes that exclude strong convergence to $\rho_s$, as shown by Theorem~\ref{thm:lorenz}(3). For $\varphi(r)=\beta-\gamma r$, the relative Casimir is $\tfrac1{2\gamma}\norm{\rho-\rho_s}^2$, and \eqref{casimir} is an identity with the exact slope bounds $\delta_0=\delta_1=\gamma$. The energy-drop result for unstable equilibria is given in Theorem~\ref{thm:RT}.
\end{remark}

\subsubsection*{C. Conditional rates and the role of the walls}
In the periodic channel $\Omega=\T_L\times(0,h)$ with $f=gx_2$ and $g>0$, we write $\bar v(x_2)\coloneqq L^{-1}\int_{\T_L}v\,dx_1$ for horizontal means.

\begin{theorem}\label{thm:rate}
Let $\Omega=\T_L\times(0,h)$ with no-slip walls and $f=gx_2$, $g>0$. Assume \ref{A2}--\ref{A3} and define the sorting error of the horizontal mean by $d(t)=\norm{\bar\rho(t)-(\bar\rho(t))^\natural}$. Consider the following conditions.
\begin{enumerate}[(i)]
\item $\rho_0^*$ is bi-Lipschitz on $(0,Lh)$: $a_0\le-(\rho_0^*)'\le a_1$ a.e., for some $0<a_0\le a_1$;
\item $\sup_{t\ge0}\norm{\rho(t)}_{H^s}\le B$ for some $s>0$ and $B<\infty$;
\item for every $t\ge0$, the horizontal mean $\bar\rho(t,\cdot)$ is nonincreasing on $(0,h)$.
\end{enumerate}
Then the following hold.
\begin{enumerate}[(a)]
\item Under (iii) alone, every weak limit point of $\rho(t)$ is a nonincreasing layered state $\sigma(x_2)$.
\item Suppose that for some $s>0$ there is a sequence $t_j\to\infty$ with $\sup_j\norm{\rho(t_j)}_{H^s}<\infty$ and $d(t_j)\to0$. Then $\rho(t)\to\rs$ in $L^2$ and $\norm{\bu(t)}_{H^2}+\norm{\nabla\Pi(t)+\rho(t)\nabla f}\to0$. In particular, (ii)--(iii) suffice.
\item Assume (i)--(ii). If $d(t)\le K\norm{\rho(t)-\bar\rho(t)}$ for all $t\ge0$ and some fixed $K\ge0$, then
\begin{equation}\label{rate}
\big(E(t)-E^\natural\big)+\norm{\rho(t)-\rs}^2+\norm{\bu(t)}_{H^1}^2+\norm{\partial_t\bu(t)}^2\le C(1+t)^{-s/2}.
\end{equation}
In addition,
\begin{equation}\label{conditional-stokes}
 \norm{\bu(t)}_{H^2}+\norm{\nabla\Pi(t)+\rs\nabla f}
 \le C(1+t)^{-s/4}.
\end{equation}
The constant depends only on $C_0,B,s,a_0,a_1,g,L,h,K$. In particular, (iii) gives this case with $K=0$. If $s>7/2$, then for $0<\gamma<(s-7/2)/4$,
\begin{equation}\label{conditional-strain-tail}
 \int_t^\infty \norm{\nabla\bu(r)}_{L^\infty}\,dr
 \le C_\gamma(1+t)^{-\gamma}.
\end{equation}
Under the same condition $s>7/2$, the flow maps converge in $C^1$ to a measure-preserving diffeomorphism $X_\infty$ with $\rho_0=\rs\circ X_\infty$ almost everywhere.
\item Assume (i)--(ii). If $d(t)\le D_*(1+t)^{-r}$ for some $D_*\ge0$ and $r>0$, the estimates \eqref{rate} and \eqref{conditional-stokes} hold with exponents $\min(s/2,2r)$ and $\min(s/4,r)$, respectively, and a constant depending only on $C_0,B,s,a_0,a_1,g,L,h,D_*,r$.
\end{enumerate}
The frozen profiles in Theorem~\ref{thm:nonmono}(3) show that (i)--(ii) alone do not imply sorting.
\end{theorem}

\begin{remark}\label{rem:rate}
For $0<s\le1$, finite accumulated strain and $\rho_0\in H^s$ give (ii) by bi-Lipschitz composition. Nonincreasing vertical profiles give (iii). The inverse-distribution calculation in Corollary~\ref{thm:lipred} verifies (i) if the initial vertical derivative is uniformly negative and the wall values are constant. The linear lower bound in Theorem~\ref{thm:lin}(6) uses a sequence of initial data with bounded $H^s$ norms. Its squared density error has the same time exponent as \eqref{rate}, which applies to nonlinear trajectories satisfying the uniform-in-time $H^s$ bound and the condition on $d(t)$.
\end{remark}

\begin{corollary}\label{thm:lipred}
Let $\Omega=\T_L\times(0,h)$ with no-slip walls and $f=gx_2$, $g>0$. Let $\bu_0$ satisfy \ref{A3}. Let $\rho_0\in C^1(\overline\Omega)$ satisfy \ref{A2} and be constant on each wall. Assume that for some $a>0$ and $b\ge0$,
\[
\partial_2\rho_0\le-a,\qquad|\partial_1\rho_0|\le b\qquad\text{in }\overline\Omega .
\]
Let $\Lambda_\infty\coloneqq\int_0^\infty\norm{\nabla\bu(t)}_{L^\infty}dt$ and $\eta\coloneqq e^{\Lambda_\infty}-1$. If $\eta<\tfrac{a}{a+b}$, then:
\begin{enumerate}[(1)]
\item $\partial_2\rho(t,\bx)\le-a(1-\eta)+b\eta<0$ for all $t,\bx$, and $\sup_t\norm{\nabla\rho(t)}_{L^\infty}\le e^{\Lambda_\infty}\norm{\nabla\rho_0}_{L^\infty}$;
\item the flow maps converge, $X_t\to X_\infty$, and $\rho_0\circ X_\infty^{-1}=\rs$: the limiting flow \emph{sorts} the initial density;
\item Estimates \eqref{rate} and \eqref{conditional-stokes} hold with $s=1$. In particular the density error, the $H^2$ velocity and the pressure-gradient error decay at least as fast as $(1+t)^{-1/4}$.
\end{enumerate}
\end{corollary}

\begin{theorem}\label{thm:wall-rate}
Let $\Omega=\T_L\times(0,h)$ with no-slip walls and $f=gx_2$, $g>0$. Let $\rho_s=\rho_s(x_2)\in H^4(0,h)$ be a positive profile satisfying $\partial_2\rho_s\le-a_s<0$. Assume \ref{A2}--\ref{A3} and $\rho_0-\rho_s\in H^4(\Omega)$, with $\rho_0$ constant on each wall. Put
\[
\varepsilon=\norm{\rho_0-\rho_s}_{H^4}+\norm{\bu_0}_{H^2}.
\]
For each fixed $K\ge1$, there exist $\varepsilon_*>0$ and $C>0$ such that, if $0<\varepsilon\le\varepsilon_*$ and
\begin{equation}\label{wall-reg}
\sup_{t\ge0}\norm{\rho(t)-\rho_s}_{H^4}\le K\varepsilon,
\end{equation}
then, for all $t\ge0$,
\begin{equation}\label{wall-rate}
 E(t)-E^\natural+\norm{\rho(t)-\rs}^2+\norm{\bu(t)}_{H^1}^2+\norm{\partial_t\bu(t)}^2
 \le C\varepsilon^2(1+t)^{-2}.
\end{equation}
In addition,
\begin{equation}\label{wall-strong}
 \norm{\bu(t)}_{H^2}+\norm{\nabla\Pi(t)+\rs\nabla f}
 \le C\varepsilon(1+t)^{-1}.
\end{equation}
The flow maps converge to a measure-preserving $C^1$ diffeomorphism $X_\infty$ with $\rho_0=\rs\circ X_\infty$. The accumulated strain and flow satisfy
\begin{align}
 \sup_{\bx\in\overline\Omega}\dist\big(X_t(\bx),X_\infty(\bx)\big)
 &\le C\varepsilon(1+t)^{-1/2}\sqrt{\log(2+t)}, \label{wall-flow}\\
 \norm{X_{t,2}-X_{\infty,2}}_{L^\infty}
 &\le C\varepsilon(1+t)^{-3/4},\label{wall-flow-vertical}\\
 \norm{\nabla X_t-\nabla X_\infty}_{L^\infty}
 +\int_t^\infty\norm{\nabla\bu(r)}_{L^\infty}\,dr
 &\le C_\gamma\varepsilon(1+t)^{-\gamma}
 \qquad (0<\gamma<1/8).\label{wall-flow-gradient}
\end{align}
In addition,
\begin{equation}\label{wall-flow-vertical-grad}
 \norm{\rho(t)-\rs}_{C^1}
 +\norm{\nabla X_{t,2}-\nabla X_{\infty,2}}_{L^\infty}
 \le C\varepsilon(1+t)^{-1/2}.
\end{equation}
The constants may depend on $K,\rho_s,\alpha_1,\alpha_2,\nu,g,L,h$.
\end{theorem}
By Proposition~\ref{thm:wall-obstruction}(2), the uniform bound \eqref{wall-reg} in this theorem forces $\partial_2\rho_0$ to be constant on each wall. The perturbation $\varepsilon x_2(h-x_2)\cos(2\pi x_1/L)$, for instance, has zero trace and a nonconstant normal derivative on each wall, so it cannot satisfy \eqref{wall-reg} in this small-data regime.

\begin{proposition}\label{thm:wall-obstruction}
In the same channel let $\rho_0\in C^1(\overline\Omega)$ satisfy \ref{A2}, let $\bu_0$ satisfy \ref{A3}, and write
\[
 \Lambda(t)=\int_0^t\norm{\nabla\bu(\tau)}_{L^\infty}\,d\tau .
\]
\begin{enumerate}[(1)]
\item Let $q=q(x_2)$ be a bounded layered profile with a one-sided essential limit at a wall $x_2=b$, $b\in\{0,h\}$. Suppose this limit differs from $\rho_0(x_1,b)$ on a set of positive horizontal measure. There exists $c>0$, depending on $\rho_0$, $q$ and the channel, such that for every $t\ge0$,
\begin{equation}\label{wall-obstruction}
 \norm{\rho(t)-q}^2\ge c e^{-\Lambda(t)},\qquad
 \norm{\rho(t)-q}^2\ge\frac{c}{1+\norm{\nabla\rho(t)}_{L^\infty}}.
\end{equation}
If the wall trace differs from the corresponding endpoint of $\rs$ and $\norm{\rho(t)-\rs}\le A(1+t)^{-\gamma}$ for $A,\gamma>0$, then
\[
 \Lambda(t)\ge 2\gamma\log(1+t)-C,\qquad
 1+\norm{\nabla\rho(t)}_{L^\infty}\ge c A^{-2}(1+t)^{2\gamma}.
\]
\item At any wall $b$ where $\rho_0$ is constant, the trace of $\partial_2\rho(t)$ equals that of $\partial_2\rho_0$. Write $v=\rho-\bar\rho$ and $\tau_b=\norm{\partial_2v(0,\cdot,b)}_{L^2(\T_L)}$. If $\tau_b>0$ and $\norm{v(t)}\to0$, then for every $s>3/2$ the $H^s$ norm of $\rho$ cannot stay bounded along a sequence of times tending to infinity. More precisely, whenever $v(t)\in H^4$,
\begin{equation}\label{wall-normal-obstruction}
 \tau_b\le C\norm{v(t)}^{5/8}\norm{v(t)}_{H^4}^{3/8}.
\end{equation}
Consequently, $\norm{v(t)}\le A(1+t)^{-r}$ with $A,r>0$ implies
$\norm{v(t)}_{H^4}\ge c\tau_b^{8/3}A^{-5/3}(1+t)^{5r/3}$ at those times.
\end{enumerate}
\end{proposition}

\subsubsection*{D. Invariant flows and Rayleigh--Taylor perturbations}
\begin{theorem}\label{thm:nonmono}
Consider solutions of \eqref{main} with data satisfying \ref{A2}--\ref{A3}.
\begin{enumerate}[(1)]
\item (Channel shear flows.) In the no-slip channel $\Omega=\T_L\times(0,h)$ with $f=gx_2$ and $g>0$, let $\rho_0=r(x_2)$ satisfy $\alpha_1\le r\le\alpha_2$ and take $\bu_0=(U_0(x_2),0)$ with $U_0\in H^2\cap H^1_0(0,h)$. Then $\rho(t)=\rho_0$ and $\bu(t)=(U(t,x_2),0)$, where $r\partial_tU=\nu\partial_2^2U$ with $U(t,0)=U(t,h)=0$. Moreover,
\[
\norm{\bu(t)}_{H^2}+\norm{\partial_t\bu(t)}\le C e^{-\lambda t}\norm{\bu_0}_{H^2},\qquad
\lambda=\frac{\nu\pi^2}{\alpha_2h^2}.
\]
\item (Invariant swirl flows in a disk.) Let $\Omega=B_R$, $f=\tfrac g2|\bx|^2$ with $g>0$, $\rho_0=R_\rho(r)$ and $\bu_0=U_0(r)\be_\theta$ satisfying \ref{A3}, where $r=|\bx|$. Then $\rho(t)\equiv\rho_0$ and $\bu=U(t,r)\be_\theta$, with exponential decay. The scalar equation on $0<r<R$ is $R_\rho(r)\partial_tU=\nu(\partial_r^2U+\tfrac1r\partial_rU-\tfrac1{r^2}U)$, with $U(t,R)=0$ and regularity at the origin inherited from the vector field.
\item (Sharpness.) In (1)--(2) the total dissipation equals $\KE_0$, which can be arbitrarily large, and no available potential energy is released. If $\alpha_1<\alpha_2$, choose admissible profiles $r(x_2)=\beta+\gamma x_2$ in (1) or $R_\rho(r)=\beta+\gamma r^2$ in (2), with $\gamma>0$. Then:
\begin{enumerate}[(a)]
\item the limiting density is heavier on top (resp.\ at the periphery) and is linearly Rayleigh--Taylor unstable;
\item $\nu\int_0^\infty\norm{\nabla\bu}^2=\KE_0<\KE_0+\APE_0$;
\item in (1), hypotheses (i)--(ii) of Theorem~\ref{thm:rate} hold but its density does not sort;
\item the limit is not a function of the distribution of $\rho_0$, because the equimeasurable datum $(\mathbf 0,\rs)$ is steady.
\end{enumerate}
\item (Constraints.) Every weak limit point $\rho_\infty$ belongs to $\Hy\cap\Kr$ and satisfies $\int\rho_\infty f=E(0)-\nu\int_0^\infty\norm{\nabla\bu}^2\le E(0)$. Under \ref{A4}, $\int\rho_\infty f\ge E^\natural$.
\item (Eventual stable layering determines the limit.) Let $\Omega=\T_L\times(0,h)$ with no-slip walls and $f=gx_2$, $g>0$. Suppose that for some $T_0\ge0$, the horizontal mean $\bar\rho(t)$ is nonincreasing in $x_2$ for $t\ge T_0$, and $\sup_{t\ge T_0}\norm{\rho(t)}_{H^s}<\infty$ for some $s>0$. Then $\rho(t)\to\rs=\rho_0^*(Lx_2)$ in $L^2$. Under these conditions the selected profile is independent of $\bu_0$.
\item (Relaxation rate under an eventual Lipschitz bound.) In the setting of (5), let $\rho_0\in C^1(\overline\Omega)$ be constant on each wall. Suppose that $\partial_2\rho(T_0)\le-a$ and $|\partial_1\rho(T_0)|\le b$ for some $a>0$ and $b\ge0$, and that $e^{\int_{T_0}^\infty\norm{\nabla\bu}_{L^\infty}}-1<\tfrac a{a+b}$. Then $\rho(t)\to\rs$, and \eqref{rate} and \eqref{conditional-stokes} hold with $s=1$ for $t\ge T_0$.
\end{enumerate}
\end{theorem}

\begin{remark}\label{rem:nonmono}
(1) In (1)--(2) translation (resp.\ rotation) invariance is preserved by uniqueness. Hence no vertical (resp.\ radial) motion is generated, and the limit is selected by the symmetry class of the data rather than by energy minimization.
(2) Part (3) gives linearly unstable equilibria within infinite-dimensional invariant classes of shear and swirl data.
(3) At the linear level, the limit is independent of $\bu_0$ (Theorem~\ref{thm:lin}).
\end{remark}

\begin{theorem}\label{thm:RT}
Let $\Omega$ be a smooth bounded domain in $\R^2$ with no-slip boundary conditions. Take $f$ and $\rho_s$ smooth up to the boundary, as in the setting of \cite[Theorems~1.2--1.3]{Li2025}. Let $\rho_s$ be hydrostatic with density uniformly separated from $\alpha_1$ and $\alpha_2$, and with $\nabla\rho_s=h\nabla f$, where $h(\bx_0)>0$ and $\nabla f(\bx_0)\neq0$ for some $\bx_0\in\Omega$. There exist $\bv_0\in H^2\cap H^1_0\cap H$, $\sigma_0\in H^1\cap L^\infty$, $\delta_*>0$ and $\kappa>0$, with $\kappa$ \emph{independent of $\delta$}, such that the following holds for every $\delta\in(0,\delta_*]$. Consider the solution with data $\bu_0^\delta=\delta\bv_0$ and $\rho_0^\delta=\rho_s+\delta\sigma_0$.
\begin{enumerate}[(1)]
\item (Velocity decay.) $\norm{\bu(t)}_{W^{1,r}}+\norm{\partial_t\bu(t)}\to0$ for every $1\le r<\infty$. The density converges in $H^{-1}$ to the set $\mathcal S^\delta$ of Theorem~\ref{thm:hydro}, and every limit $\rho_\infty$ is a hydrostatic equilibrium with $\nabla\Pi_\infty=-\rho_\infty\nabla f$.
\item (Definite energy release.) Let $T^\delta$ be the escape time at which $\norm{\bu(T^\delta)}_{L^1}\ge\varepsilon$ for the fixed $\varepsilon>0$ of \cite[Theorem~1.3]{Li2025}. Set $m=\varepsilon|\Omega|^{-1/2}$ and $t_*^\delta=T^\delta+m/C_1$, where $C_1$ uniformly bounds $\norm{\bu_t}$ for the perturbation family. Then $\nu\int_0^{t_*^\delta}\norm{\nabla\bu}^2\,dt\ge4\kappa$ and $E(t)\le\int\rho_0^\delta f-3\kappa$ for $t\ge t_*^\delta$. For any fixed $\eta\in(0,1)$ one may take $4\kappa=(2-\eta^3)\nu m^3/(3C_P^2C_1)$ after reducing the perturbation range.
\item (No return.) $\norm{\rho(t)-\rho_s}\ge\kappa/\norm f$ for $t\ge t_*^\delta$, and $\norm{\rho_\infty-\rho_s}\ge\kappa/\norm f$ for every limit point. These separations also hold in $H^{-1}$ with a positive lower bound independent of the perturbation size.
\item (Overturning fraction.) Assume \ref{A4}. Then $\APE(\rho_s)\coloneqq\int(\rho_s-\rho_s^\natural)f>0$, and the fraction of available potential energy released,
\[
\theta^\delta\coloneqq\frac{\int\rho_0^\delta f-E_\infty}{\APE_0^\delta}\in[0,1],
\]
where $\APE_0^\delta\coloneqq\int(\rho_0^\delta-(\rho_0^\delta)^\natural)f$, satisfies $\theta^\delta\ge\theta_*\coloneqq\kappa/\APE(\rho_s)$ for small $\delta$. Moreover $\theta^\delta=1$ if and only if $\rho(t)\to(\rho_0^\delta)^\natural$ in $L^2$ (complete overturning). If $(\rho_0^\delta)^\natural=\phi_\delta(f)$ with $\phi_\delta$ $\delta_1$-Lipschitz, then
\[
\limsup_{t\to\infty}\norm{\rho(t)-(\rho_0^\delta)^\natural}^2\le2\delta_1(1-\theta^\delta)\APE_0^\delta .
\]
\item (Gravity: layering and descent of the center of mass.) Let $f=gx_2$ with $g>0$, and let the horizontal sections of $\Omega$ be intervals. Then every limit point is horizontally layered, $\rho_\infty=\sigma(x_2)$. Its center of mass $z(\sigma)\coloneqq\int\sigma x_2/\int\sigma$ satisfies
\[
z(\rs_\delta)\le z(\sigma)\le z(\rho_0^\delta)-\frac{3\kappa}{g\int\rho_0^\delta},\qquad\rs_\delta\coloneqq(\rho_0^\delta)^\natural .
\]
Along a subsequence with weak limit $\sigma$, convergence is strong in $L^2$ if and only if $\sigma$ is a layered rearrangement of $\rho_0^\delta$ (no mixing). Complete overturning means $\sigma=\rs_\delta$, and $\norm{\rs_\delta-\rho_s^\natural}\le\delta\norm{\sigma_0}$, where $\rho_s^\natural$ is the decreasing vertical rearrangement of $\rho_s$.
\end{enumerate}
\end{theorem}

\begin{remark}\label{rem:RT}
The perturbation family of Theorem~\ref{thm:RT} releases at least the fraction $\theta_*>0$ of the available potential energy under \ref{A4}. In the gravitational setting of part (5), the center of mass descends by an amount independent of the perturbation size. Every limiting density stays separated from $\rho_s$, while the velocity vanishes.

In the gravitational setting of part (5), a weak density limit falls into one of three cases.
\begin{enumerate}[(a)]
\item complete overturning to the sorted state, $\theta=1$;
\item partial overturning without mixing, to a different layered rearrangement with a lower center of mass;
\item partial overturning with mixing, when $\norm{\rho_\infty}^2<\norm{\rho_0^\delta}^2$.
\end{enumerate}
For a distinct inviscid mixing mechanism, see \cite{Gebhard2021}. For two-valued density data, Theorem~\ref{thm:mix-frontier} describes the weak closure and gives the minimum residual potential energy for each loss of density variance. Its equality cases are admissible weak-closure profiles. The dynamical realization of a mixed state remains an open question. By Theorem~\ref{thm:nonmono}(5), scenario (a) follows if the horizontal mean is eventually nonincreasing in the periodic channel and the density is uniformly bounded in $H^s$ for some $s>0$.
\end{remark}

\subsubsection*{E. Linear theory}
\begin{theorem}\label{thm:lin}
Assume \ref{A1}, or its periodic-channel version. Let $\rho_s\in W^{1,\infty}$ be hydrostatic with $\rho_s\ge\alpha_1$ and $\nabla\rho_s=-\delta\nabla f$, where $0<\delta_0\le\delta\le\delta_1<\infty$ a.e. Assume that $\nabla f\neq0$ on an open set. Consider the linearization with no-slip walls,
\begin{equation}\label{lin-gen}
\rho_s\partial_t\bu=\nu\Delta\bu-\nabla P-\varrho\nabla f,\qquad \partial_t\varrho=\delta\,\bu\cdot\nabla f,\qquad\Div\bu=0,\qquad\bu|_{\partial\Omega}=0 .
\end{equation}
Let $\Hy_L\coloneqq\{\sigma\in L^2:\sigma\nabla f\in G\}$, and let $\Pi_{\Hy}$ be the orthogonal projection onto $\Hy_L$ in $L^2(\delta^{-1}d\bx)$. The linear estimate constants may depend on the background $\rho_s$ and the bounds $\delta_0,\delta_1$. For every $(\bu_0,\varrho_0)\in H\times L^2$:
\begin{enumerate}[(1)]
\item The hydrostatic component $\Pi_{\Hy}\varrho(t)=\Pi_{\Hy}\varrho_0$ is conserved, and
\[
 \norm{\varrho(t)-\Pi_{\Hy}\varrho_0}
 +\norm{\bu(t)}_{H^2}+\norm{\partial_t\bu(t)}
 +\norm{\nabla P(t)+(\Pi_{\Hy}\varrho_0)\nabla f}\longrightarrow0.
\]
In particular, the limiting density is independent of $\bu_0$.
\item On $\{\Pi_{\Hy}\varrho=0\}$ the solution operator has energy norm $1$ for all $t$. Nevertheless, for $t>0$,
\begin{equation}\label{linear-uniform-velocity}
 t^{1/2}\norm{\bu(t)}_{H^1}
 +t\norm{\partial_t\varrho(t)}
 \le C\norm{(\bu_0,\varrho_0-\Pi_{\Hy}\varrho_0)}_{L^2\times L^2}.
\end{equation}
Both time exponents are optimal among uniform estimates with this input norm. More generally, for $1\le r\le2$ and $t\ge1$,
\begin{equation}\label{linear-velocity-interpolated}
 \norm{\bu(t)}_{H^r}\le C_r t^{-(2-r)/2}
 \norm{(\bu_0,\varrho_0-\Pi_{\Hy}\varrho_0)}_{L^2\times L^2}.
\end{equation}
In the gravitational no-slip channel these exponents are optimal throughout $1\le r\le2$. In particular no uniform $H^2$ velocity bound with a factor tending to zero can hold for all $L^2$ initial states.
\item For displacement data $\varrho_0=-\bv\cdot\nabla\rho_s$ and $\bu_0=\mathcal Q[\rho_s^{-1}(\nu\Delta\bv-\sigma\nabla f)]$, with $\bv\in H^2\cap H^1_0\cap H$ and $\sigma\in L^2$ ($\mathcal Q$ is the $\rho_s$-weighted Leray projector), we have $\norm{\bu(t)}+\norm{\varrho(t)}\le C(1+t)^{-1}(\norm\bv_{H^2}+\norm\sigma)$. For $t\ge1$ and $1\le r\le2$,
\begin{equation}\label{displacement-velocity}
 \norm{\bu(t)}_{H^r}\le C_r t^{-(4-r)/2}(\norm\bv_{H^2}+\norm\sigma).
\end{equation}
This velocity exponent is optimal in the gravitational no-slip channel for the indicated norm on $(\bv,\sigma)$.
\item If $f=gx_2$ and the horizontal sections are intervals, then $\Pi_{\Hy}\varrho_0$ is the average of $\varrho_0$ over horizontal sections.
\item For smoother displacement data, $x_0=(A_1(I-A_1)^{-1})^2z$ with $z\in D(A_1^2)$ (Section~\ref{sec:linear}), we have $\norm{\bu(t)}_{W^{2,p}}+\norm{\varrho(t)}_{L^p}\le C_p(1+t)^{-3/2}\norm{z}_{D(A_1^2)}$ for $1<p<\infty$. Hence $\int_0^\infty\norm{\nabla\bu}_{L^\infty}<\infty$, and the linearized Lagrangian displacement converges.
\item (Slow modes and optimality of the exponent.) Write $\langle a,b\rangle_{\delta^{-1}}\coloneqq\int\delta^{-1}ab$ and $\mathcal E_0\coloneqq\frac12\norm{\sqrt{\rho_s}\,\bu_0}^2+\frac12\norm{\delta^{-1/2}\varrho_0}^2$. For all data and every $\varphi\in L^2$,
\begin{equation}\label{slowmode}
\abs[\big]{\langle\varrho(t)-\varrho_0,\varphi\rangle_{\delta^{-1}}}\le\Big(\frac{t\,\mathcal E_0}{\nu}\Big)^{1/2}\norm{\varphi\nabla f}_{V'}\qquad(t\ge0),
\end{equation}
where $V'$ is the dual of $V=H^1_0\cap H$ with the gradient norm. If $\Omega=\T_L\times(0,h)$ and $f=gx_2$, then for every $s\ge0$ there are $c>0$ and data $(\mathbf 0,\varrho_{0,n})$, $n\in\mathbb N$, with $\Pi_{\Hy}\varrho_{0,n}=0$ and $\norm{\varrho_{0,n}}_{H^s}=1$, such that
\[
\norm{\varrho_n(t)}\ge c\,n^{-s}\qquad\text{for }0\le t\le c\,n^4 .
\]
Consequently, no estimate $\norm{\varrho(t)-\Pi_{\Hy}\varrho_0}\le r(t)\norm{\varrho_0}_{H^s}$ with $r(t)=o(t^{-s/4})$ holds.
\end{enumerate}
\end{theorem}

\begin{remark}\label{rem:lin}
If $\rho_s=\varphi(f)$, $\nabla f\ne0$ a.e., and almost every level of $f$ is regular and connected, then $\Hy_L$ consists of functions of $f$ and $\Pi_{\Hy}=\mathbb E[\,\cdot\,|f]$. By Lemma~\ref{lem:firstvar} this is the first variation of sorting when the potential satisfies the level-set measure bound. In the channel, the slow modes in (6) persist through times of order $n^4$ and the squared density lower bound has the exponent $s/2$ in \eqref{rate}.
\end{remark}

\begin{theorem}\label{thm:planar-spectrum}
Use the assumptions of Theorem~\ref{thm:lin}, and let $A_1$ be the linearized generator after removing the hydrostatic kernel. There is a punctured disk about zero in which the spectrum of $A_1$ consists of real, negative, semisimple eigenvalues $-a_n$ tending to zero. With multiplicities and in decreasing order of $a_n$,
\[
 c n^{-1}\le a_n\le C n^{-1}.
\]
If $\delta$ is continuous on $\overline\Omega$, then
\begin{equation}\label{planar-weyl}
 \lim_{n\to\infty} n a_n
 =\frac{1}{8\pi\nu}\int_\Omega\delta|\nabla f|^2\,\dx .
\end{equation}
Let $\mu_n$ be the ordered positive eigenvalues of the compact buoyancy form relative to the viscous form. In the comparison below, $a_n$ is matched to $\mu_n$ by the pencil, which may shift the initial index by a finite amount. With $M$ the velocity mass operator relative to the viscous form, for all sufficiently large $n$,
\begin{equation}\label{pencil-root-bounds}
 c\mu_n^3\le a_n-\mu_n,\qquad
 a_n\le\frac{2\mu_n}{1+\sqrt{1-4\mu_n\norm M}}.
\end{equation}
In a periodic channel, the planar count includes all horizontal frequencies.
\end{theorem}

For a horizontally periodic viscous layer with rigid walls, Gurski and Pego \cite{GurskiPego2002} established slowly damped nonoscillatory modes and gave asymptotic approximations. Gurski, Koll\'ar and Pego \cite{GurskiKollarPego2004} proved existence of arbitrarily slowly damped modes in bounded domains. See Koll\'ar \cite{Kollar2011} for an abstract pencil argument. Theorem~\ref{thm:planar-spectrum} gives the explicit planar Weyl coefficient. Theorem~\ref{thm:channel-spectrum} gives the joint distribution of the horizontal and vertical indices for variable stratification and determines the fixed-branch corrections at interior and wall maxima. For constant $\delta$ under gravity this coefficient is $\delta g^2|\Omega|/(8\pi\nu)$. Counting all frequencies and following a single nonzero horizontal frequency give different exponents. \autoref{fig:spectral-count} shows why these counts differ: the planar asymptotic counts all lattice points in the local frequency region, whereas a fixed horizontal mode samples just one vertical slice.

\begin{figure}[tbp]
\centering
\begin{tikzpicture}[x=3.6cm,y=3.6cm,>=Stealth,font=\small]
\filldraw[fill=blue!12,draw=blue!70!black,thick] (-.5,0) circle[radius=.5];
\filldraw[fill=blue!12,draw=blue!70!black,thick] (.5,0) circle[radius=.5];
\draw[->] (-1.15,0) -- (1.24,0) node[right] {$\zeta_\parallel$};
\draw[->] (0,-.60) -- (0,.66) node[above] {$\zeta_\perp$};
\draw[dashed,teal!65!black,thick] (.67,-.60) -- (.67,.66)
    node[above right,font=\footnotesize] {one horizontal frequency};
\foreach \x/\lab in {-1/{-1},1/{1}} {
  \draw (\x,.022) -- (\x,-.022) node[below=2pt] {$\lab$};
}
\node[below left=2pt] at (0,0) {$0$};
\end{tikzpicture}
\caption{Local phase region in the planar spectral count. At a point where
$\nabla f\ne0$, align $\zeta_\parallel$ with $\nabla^\perp f$ and multiply the
frequency by $\sqrt{\tau\nu/(\delta|\nabla f|^2)}$. The local form inequality
becomes $|\zeta|^2<|\zeta_\parallel|$, the union of the two open disks shown.
For vertical gravity, a fixed horizontal frequency samples a vertical slice
at a rescaled abscissa proportional to $\sqrt{\tau}$. In the regime of
\eqref{channel-joint-count}, the abscissa is $\sqrt{q/\delta(z)}$ and the
slice half-width, in units of $|k|$, is $(\sqrt{\delta(z)/q}-1)_+^{1/2}$.
The area count and the fixed-frequency count yield the distinct exponents
in Theorems~\ref{thm:planar-spectrum} and~\ref{thm:channel-spectrum}.}
\label{fig:spectral-count}
\end{figure}

\begin{theorem}\label{thm:channel-spectrum}
Under the stable-stratification assumptions of Theorem~\ref{thm:lin}, consider \eqref{lin-gen} on $\T_L\times(0,h)$ with $f=gx_2$ and $g>0$. Let $k\in(2\pi/L)\mathbb Z\setminus\{0\}$ and $n\ge1$. Denote the eigenvalues of the clamped buoyancy form relative to the viscous form by $\mu_{k,n}$ in decreasing order. There are constants independent of $k,n$ such that
\[
 c\frac{k^2}{(k^2+n^2)^2}\le\mu_{k,n}\le C\frac{k^2}{(k^2+n^2)^2}.
\]
For all sufficiently large $k^2+n^2$ the matching real eigenvalue $\lambda_{k,n}=-a_{k,n}$ satisfies
\begin{equation}\label{channel-uniform-correction}
 c\frac{k^6}{(k^2+n^2)^6}\le a_{k,n}-\mu_{k,n}\le C\frac{k^4}{(k^2+n^2)^5}.
\end{equation}
In particular, $a_{k,n}-\mu_{k,n}\asymp |k|^{-6}$ uniformly for $n\le C_*|k|$ with fixed $C_*>0$. For fixed $k$ this gives $a_{k,n}\asymp_k n^{-4}$ and an $O_k(n^{-10})$ inertial correction. If $\delta$ is continuous on $[0,h]$, then
\begin{equation}\label{channel-weyl}
 \lim_{n\to\infty}n^4 a_{k,n}
 =\frac{g^2k^2}{\nu\pi^4}
    \left(\int_0^h\delta(x_2)^{1/4}\,dx_2\right)^4 .
\end{equation}

For continuous $\delta$, write $D=\max_{[0,h]}\delta$. For $0<q<D$,
\begin{equation}\label{channel-joint-count}
 \lim_{|k|\to\infty}\frac{1}{|k|}
  \#\left\{n:\frac{\nu}{g^2}k^2a_{k,n}>q\right\}
 =\mathcal N(q),\qquad
 \mathcal N(q)=\frac1\pi\int_0^h\left(\sqrt{\frac{\delta(z)}q}-1\right)_+^{1/2}\,dz .
\end{equation}
In particular, if $n/|k|\to\gamma\in(0,\infty)$, then $(\nu/g^2)k^2a_{k,n}\to\mathcal N^{-1}(\gamma)$.

If $n$ is fixed, with only the bounded positive stratification assumed above, then
\begin{equation}\label{large-horizontal-branch}
 \lim_{|k|\to\infty}k^2 a_{k,n}=\frac{g^2}{\nu}D,
 \qquad D=\operatorname*{ess\,sup}_{(0,h)}\delta .
\end{equation}
Suppose in addition that $\delta\in C^2([0,h])$ and that the set $Z$ of its global maxima consists of finitely many interior points $z_1,\ldots,z_M$ with $\delta''(z_i)<0$. Let $e_n$ be the $n$-th term, in increasing order and counted with multiplicity, of the multiset
\[
 \big\{(2j-1)\sqrt{-D\delta''(z_i)}: 1\le i\le M,\ j\ge1\big\}.
\]
For each fixed $n$, as $|k|\to\infty$,
\begin{equation}\label{large-horizontal-second}
 k^2 a_{k,n}=\frac{g^2}{\nu}
      \left(D-\frac{e_n}{|k|}+o(|k|^{-1})\right).
\end{equation}
Every corresponding stream function $\psi_{k,n}$ of unit $L^2(0,h)$ norm satisfies
\[
 \int_0^h\operatorname{dist}(z,Z)^2|\psi_{k,n}(z)|^2\,dz
 \le C_n |k|^{-1}.
\]
Alternatively, suppose $\delta\in C^1([0,h])$ has a unique global maximum at a wall $z_*$, and that
$D-\delta(z)=b|z-z_*|+o(|z-z_*|)$ with $b>0$ as $z\to z_*$ from the interior. Let $\zeta_n$ denote the $n$-th positive number, in increasing order, such that $\operatorname{Ai}(-\zeta_n)=0$. For each fixed $n$,
\begin{equation}\label{large-horizontal-airy}
 \frac{\nu}{g^2}k^2a_{k,n}
   =D-(2D)^{1/3}b^{2/3}\zeta_n\,|k|^{-2/3}+o(|k|^{-2/3}).
\end{equation}
Every corresponding stream function of unit $L^2(0,h)$ norm satisfies
\begin{equation}\label{wall-peak-localization}
 \int_0^h |z-z_*|\,|\psi_{k,n}(z)|^2\,dz\le C_n |k|^{-2/3}.
\end{equation}
In either case with a unique maximum $z_*$, for each fixed $n$ the inertial correction satisfies
\begin{equation}\label{channel-inertial-limit}
 \lim_{|k|\to\infty}|k|^6(a_{k,n}-\mu_{k,n})
 =\frac{\rho_s(z_*)g^4 D^2}{\nu^3}.
\end{equation}
\end{theorem}

For constant $\delta$, the limit in \eqref{channel-joint-count} gives $\mathcal N^{-1}(\gamma)=\delta/[1+(\pi\gamma/h)^2]^2$. Under the nondegenerate-maxima hypothesis of the theorem, one also has
\[
 \mathcal N(D-\eta)\sim\eta\sum_{i=1}^M\frac1{2\sqrt{-D\delta''(z_i)}}\qquad(\eta\downarrow0).
\]
Indeed, only neighborhoods of the maxima contribute for small $\eta$. In the coordinate $z=z_i+\sqrt\eta\,y$, the support stays in a fixed bounded interval by nondegeneracy. After division by $\sqrt\eta$, the local integrand converges uniformly there to $[(1-b_i y^2)_+/(2D)]^{1/2}$, where $b_i=-\delta''(z_i)/2$. The integral of $(1-b_i y^2)_+^{1/2}$ is the area of a unit semicircle divided by $\sqrt{b_i}$, giving the displayed coefficient. For a single maximum the corrections in \eqref{large-horizontal-second} grow with the odd integers $2n-1$. \autoref{fig:channel-oscillators} shows how the ordering changes with two maxima of different curvatures. Both maxima have the same leading value $D$, and the curvature at each one determines the spacing of its corrections.

\begin{figure}[!tb]
\centering
\begin{tikzpicture}[>=Stealth,font=\small]
\begin{scope}
  \draw[->] (0,0) -- (4.4,0) node[right] {$z$};
  \draw[->] (0,0) -- (0,2.2) node[above] {$D-\delta(z)$};
  \draw[black!75,thick] (.10,1.3) .. controls (.35,.94) and (.78,0) .. (1.23,0)
     .. controls (1.68,0) and (2.15,1.65) .. (2.48,1.65)
     .. controls (2.81,1.65) and (2.99,0) .. (3.21,0)
     .. controls (3.43,0) and (3.69,1.60) .. (4.10,1.88);
  \fill[blue!70!black] (1.23,0) circle (1.5pt) node[below=3pt] {$z_1$};
  \fill[teal!70!black] (3.21,0) circle (1.5pt) node[below=3pt] {$z_2$};
\end{scope}
\begin{scope}[xshift=6cm,x=1cm,y=.26cm]
  \draw[->] (0,0) -- (0,8.8) node[above] {correction in units of $\sqrt{2Db_1}$};
  \foreach \level in {1,2,3,5,6,7} {
    \draw (-.07,\level) -- (0,\level) node[left=2pt] {$\level$};
    \draw[black!18,densely dotted] (0,\level) -- (2.7,\level);
  }
  \foreach \level in {1,3,5,7}
    \draw[blue!70!black,thick] (.35,\level) -- (1.22,\level);
  \foreach \level in {2,6}
    \draw[teal!70!black,thick] (1.62,\level) -- (2.49,\level);
  \node[below=3pt] at (.79,0) {$z_1$};
  \node[below=3pt] at (2.05,0) {$z_2$};
\end{scope}
\end{tikzpicture}
\caption{Ordering of the first six corrections in \eqref{large-horizontal-second}
for two equal-height maxima with curvature parameters
$b_i=-\delta''(z_i)/2$ and $b_2=4b_1$. The left panel sketches the
potential deficit. On the right, the levels at $z_1$ start at $1$ with spacing
$2$, and those at $z_2$ start at $2$ with spacing $4$. Reading the two columns
together from bottom to top gives the ordered corrections.}
\label{fig:channel-oscillators}
\end{figure}

For a wall maximum as in \eqref{large-horizontal-airy}, the joint count has the edge asymptotic
\begin{equation}\label{wall-peak-count}
 \mathcal N(D-\eta)\sim \frac{2}{3\pi b\sqrt{2D}}\eta^{3/2}
 \qquad(\eta\downarrow0).
\end{equation}
To see this, set the inward distance equal to $\eta y$. The rescaled support is bounded and the integrand divided by $\sqrt\eta$ converges uniformly there to $\big[(1-by)_+/(2D)\big]^{1/2}$. Integration over $y\ge0$ gives the coefficient. \autoref{fig:wall-peak} compares this linear growth of the potential deficit with the quadratic growth at an interior maximum. The two shapes yield different localization scales, as quantified by the moment bound following \eqref{large-horizontal-second} and by \eqref{wall-peak-localization}.

\begin{figure}[!tb]
\centering
\begin{tikzpicture}[>=Stealth,font=\small]
\begin{scope}
  \fill[blue!10] (1.05,0) rectangle (2.95,1.25);
  \draw[->] (0,0)--(4.4,0) node[right] {$z$};
  \draw[->] (0,0)--(0,2.15) node[above] {$V$};
  \draw[blue!65!black,thick,domain=.53:3.47,samples=32] plot (\x,{.84*(\x-2)^2});
  \draw[blue!65!black,<->] (1.13,1.0)--(2.87,1.0) node[midway,above] {$O(|k|^{-1/2})$};
  \node[below=2pt] at (2,0) {$z_*$};
  \node at (2,2.05) {Interior maximum};
\end{scope}
\begin{scope}[xshift=6cm]
  \fill[teal!11] (.06,0) rectangle (1.43,1.25);
  \draw[->] (0,0)--(4.4,0) node[right] {$z$};
  \draw[->] (0,0)--(0,2.15) node[above] {$V$};
  \draw[teal!60!black,thick] (0,0)--(3.55,1.8);
  \draw[teal!60!black,<->] (.16,1.0)--(1.42,1.0) node[midway,above,xshift=4pt] {$O(|k|^{-2/3})$};
  \node[below=2pt] at (0,0) {$z_*$};
  \node at (2,2.05) {Wall maximum};
\end{scope}
\end{tikzpicture}
\caption{Local shape of the potential deficit $V=D-\delta$ near an interior maximum and a wall maximum. The shaded bands represent the localization scales given by the corresponding moment bounds. The curves represent the quadratic and linear approximations to $V$ at the peaks.}
\label{fig:wall-peak}
\end{figure}

\begin{theorem}\label{thm:channel-traces}
Use the setting of Theorem~\ref{thm:channel-spectrum}. Assume $\bu_0\in H^1\cap H$ and
\[
 w_0=\frac{\varrho_0-\bar\varrho_0}{g\delta}\in H^4(\Omega),\qquad
 B_0=\norm{\bu_0}_{H^1}+\norm{w_0}_{H^4}.
\]
Here $\bar\varrho_0$ is the horizontal average and the wall traces below refer to $w_0$.
\begin{enumerate}[(i)]
\item If $w_0$ and its normal derivative vanish on both walls, then for $t\ge1$
\begin{equation}\label{linear-compatible}
 \norm{\varrho(t)-\bar\varrho_0}+\norm{\bu(t)}_{H^2}
 +t^{1/2}\norm{\bu(t)}_{H^1}\le C B_0 t^{-1}.
\end{equation}
The decay exponents for the density in $L^2$ and the velocity in $H^1$ are optimal in uniform estimates with the norm $B_0$.
\item If $w_0$ has a nonzero trace on a wall, then for this datum there are $c>0$ and $T>0$ such that for $t\ge T$,
\[
 c t^{-1/8}\le\norm{\varrho(t)-\bar\varrho_0}\le C B_0 t^{-1/8}.
\]
\item If $w_0$ has zero trace on both walls and its normal derivative has a nonzero trace on at least one wall, then the bounds in (ii) hold with $t^{-3/8}$ in place of $t^{-1/8}$.
\item In (ii) and (iii), the lower bound already holds for the $L^2$ norm restricted to a strip of width $C_*t^{-1/4}$ at any wall whose indicated trace is nonzero. The positive constant $C_*$ may depend on the datum. In particular, this strip contains at least a fixed positive fraction of the squared $L^2$ density error for all sufficiently large $t$.
\item In (ii) and (iii), let $\alpha=1/8$ and $\alpha=3/8$, respectively. Then for $t\ge1$,
\begin{equation}\label{incompatible-velocity}
 t^{1/2}\norm{\bu(t)}_{H^1}
 +t\norm{\partial_t\varrho(t)}
 +\norm{\bu(t)}_{H^2}
 \le C B_0 t^{-\alpha}.
\end{equation}
\end{enumerate}
\end{theorem}

Dalibard, Guillod and Leblond \cite[Theorem~1.2]{DGL2025} derived boundary layer profiles on the $t^{-1/4}$ scale for linear Stokes--transport about an affine background. With variable inertia and bounded positive stratification, Theorem~\ref{thm:channel-traces} gives two-sided rates for each of the two leading traces and a localized lower bound, using a compatible-data estimate and a supported adjoint test. In \autoref{fig:trace-strip}, a pairing of order $\ell^{j+1}$ is divided by the $O(\ell^{1/2})$ norm of the supported test, giving a density lower bound of order $\ell^{j+1/2}$. \par
\begin{figure}[tbp]
\centering
\begin{tikzpicture}[>=Stealth,font=\small]
\fill[blue!13] (0,0) rectangle (5.6,.58);
\draw[thick] (0,0) -- (5.6,0) node[right] {wall};
\draw[thick] (0,1.72) -- (5.6,1.72);
\draw[densely dotted,blue!70!black] (0,.58) -- (5.6,.58);
\draw[dashed,gray!70] (0,0) -- (0,1.72);
\draw[dashed,gray!70] (5.6,0) -- (5.6,1.72);
\node[above] at (2.8,1.72) {periodic channel};
\node[blue!60!black] at (2.8,.28) {$\operatorname{supp}\phi_\ell$};
\draw[<->] (-.22,0) -- (-.22,.58)
  node[midway,left] {$\ell$};
\node[anchor=west,align=left] at (6.35,.95) {
  $\ell=C_*t^{-1/4}$\strut\\[3pt]
  $\|\varrho-\bar\varrho_0\|_{L^2(0<r<\ell)}
    \gtrsim \ell^{j+1/2}$
};
\end{tikzpicture}
\caption{Localization of the adjoint test in Theorem~\ref{thm:channel-traces}.
Here $r$ is distance to the selected wall. The first nonzero trace has order
$j=0$ or $j=1$. The localized pairing is of order $\ell^{j+1}$ and the test
has $L^2$ norm of order $\ell^{1/2}$, giving the indicated lower bound.}
\label{fig:trace-strip}
\end{figure}

For comparison, we next consider the periodic Boussinesq model. Let $\T^2=(\R/2\pi\mathbb Z)^2$ and $\bar\rho,\nu,\delta,g>0$, and consider
\begin{equation}\label{bous}
\bar\rho\,\partial_t\bu=\nu\Delta\bu-\nabla P-g\varrho\,\be_2,\qquad\partial_t\varrho=\delta g\,u_2,\qquad\Div\bu=0\qquad\text{in }\T^2,
\end{equation}
with data of zero mean. This periodic comparison model retains a constant stratification gradient $-\delta g\be_2$ and replaces the inertia by the constant reference density $\bar\rho$. The vertical gravity vector $g\be_2$ is prescribed on the torus. The energy $\mathcal E(t)\coloneqq\frac12\bar\rho\norm{\bu}^2+\frac1{2\delta}\norm{\varrho}^2$ satisfies $\mathcal E'=-\nu\norm{\nabla\bu}^2$, and the hydrostatic part is $\Pi\varrho\coloneqq\frac1{2\pi}\int_{\T}\varrho\,dx_1$. Norms are the homogeneous Fourier norms $\norm{x}_{H^s}^2\coloneqq\sum_{k\neq0}|k|^{2s}|\hat x_k|^2$.

\begin{theorem}\label{thm:model}
Let $s\ge0$ and $x_0=(\bu_0,\varrho_0)\in H^s$ with zero mean and $\Div\bu_0=0$. There is $C=C(s,\bar\rho,\nu,\delta,g)$ such that the solution of \eqref{bous} satisfies:
\begin{enumerate}[(1)]
\item (Optimal relaxation rate.) $\norm{\bu(t)}+\norm{\varrho(t)-\Pi\varrho_0}\le C(1+t)^{-s/4}\norm{x_0}_{H^s}$ for all $t\ge0$.
\item (Optimal velocity decay.) For $0\le m\le s+2$ and $t\ge1$,
\begin{equation}\label{model-smoothing}
 \norm{\bu(t)}_{H^m}\le C t^{-\gamma(s,m)}\norm{x_0}_{H^s},
 \qquad
 \gamma(s,m)=\min\left\{\frac{s+3-m}{4},\frac{s+2-m}{2}\right\}.
\end{equation}
For $m\le s$ the bound holds with $(1+t)^{-\gamma(s,m)}$ for all $t\ge0$. The decay exponents are optimal, including the absence of uniform decay at $m=s+2$. For $m>s+2$ a uniform $H^s\to H^m$ velocity bound fails at each fixed positive time.
\item (Sharp threshold for integrated strain.) If $s>5/2$, then $\int_0^\infty\norm{\nabla\bu(t)}_{L^\infty}\,dt\le C\norm{x_0}_{H^s}$, and the Lagrangian flow is uniformly bi-Lipschitz. There exist zero-velocity initial data in $H^{5/2}$ with $\Pi\varrho_0=0$ for which the integral is infinite.
\item (Sharpness.) For every $s\ge0$ there are $c>0$ and data $(\mathbf 0,\varrho_{0,n})$ with $\Pi\varrho_{0,n}=0$ and $\norm{\varrho_{0,n}}_{H^s}=1$, namely multiples of $\cos x_1\cos(nx_2)$, such that $\norm{\varrho_n(t)}\ge c\,n^{-s}$ for $0\le t\le c\,n^4$. Hence the exponent in (1) is optimal for every $s$.
\end{enumerate}
\end{theorem}

\begin{remark}\label{rem:model}
For $k_1\ne0$ the roots of $\bar\rho\lambda^2+\nu|k|^2\lambda+\delta g^2k_1^2/|k|^2=0$ separate for large $|k|$ into a slow branch $\lambda_+\sim-\delta g^2k_1^2/(\nu|k|^4)$ and a viscous branch $\lambda_-\sim-\nu|k|^2/\bar\rho$. The sharp energy rate is $t^{-s/2}$ after subtracting the stationary hydrostatic energy. The same mode estimates hold for stress-free walls on $\T\times(0,\pi)$ in the sine-series Sobolev scales for the stream function and the density modes with nonzero horizontal frequency. The no-slip strain threshold for general data remains open.
\end{remark}

\subsection{Proof overview}

\autoref{fig:roadmap} traces the nonlinear argument from uniform estimates to hydrostatic relaxation and energy selection. The velocity-peak branch supplies the dissipation bound for Rayleigh--Taylor escape. The lower branches show where wall assumptions enter the decay estimate and where linear lower bounds and invariant solutions limit that estimate. The linear proofs use the density Schur complement, the buoyancy pencil and wall-localized tests.

\begin{figure}[t]
\centering
\begin{tikzpicture}[>=Stealth,font=\footnotesize,node distance=5mm and 7mm,
  box/.style={draw,rounded corners=2pt,align=center,inner sep=3pt,minimum height=8mm,fill=#1},
  box/.default=white]
\node[box=gray!10] (gwp) {Thm~\ref{thm:gwp}\\ uniform bounds};
\node[box=gray!10,right=of gwp] (dec) {Thm~\ref{thm:decay}\\ $\bu\to0$, burst};
\node[box=gray!10,right=of dec] (hyd) {Thm~\ref{thm:hydro}\\ $H^{-1}$ relaxation};
\node[box=gray!10,right=of hyd] (lor) {Thm~\ref{thm:lorenz}\\ Lorenz selection};
\node[box=green!8,below=of lor] (cas) {Thm~\ref{thm:casimir}\\ Casimir stability};
\node[box=blue!8,below=11mm of hyd] (rate) {Thm~\ref{thm:rate}\\ conditional rate};
\node[box=blue!8,below=of rate] (lip) {Cor~\ref{thm:lipred}\\ Lipschitz reduction};
\node[box=orange!12,below=11mm of dec] (rt) {Thm~\ref{thm:RT}\\ RT perturbations};
\node[box=orange!12,below=of rt,xshift=-14mm] (nm) {Thm~\ref{thm:nonmono}\\ invariant flows};
\node[box=violet!8,below=of cas] (lin) {Thm~\ref{thm:lin}\\ linear theory};
\node[box=blue!8,below=of lin] (wall) {Thm~\ref{thm:wall-rate}\\ constant wall values};
\node[box,dashed,below=of lip] (conj) {Open estimate\\ \eqref{lip-conj}};
\draw[->] (gwp) -- (dec); \draw[->] (dec) -- (hyd); \draw[->] (hyd) -- (lor);
\draw[->] (lor) -- (cas); \draw[->] (lor) -- (rate);
\draw[->] (dec) -- node[left,font=\scriptsize]{\eqref{burst-two-sided}} (rt);
\draw[->] (lip) -- (rate);
\draw[->,blue!60!black] (rate.south east) to[bend left=12] (wall.west);
\draw[->,dashed] (conj) -- (lip);
\draw[->,violet!70!black] (lin.north west) to[bend right=8] (rate.east);
\draw[->,red!70!black] (nm.east) to[bend right=10] (rate.south west);
\end{tikzpicture}
\caption{Relations among the main results. The purple arrow indicates the linear lower bound relevant to the conditional decay estimate. The red arrow indicates the counterexample without ordering control. The blue arrow adds constant wall values and uniform small $H^4$ control. The dashed box indicates an open nonlinear estimate.}
\label{fig:roadmap}
\end{figure}

\section{Preliminaries}\label{sec:prelim}

\begin{lemma}\label{lem:steady}
Let $(\bu,\rho,\Pi)$ be a steady solution of \eqref{main} with $\bu\in H^1_0\cap H^2$, $\rho\in L^\infty$ and $\nabla\Pi\in L^2$. Then $\bu\equiv\mathbf 0$ and $\nabla\Pi=-\rho\nabla f$.
\end{lemma}

\begin{proof}
Test the momentum equation with $\bu$. Since $\Div(\rho\bu)=\bu\cdot\nabla\rho=0$ in $\mathcal D'$ and $\bu|_{\partial\Omega}=0$, we have
$\int\rho(\bu\cdot\nabla\bu)\cdot\bu=-\tfrac12\int\Div(\rho\bu)|\bu|^2=0$,
$\int\nabla\Pi\cdot\bu=0$ and
$\int\rho\nabla f\cdot\bu=-\int f\Div(\rho\bu)=0$.
Hence $\nu\norm{\nabla\bu}_{L^2}^2=0$, so $\bu=0$ by Poincar\'e's inequality.
\end{proof}

We use the following Stokes estimate \cite{Temam1977,Galdi2011}.

\begin{lemma}\label{lem:stokes}
Let $1<p<\infty$ and $\mathbf F\in L^p(\Omega)$. The problem $-\nu\Delta\bu+\nabla\pi=\mathbf F$, $\Div\bu=0$, $\bu|_{\partial\Omega}=0$ has a unique velocity and a pressure modulo constants, with $\norm{\bu}_{W^{2,p}}+\norm{\nabla\pi}_{L^p}\le C\norm{\mathbf F}_{L^p}$. If $\mathbf F\in G$, then $\bu=0$.
\end{lemma}

We use the following interpolation bounds \cite{Adams}.

\begin{lemma}\label{lem:interp}
For $g\in H^1(\Omega)$ we have $\norm{g}_{L^4}^2\le C\norm{g}\norm{g}_{H^1}$, and $\norm{g}_{L^4}^2\le C\norm{g}\norm{\nabla g}$ if $g\in H^1_0$. For $2\le p<\infty$, $\norm{g}_{L^p}\le C\norm{g}^{2/p}\norm{g}_{H^1}^{1-2/p}$. Moreover $W^{1,p}\hookrightarrow L^\infty$ for $p>2$, and $\norm{g}\le C\norm{\nabla g}$ for $g\in H^1_0$.
\end{lemma}

\begin{lemma}\label{lem:transport}
Let $\bu\in L^1(0,T;W^{1,\infty}(\Omega))$ with $\Div\bu=0$ and $\bu\cdot\bn|_{\partial\Omega}=0$. Then the flow $X_t$ of $\bu$ is a measure-preserving bi-Lipschitz homeomorphism of $\overline\Omega$. For $\rho_0\in L^\infty$, the unique weak solution of $\partial_t\rho+\bu\cdot\nabla\rho=0$ is $\rho(t)=\rho_0\circ X_t^{-1}$. In particular $\mu_{\rho(t)}=\mu_{\rho_0}$ and $\int G(\rho(t))=\int G(\rho_0)$ for every continuous $G$.
\end{lemma}

\begin{proof}
The Cauchy--Lipschitz theorem applies, and trajectories cannot leave $\overline\Omega$ because $\bu$ is tangent to the boundary. Liouville's theorem gives $\det\nabla X_t=1$. Uniqueness of weak solutions holds by the DiPerna--Lions theory, or directly by duality since the field is Lipschitz.
\end{proof}

\begin{lemma}\label{lem:rearr}
\begin{enumerate}[(1)]
\item If $\sigma_n\in\Rr$ and $\sigma_n\to\sigma$ in $L^1$, then $\sigma\in\Rr$.
\item Every $\sigma\in\Kr$ satisfies $\alpha_1\le\sigma\le\alpha_2$, $\int\sigma=\int\rho_0$ and $\norm{\sigma}\le\norm{\rho_0}$. Moreover $\norm{\sigma}=\norm{\rho_0}$ if and only if $\sigma\in\Rr$.
\item (Bathtub principle) For every $\sigma\in\Kr$ and every measurable $A\subset\Omega$ with $|A|=a$,
\[
\int_A\sigma\dx\ \ge\ \int_{|\Omega|-a}^{|\Omega|}\rho_0^*(s)\,ds .
\]
\item $\Kr$ is convex.
\end{enumerate}
\end{lemma}

\begin{proof}
(1) $L^1$ convergence implies convergence in measure. Hence $\mu_{\sigma_n}(c)\to\mu_\sigma(c)$ at every continuity point $c$ of $\mu_\sigma$. Thus $\mu_\sigma=\mu_{\rho_0}$ outside a countable set, and everywhere by right continuity.

(2) The bounds and the mean are preserved under weak limits, and the norm is weakly lower semicontinuous. If $\sigma_n\rightharpoonup\sigma$ with $\sigma_n\in\Rr$ and $\norm\sigma=\norm{\rho_0}=\norm{\sigma_n}$, then $\sigma_n\to\sigma$ strongly, so $\sigma\in\Rr$ by (1).

(3) For $\sigma\in\Rr$, the Hardy--Littlewood inequality \cite[Thm.~3.4]{LiebLoss}, applied to $\sigma$ and $1-\mathbf 1_A$, gives $\int_{\Omega\setminus A}\sigma\le\int_0^{|\Omega|-a}\rho_0^*$. Subtracting from $\int\sigma=\int_0^{|\Omega|}\rho_0^*$ yields the claim. Finally, $\sigma\mapsto\int_A\sigma$ is weak-$*$ continuous.

(4) Take $\sigma_1,\sigma_2\in\Rr$ and $0\le\lambda\le1$. Partition $\Omega$ into cells of vanishing diameter. In each cell realize the mixture of their two local distributions with weights $\lambda$ and $1-\lambda$, using quantiles on the nonatomic cell. The global distribution remains that of $\rho_0$, and the cell integrals agree with those of $\lambda\sigma_1+(1-\lambda)\sigma_2$. Testing against continuous functions and then using density in $L^1$ gives weak-$*$ convergence to that convex combination. Taking closures proves the claim.
\end{proof}

\begin{lemma}\label{lem:sorted}
Assume \ref{A4} and let $F\coloneqq\Lambda\circ f$. Then:
\begin{enumerate}[(1)]
\item $|\{F<s\}|=s$ for all $s\in[0,|\Omega|]$.
\item $\rs=\rho_0^*\circ F\in\Rr$, and $\rs$ is a nonincreasing function of $f$.
\item $\int_{\{f>c\}}\rs=\int_{|\Omega|-|\{f>c\}|}^{|\Omega|}\rho_0^*$ for all $c$.
\item $\rs\in\Hy$.
\end{enumerate}
\end{lemma}

\begin{proof}
(1) For $0<s<|\Omega|$ let $c^*=\inf\{c:\Lambda(c)\ge s\}$. Since $\Lambda$ is continuous and nondecreasing, $\Lambda(f(\bx))<s$ holds if and only if $f(\bx)<c^*$. Hence $|\{F<s\}|=\Lambda(c^*)=s$. The two endpoints follow by monotone limits.

(2) Since $\rho_0^*$ is nonincreasing and right-continuous, $\{\rho_0^*>c\}=(0,\mu_{\rho_0}(c))$ within $(0,|\Omega|)$. The endpoint values have no effect on the distribution. By (1), $\mu_{\rs}(c)=|\{F<\mu_{\rho_0}(c)\}|=\mu_{\rho_0}(c)$. Since $\Lambda$ is nondecreasing, $\rs$ is a nonincreasing function of $f$.

(3) Up to a null set, $\{f>c\}=\{F>\Lambda(c)\}$. Moreover $F$ pushes Lebesgue measure on $\Omega$ forward to Lebesgue measure on $(0,|\Omega|)$, and $|\Omega|-\Lambda(c)=|\{f>c\}|$ by \ref{A4}.

(4) Write $\rs=\phi(f)$ with $\phi\coloneqq\rho_0^*\circ\Lambda$ nonincreasing and bounded, and set $\Phi(r)\coloneqq\int_m^r\phi$. Then $\Phi$ is Lipschitz and concave, so it is differentiable outside a countable set $N$, and $|f^{-1}(N)|=0$ by \ref{A4}. The chain rule gives $\nabla(\Phi\circ f)=\phi(f)\nabla f=\rs\nabla f$ a.e. Hence $\rs\nabla f\in G$.
\end{proof}

The following variational result is related to Burton's rearrangement theory \cite{Burton1987,Burton1989}. We give a proof for the present setting.

\begin{lemma}\label{lem:key}
Assume \ref{A4}. For every $\sigma\in\Kr$,
\[
\int\sigma f\ \ge\ \int\rs f,
\]
with equality if and only if $\sigma=\rs$. Moreover, if $\sigma_n\in\Rr$ and $\sigma_n\rightharpoonup\rs$ in $L^2$, then $\sigma_n\to\rs$ strongly in $L^2$.
\end{lemma}

\begin{proof}
Let $m\le f\le M$. By Fubini, $f(\bx)=m+\int_m^M\mathbf 1_{\{f(\bx)>c\}}dc$. Since $\int\sigma=\int\rs$,
\[
\int(\sigma-\rs)f=\int_m^M\Big(\int_{\{f>c\}}(\sigma-\rs)\Big)dc .
\]
By Lemma~\ref{lem:rearr}(3) and Lemma~\ref{lem:sorted}(3), the inner integral is nonnegative for every $c$. This gives the required inequality.

Assume now that equality holds. Then $\int_{\{f>c\}}(\sigma-\rs)=0$ for a.e.\ $c$. By \ref{A4}, the map $c\mapsto\int_{\{f>c\}}g$ is continuous for $g\in L^1$, so the identity holds for every $c$. The sets $\{f>c\}$ form a $\pi$-system generating $\sigma(f)$, the $\sigma$-algebra generated by $f$. By Dynkin's lemma, $\int(\sigma-\rs)\psi(f)=0$ for all bounded Borel $\psi$. Thus $\sigma-\rs$ is orthogonal in $L^2$ to the space of $\sigma(f)$-measurable functions, which contains $\rs$. Using Lemma~\ref{lem:rearr}(2),
\[
\norm{\rho_0}^2\ge\norm{\sigma}^2=\norm{\rs}^2+\norm{\sigma-\rs}^2=\norm{\rho_0}^2+\norm{\sigma-\rs}^2 ,
\]
which gives $\sigma=\rs$. For the last assertion, $\norm{\sigma_n}=\norm{\rho_0}=\norm{\rs}$; together with weak convergence this yields strong convergence.
\end{proof}

\begin{lemma}\label{lem:barbalat}
Let $y\ge0$ be absolutely continuous on $[0,\infty)$ with $y,\,y'\in L^1(0,\infty)$. Then $y(t)\to0$ as $t\to\infty$.
\end{lemma}

\begin{proof}
Since $y'\in L^1$, the limit $\lim_{t\to\infty}y(t)=y(0)+\int_0^\infty y'$ exists. Since $y\in L^1$, this limit is $0$.
\end{proof}

\begin{lemma}\label{lem:dual}
Let $\phi:[m,M]\to\R$ be nonincreasing and $\delta_1$-Lipschitz for some $\delta_1>0$. Then there is a convex function $\mathcal G:\R\to\R$ such that, for all $r\in\R$ and $c\in[m,M]$,
\begin{equation}\label{dual-1}
\mathcal G(r)-\mathcal G(\phi(c))+c\,(r-\phi(c))\ \ge\ \frac1{2\delta_1}\,(r-\phi(c))^2 .
\end{equation}
If moreover $-\phi'\ge\delta_0>0$ a.e., then $\mathcal G\in C^{1,1}$ with $\mathcal G'(\phi(c))=-c$, and the left-hand side of \eqref{dual-1} is also $\le\tfrac1{2\delta_0}(r-\phi(c))^2$.
\end{lemma}

\begin{proof}
Extend $\phi$ by affine functions of slope $-\delta_1$ outside $[m,M]$ and put $K(c)=-\int_m^c\phi$. Define $\mathcal G(r)=\sup_{d\in\R}(-rd-K(d))$. The quadratic tails of $K$ make the supremum finite. At $r_0=\phi(c)$ the supremum is attained at $d=c$ by convexity of $K$ and $K'(c)=-r_0$. Hence
\[
 \mathcal G(r)-\mathcal G(r_0)+c(r-r_0)
 =\sup_{t\in\R}\big\{-(r-r_0)t-[K(c+t)-K(c)-K'(c)t]\big\}.
\]
The bracket is at most $\delta_1t^2/2$ by integrating the Lipschitz bound on $K'$. Evaluating the supremum at $t=-(r-r_0)/\delta_1$ proves \eqref{dual-1}. If $-\phi'\ge\delta_0$, the bracket is at least $\delta_0t^2/2$, so maximizing the corresponding quadratic yields the upper bound. In this case the maximizer $d(r)=\phi^{-1}(r)$ is unique and $\delta_0^{-1}$-Lipschitz. Evaluating the supremum at $d(r)$ for $r'$ and at $d(r')$ for $r$ bounds the increment of $\mathcal G$ between $-d(r)(r'-r)$ and $-d(r')(r'-r)$. Thus $\mathcal G'(r)=-d(r)$, with $\mathcal G'(\phi(c))=-c$.
\end{proof}

\begin{lemma}\label{lem:nonexp}
Assume \ref{A4}, and for $g\in L^\infty(\Omega)$ let $g^\natural\coloneqq g^*\circ F$, with $F=\Lambda\circ f$ as in Lemma~\ref{lem:sorted}. Then $\norm{g^\natural-h^\natural}\le\norm{g-h}$ and $\norm{g^\natural-h^\natural}_{L^\infty}\le\norm{g-h}_{L^\infty}$ for all $g,h\in L^\infty(\Omega)$. Moreover, if $g=\psi\circ f$ with $\psi$ nonincreasing, then $g^\natural=g$.
\end{lemma}

\begin{proof}
By Lemma~\ref{lem:sorted}(1), $F$ pushes Lebesgue measure on $\Omega$ forward to Lebesgue measure on $(0,|\Omega|)$. Hence $\norm{g^\natural-h^\natural}_{L^2(\Omega)}=\norm{g^*-h^*}_{L^2(0,|\Omega|)}\le\norm{g-h}$ by the nonexpansivity of the nonincreasing rearrangement \cite{CZR1986}. For the $L^\infty$ bound, use $h-\eta\le g\le h+\eta$ with $\eta=\norm{g-h}_{L^\infty}$. Sorting preserves order and commutes with addition of constants by the quantile definition. If $g=\psi(f)$ with $\psi$ nonincreasing, then on each set $\{f>c\}$ the function $g$ takes the smallest values of its distribution. Hence $\int g f=\int g^\natural f$ by the layer-cake argument of Lemma~\ref{lem:key} (applied with $\rho_0\coloneqq g$), and Lemma~\ref{lem:key} gives $g=g^\natural$.
\end{proof}

\begin{lemma}\label{lem:interp2}
Let $\Omega=\T_L\times(0,h)$ and $s>0$. For every $q\in H^s(\Omega)$,
\begin{equation}\label{cell-interpolation}
 \norm q\le C_s\norm{q}_{(H^2_0)'}^{s/(s+2)}
                    \norm{q}_{H^s}^{2/(s+2)}.
\end{equation}
Here the test functions in $H^2_0$ are periodic horizontally and clamped on the walls. The exponent $s/(s+2)$ is optimal for a homogeneous estimate involving these two norms.
\end{lemma}

\begin{proof}
Partition the channel into rectangles of side lengths comparable to $\ell$, with $0<\ell\le\ell_0$ and $\ell_0<\min(L,h,1)$. Use an equal subdivision separately in each direction. Write $m=\lceil s\rceil-1$ and let $\mathcal P_m$ be the polynomials of total degree at most $m$. On each cell $Q$, let $\chi_Q$ be the rescaling of a fixed nonnegative smooth function compactly supported in the reference square and positive on an open subset. Define $p_Q\in\mathcal P_m$ by
\[
 \int_Q\chi_Q p_Q\overline v=\int_Q\chi_Q q\overline v
 \qquad (v\in\mathcal P_m).
\]
The weighted norm $\norm{\sqrt{\chi_Q}v}_{L^2(Q)}$ and the ordinary $L^2(Q)$ norm are equivalent on $\mathcal P_m$, by finite dimensionality on the reference square. The equivalence is uniform after rescaling. In particular the projection $q\mapsto p_Q$ is bounded in $L^2(Q)$, and
$\norm{q-p_Q}_{L^2(Q)}\le C\inf_{v\in\mathcal P_m}\norm{q-v}_{L^2(Q)}$.

Choose a polynomial $v\in\mathcal P_m$ matching the cell means of all derivatives of $q$ through order $m$, starting at the highest order. Iterated Poincar\'e inequalities give $\norm{q-v}_{L^2(Q)}\le C\ell^s|q|_{H^s(Q)}$. When $s=m+\sigma$ is noninteger, the top-order inequality is
$\norm{w-w_Q}_{L^2(Q)}^2=(2|Q|)^{-1}\iint_{Q\times Q}|w(x)-w(y)|^2\,dx\,dy\le C\ell^{2\sigma}|w|_{H^\sigma(Q)}^2$.
Summing over the cells gives $\norm{q-p}\le C\ell^s\norm q_{H^s}$.

In \autoref{fig:cell-moments} the support of $\chi_Qp_Q$ stays inside the cell even when the cell touches a physical wall. Extending each summand by zero and setting $\psi=\sum_Q\chi_Qp_Q$ therefore gives the clamped test used in the dual estimate. This test belongs to $H^2_0(\Omega)$. By scaling and disjointness, $\norm\psi_{H^2}\le C\ell^{-2}\norm p$. The moment identities give
\[
 c\norm p^2\le\sum_Q\int_Q\chi_Q|p_Q|^2
   =\operatorname{Re}\langle q,\psi\rangle
   \le C\ell^{-2}\norm{q}_{(H^2_0)'}\norm p .
\]
Combining the bounds for $p$ and $q-p$, we obtain
\[
 \norm q\le C\big(\ell^s\norm q_{H^s}+\ell^{-2}\norm q_{(H^2_0)'}\big).
\]
For $q\ne0$, balance the terms by choosing $\ell$ proportional to
$(\norm{q}_{(H^2_0)'}/\norm q_{H^s})^{1/(s+2)}$. The factor can be chosen to ensure $\ell\le\ell_0$ since the dual norm is bounded by the $H^s$ norm.

For optimality take $q_n=a(x_1)b(x_2)\cos(nx_2)$ with nonzero smooth periodic $a$ and $b\in C_c^\infty(0,h)$. The $L^2$ norm stays bounded away from zero by the Riemann--Lebesgue lemma, and $\norm{q_n}_{H^s}\le C_s n^s$ by differentiation and Sobolev interpolation. Two integrations by parts in $x_2$ give $\norm{q_n}_{(H^2_0)'}\le Cn^{-2}$. A homogeneous bound with the dual norm raised to any exponent in $(s/(s+2),1]$ would make $\norm{q_n}$ tend to zero.
\end{proof}

\begin{figure}[tbp]
\centering
\begin{tikzpicture}[>=Stealth,font=\small]
  \fill[blue!11] (2.28,0) rectangle (3.04,.76);
  \draw[gray!55,step=.76] (0,0) grid (4.56,2.28);
  \draw[white] (0,0) -- (0,2.28) (4.56,0) -- (4.56,2.28);
  \draw[dashed,gray!75] (0,0) -- (0,2.28) (4.56,0) -- (4.56,2.28);
  \draw[thick] (0,0) -- (4.56,0) (0,2.28) -- (4.56,2.28);
  \draw[blue!70!black,thick] (2.28,0) rectangle (3.04,.76);
  \node[blue!70!black] at (2.66,.38) {$Q$};
  \node[above] at (2.28,2.28) {cell partition};
  \node[below] at (2.28,0) {no-slip wall};
  \draw[->,gray!75] (4.86,1.10) -- (5.69,1.10);
  \begin{scope}[xshift=6cm]
    \draw[blue!70!black,thick] (0,0) rectangle (2.35,2.28);
    \fill[blue!15] (1.175,1.14) ellipse (.77 and .73);
    \draw[blue!60!black,dashed] (1.175,1.14) ellipse (.77 and .73);
    \node at (1.175,1.14) {$\chi_Q p_Q$};
    \node[above] at (1.175,2.28) {local test};
    \draw[<->] (0,-.15) -- (2.35,-.15) node[midway,below] {$\asymp\ell$};
  \end{scope}
\end{tikzpicture}
\caption{The test in Lemma~\ref{lem:interp2} is assembled from cell moments. Each
$\chi_Q p_Q$ vanishes in a neighborhood of the cell edges, including for cells
meeting a physical wall. Zero extension and summation give a clamped test
without a trace condition on $q$.}
\label{fig:cell-moments}
\end{figure}

\subsection*{First variation of the sorting map}

\begin{lemma}\label{lem:firstvar}
Assume \ref{A4}, and that $|\{c<f\le c+\eta\}|\le\Lambda_1\eta$ for all $c\in\R$ and $\eta>0$. Let $\rho_s=\varphi(f)$ with $\varphi\in W^{1,\infty}([m,M])$ and $-\varphi'\ge\delta_0>0$ a.e., let $\varrho_0\in L^\infty$ with $M_0\coloneqq\norm{\varrho_0}_{L^\infty}$, and set $\rho_\varepsilon\coloneqq\rho_s+\varepsilon\varrho_0$ for $\varepsilon>0$. Then $\norm{\rho_\varepsilon^\natural-\rho_s}_{L^\infty}\le\varepsilon M_0$, and
\[
\sup_{c\in\R}\Big|\int_{\{f>c\}}\big(\rho_\varepsilon^\natural-\rho_s-\varepsilon\,\mathbb E[\varrho_0|f]\big)\Big|\le\frac{4\Lambda_1M_0^2}{\delta_0}\,\varepsilon^2 .
\]
Consequently $\varepsilon^{-1}(\rho_\varepsilon^\natural-\rho_s)\overset*\rightharpoonup\mathbb E[\varrho_0|f]$ in $L^\infty(\Omega)$ as $\varepsilon\to0$.
\end{lemma}

\begin{proof}
Nonexpansivity in $L^\infty$ gives the first bound. Let $B=\{f>c\}$ and choose a bathtub minimizer $A$ of $\int_A\rho_\varepsilon$ with $|A|=|B|$. If $A\triangle B$ has positive measure, its threshold lies within $\varepsilon M_0$ of $\varphi(c)$, by comparison on $A\setminus B$ and $B\setminus A$. Both difference sets therefore lie in $|f-c|\le2\varepsilon M_0/\delta_0$. Each has measure at most $2\Lambda_1\varepsilon M_0/\delta_0$ and the difference of the two mean density values is at most $2\varepsilon M_0$. Thus
\[
 0\le\int_B\rho_\varepsilon-\int_B\rho_\varepsilon^\natural
 \le 4\Lambda_1M_0^2\varepsilon^2/\delta_0 .
\]
The equality $\int_B\rho_\varepsilon=\int_B(\rho_s+\varepsilon\mathbb E[\varrho_0|f])$ gives the claimed estimate. The bounded difference quotients are $\sigma(f)$-measurable, and the tails $\{f>c\}$ determine all their weak-$*$ limits by simple-function approximation. This gives the last assertion.
\end{proof}

\subsection*{Column energy and second variation in a gravitational channel}

\begin{proposition}\label{prop:column-energy}
Let $\Omega=\T_L\times(0,h)$, $f=gx_2$ with $g>0$, and $a\le\sigma\le b$. Denote its decreasing rearrangement by $\sigma^\natural$. For a level $c\in(a,b)$ and horizontal coordinate $x_1$, set
\[
 A_c(x_1)=\{z\in(0,h):\sigma(x_1,z)>c\},\qquad
 Z_c(x_1)=|A_c(x_1)|.
\]
Write $\eta_c(x_1)=|A_c(x_1)\setminus(0,Z_c(x_1))|$. The excess potential energy has the exact decomposition
\begin{equation}\label{column-energy}
 \begin{aligned}
 \int_\Omega(\sigma-\sigma^\natural)gx_2\,\dx
 &=g\int_a^b\int_{\T_L}
    \left(\int_{A_c(x_1)}z\,dz-\frac{Z_c^2}{2}\right)dx_1\,dc \\
 &\quad+\frac g2\int_a^b\int_{\T_L}(Z_c-\bar Z_c)^2\,dx_1\,dc .
 \end{aligned}
\end{equation}
The integrand in parentheses is at least $\eta_c^2$. If $\sigma$ is nonincreasing vertically on almost every column, the first integral vanishes and the excess energy is exactly the integrated variance of the column heights, with coefficient $gL/2$.

Let $\sigma^{\rm col}$ be the result of decreasing rearrangement separately on each column. If $\Delta=b-a>0$, then
\begin{equation}\label{column-coercivity}
 \begin{aligned}
 \int_\Omega(\sigma-\sigma^\natural)gx_2\,\dx
 &\ge \frac{g}{4\Delta}\int_{\T_L}
       \norm{\sigma(x_1,\cdot)-\sigma^{\rm col}(x_1,\cdot)}_{L^1(0,h)}^2\,dx_1 \\
 &\quad+\frac{g}{2\Delta}\int_{\T_L}
       \norm{\sigma^{\rm col}(x_1,\cdot)-\sigma^\natural}_{L^1(0,h)}^2\,dx_1 .
 \end{aligned}
\end{equation}
Both coefficients are sharp.
\end{proposition}

\begin{proof}
At level $c$, the sorted superlevel set has height $\bar Z_c$, so its vertical first moment is $L\bar Z_c^2/2$. Apply the layer-cake formula and use
$\int_{\T_L} Z_c^2-L\bar Z_c^2=\int_{\T_L}(Z_c-\bar Z_c)^2$.
In \autoref{fig:column-energy}(a), the exchanged pieces have length $\eta_c$. Their moment deficits are the integrals of $z-Z_c$ and $Z_c-z$, each at least $\eta_c^2/2$ by placing the piece adjacent to $Z_c$. Panel (b) shows the subsequent replacement of $Z_c$ by $\bar Z_c$, whose moment deficit is the variance term in \eqref{column-energy}.

By layer cake the column distances in \eqref{column-coercivity} are $2\int_a^b\eta_c\,dc$ and $\int_a^b|Z_c-\bar Z_c|\,dc$, respectively. Apply Cauchy--Schwarz in $c$ to the two terms of \eqref{column-energy}. Equality in the first coefficient is attained by exchanging two adjacent layers of equal thickness in a horizontally constant two-phase state, so the second term vanishes. Equality in the second is attained by $\sigma=a+\Delta\mathbf1_{\{z<Z(x_1)\}}$ for any nonconstant measurable $Z$ taking values in $(0,h)$, for which the first term vanishes and the height difference is independent of $c\in(a,b)$.
\end{proof}

\begin{figure}[tbp]
\centering
\begin{tikzpicture}[>=Stealth,font=\small]
  \begin{scope}
    \node[anchor=west] at (0,2.97) {(a) Sorting within a column};
    \fill[blue!16] (.3,0) rectangle (1.12,.72);
    \fill[blue!16] (.3,1.07) rectangle (1.12,1.59);
    \fill[orange!28] (.3,1.59) rectangle (1.12,1.94);
    \fill[blue!16] (3.30,0) rectangle (4.12,1.59);
    \fill[orange!28] (3.30,.72) rectangle (4.12,1.07);
    \foreach \x in {.3,3.30} {\draw[gray!70] (\x,0) rectangle +(0.82,2.52);}
    \draw[dashed,blue!65!black] (.12,1.59) -- (4.30,1.59)
      node[right] {$Z_c$};
    \draw[->,orange!75!black,thick] (1.28,1.81) to[out=5,in=150]
      (3.15,.91);
    \node[orange!75!black,font=\footnotesize] at (2.22,2.10) {equal lengths $\eta_c$};
    \node[below] at (.71,0) {$A_c$};
    \node[below] at (3.71,0) {$(0,Z_c)$};
  \end{scope}
  \begin{scope}[xshift=6.05cm]
    \node[anchor=west] at (0,2.97) {(b) Averaging column heights};
    \path[fill=blue!13] (0,0) -- (0,1.25) --
      plot[domain=0:4,samples=64] (\x,{1.25+.43*sin(90*\x)})
      -- (4,0) -- cycle;
    \draw[gray!70] (0,0) rectangle (4,2.52);
    \draw[blue!70!black,thick]
      plot[domain=0:4,samples=64] (\x,{1.25+.43*sin(90*\x)});
    \draw[orange!85!black,thick,dashed] (0,1.25) -- (4,1.25);
    \node[blue!70!black] at (1.02,1.92) {$Z_c(x_1)$};
    \node[orange!75!black,anchor=west] at (2.76,1.48) {$\bar Z_c$};
    \draw[<->,gray!80] (1,1.28) -- (1,1.66);
    \node[below] at (2,0) {periodic horizontal coordinate};
  \end{scope}
\end{tikzpicture}
\caption{The two contributions in \eqref{column-energy}. In (a), the level
set in a column is replaced by a bottom interval of the same length. In (b),
the graph of the resulting heights is replaced by its horizontal mean.
The first-moment difference in (b) is one half of the horizontal integral
of $(Z_c-\bar Z_c)^2$.}
\label{fig:column-energy}
\end{figure}

\begin{proposition}\label{prop:secondvar}
Let $\Omega=\T_L\times(0,h)$ and $f=gx_2$ with $g>0$. Suppose $\rho_s=\rho_s(z)\in C^3([0,h])$ is positive and $\rho_s'(z)\le-c_0<0$. Let $v\in C^3(\overline\Omega)$ have a constant trace on each wall. Set $\rho_\varepsilon=\rho_s+\varepsilon v$ and
$V(z)=\overline{(v-\bar v)^2}(z)$, with horizontal averages denoted by overbars. As $\varepsilon\to0$,
\begin{equation}\label{sort-second}
 \rho_\varepsilon^\natural(z)
 =\rho_s(z)+\varepsilon\bar v(z)
   -\frac{\varepsilon^2}{2}\frac{d}{dz}\left(\frac{V(z)}{\rho_s'(z)}\right)
   +O(\varepsilon^3)
\end{equation}
uniformly on $[0,h]$. With $\delta(z)=-\rho_s'(z)/g$,
\begin{equation}\label{ape-second}
 \int_\Omega (\rho_\varepsilon-\rho_\varepsilon^\natural)gx_2\,\dx
 =\frac{\varepsilon^2}{2}\int_\Omega\delta^{-1}|v-\bar v|^2\,\dx+O(\varepsilon^3).
\end{equation}
\end{proposition}

\begin{proof}
For small $|\varepsilon|$, the vertical derivative of $\rho_\varepsilon$ is strictly negative, and its endpoint values are independent of $x_1$. Thus at each level in the range the superlevel set is bounded by a unique graph. Let $q_\varepsilon(z)=\rho_\varepsilon^\natural(z)$ and write that graph as $Z_\varepsilon(x_1,z)$ when its density level is $q_\varepsilon(z)$. Equimeasurability gives the equations
\[
 \rho_s(Z_\varepsilon)+\varepsilon v(x_1,Z_\varepsilon)
 =q_\varepsilon,\qquad \overline{Z_\varepsilon}=z .
\]
The ordinary implicit function theorem first solves for the graph height as a function of the level, and then inverts its horizontal mean, whose derivative is uniformly bounded away from zero. The derivatives in $\varepsilon$ through order three are uniform up to the walls, by extending the original $C^3$ functions slightly past the endpoints. The level inside the channel agrees with this extension because the walls have constant values.

At $\varepsilon=0$, $Z_0=z$. Differentiating in $\varepsilon$ and averaging gives
$q'_0=\bar v$ and $Z'_0=(\bar v-v)/\rho_s'$, where primes on $q$ and $Z$ in these two identities denote $\varepsilon$ derivatives. A second differentiation and the zero average of $Z''_0$ give
\[
 q''_0
 =\rho_s''\overline{(Z'_0)^2}+2\overline{(\partial_zv)Z'_0}
 =\frac{\rho_s''V}{(\rho_s')^2}-\frac{\partial_z V}{\rho_s'}
 =-\partial_z\left(\frac{V}{\rho_s'}\right).
\]
This proves \eqref{sort-second}. For the energy, Proposition~\ref{prop:column-energy} gives only the variance term. Change level variables by $c=q_\varepsilon(z)$, so $Z_c=Z_\varepsilon(\cdot,z)$ and $\bar Z_c=z$. Then $Z_\varepsilon-z=\varepsilon(\bar v-v)/\rho_s'+O(\varepsilon^2)$ and $-\partial_z q_\varepsilon=-\rho_s'+O(\varepsilon)$ uniformly. Substitution in \eqref{column-energy} gives \eqref{ape-second}.
\end{proof}

\section{Uniform-in-time estimates and global well-posedness}\label{sec:unif}

In Lemmas~\ref{lem:energy}--\ref{lem:ut} we derive a priori estimates first for smooth solutions and then for the Stokes--Galerkin approximations described in Section~\ref{sec:proof-gwp}. The estimates depend on $\alpha_1,\alpha_2,\nu,\Omega,\norm{f}_{W^{2,\infty}}$ and $\norm{\bu_0}_{H^2}$. Section~\ref{sec:proof-gwp} explains how to pass to rough densities. The no-flux transport equation gives the distributional pairing
\begin{equation}\label{transport-weak}
\langle\partial_t\rho,\phi\rangle=\int\rho\,\bu\cdot\nabla\phi\dx
\qquad \big(\phi\in W^{1,1}(\Omega)\big).
\end{equation}

\begin{lemma}\label{lem:energy}
For all $t\ge0$, $E(t)+\nu\int_0^t\norm{\nabla\bu}^2=E(0)$. Consequently
\[
\sup_{t\ge0}\norm{\bu(t)}^2+\int_0^\infty\norm{\nabla\bu}^2dt\le C_0 .
\]
\end{lemma}

\begin{proof}
Testing the momentum equation with $\bu$ and using transport gives
$\tfrac12\frac{d}{dt}\int\rho|\bu|^2+\nu\norm{\nabla\bu}^2=-\int\rho\bu\cdot\nabla f=-\frac{d}{dt}\int\rho f$.
The bound follows from $\rho\ge\alpha_1$ and $|\int\rho f|\le\alpha_2\norm{f}_{L^1}$.
\end{proof}

\begin{lemma}\label{lem:H2}
For all $t\ge0$,
\[
\norm{\bu}_{H^2}+\norm{\nabla\Pi}\le C\norm{\sqrt\rho\,\partial_t\bu}+C_0\big(1+\norm{\nabla\bu}^2\big).
\]
\end{lemma}

\begin{proof}
Apply Lemma~\ref{lem:stokes} with $\mathbf F=-\rho\partial_t\bu-\rho\bu\cdot\nabla\bu-\rho\nabla f$. By Lemma~\ref{lem:interp} and Lemma~\ref{lem:energy},
\[
\norm{\rho\bu\cdot\nabla\bu}\le\alpha_2\norm{\bu}_{L^4}\norm{\nabla\bu}_{L^4}\le C\norm{\bu}^{1/2}\norm{\nabla\bu}\norm{\bu}_{H^2}^{1/2}\le\frac1{2C}\norm{\bu}_{H^2}+C_0\norm{\nabla\bu}^2,
\]
and $\norm{\rho\nabla f}\le C$. Absorbing the first term gives the claim.
\end{proof}

\begin{lemma}\label{lem:H1}
\[
\sup_{t\ge0}\norm{\nabla\bu(t)}^2+\int_0^\infty\norm{\sqrt\rho\,\partial_t\bu}^2dt\le C_0 .
\]
\end{lemma}

\begin{proof}
Set $y=\norm{\nabla\bu}^2$ and $x=\int\rho\nabla f\cdot\bu$. Test the equation with $\bu_t$. The force term equals
$-x'+\int\rho\bu\cdot\nabla(\nabla f\cdot\bu)\le-x'+Cy$.
Using Lemmas~\ref{lem:interp} and~\ref{lem:H2}, the convection term is bounded by
$\frac34\norm{\sqrt\rho\bu_t}^2+C_0y(1+y)$. Hence
\begin{equation*}
 \frac{d}{dt}\left(\frac\nu2y+x\right)+\frac14\norm{\sqrt\rho\bu_t}^2
 \le C_0 y(1+y).
\end{equation*}
Since $x$ is bounded and $y$ is integrable by Lemma~\ref{lem:energy}, integration and Gronwall's inequality give the claim.
\end{proof}

\begin{lemma}\label{lem:ut}
For the smooth approximations, let $z\coloneqq\norm{\sqrt\rho\,\partial_t\bu}^2$. Then
\[
\sup_{t\ge0}z(t)+\int_0^\infty\norm{\nabla\partial_t\bu}^2dt+\int_0^\infty|z'(t)|\,dt\le C_0 .
\]
\end{lemma}

\begin{proof}
Write $\bu_t=\partial_t\bu$. Differentiating the momentum equation in time,
\[
\rho\partial_t\bu_t+\rho\bu\cdot\nabla\bu_t-\nu\Delta\bu_t+\nabla\partial_t\Pi=-\partial_t\rho\,(\bu_t+\bu\cdot\nabla\bu+\nabla f)-\rho\,\bu_t\cdot\nabla\bu .
\]
Test with $\bu_t$. The terms $\pm\tfrac12\int\partial_t\rho|\bu_t|^2$ coming from $\rho\partial_t\bu_t\cdot\bu_t$ and from $\rho\bu\cdot\nabla\bu_t\cdot\bu_t$ cancel. Using \eqref{transport-weak} we get
\begin{align*}
\frac12z'+\nu\norm{\nabla\bu_t}^2
=&-2\int\rho\,(\bu\cdot\nabla\bu_t)\cdot\bu_t-\int\rho\,(\bu_t\cdot\nabla\bu)\cdot\bu_t\\
&-\int\rho\,\bu\cdot\nabla\big[(\bu\cdot\nabla\bu)\cdot\bu_t\big]-\int\rho\,\bu\cdot\nabla(\nabla f\cdot\bu_t)
\eqqcolon K_1+K_2+K_3+K_4 .
\end{align*}
Set $y=\norm{\nabla\bu}^2$. H\"older's inequality and Lemma~\ref{lem:interp} bound $K_1,K_2,K_4$ by $\frac{3\nu}8\norm{\nabla\bu_t}^2+C_0y(1+z)$ in sum, using $\norm{\bu}_{L^4}^4\le C_0 y$. For $K_3$, expand the derivative and apply H\"older in the three resulting terms:
\[
 |K_3|\le C\norm{\bu}_{L^\infty}
 \Big(\norm{\nabla\bu}_{L^4}^2\norm{\bu_t}
 +\norm{\bu}_{L^4}\norm{\nabla^2\bu}\norm{\bu_t}_{L^4}
 +\norm{\bu}_{L^4}\norm{\nabla\bu}_{L^4}\norm{\nabla\bu_t}\Big).
\]
Use $\norm{\nabla\bu}_{L^4}^2\le C y^{1/2}\norm{\bu}_{H^2}$ in the first term and $\norm{\bu}_{L^4}\le C y^{1/2}$ in the other two. The Sobolev and Poincar\'e bounds give $|K_3|\le C y^{1/2}\norm{\bu}_{H^2}^2\norm{\nabla\bu_t}$. Since $\norm{\bu}_{H^2}^2\le C_0(1+z)$, Young's inequality bounds this by $\frac\nu8\norm{\nabla\bu_t}^2+C_0y(1+z)^2$. Therefore
\begin{equation}\label{ut-1}
\Big|\frac12z'+\nu\norm{\nabla\bu_t}^2\Big|\le\frac\nu2\norm{\nabla\bu_t}^2+C_0\,y\,(1+z+z^2).
\end{equation}
The coefficient $y(1+z)$ belongs to $L^1(0,\infty)$ by Lemmas~\ref{lem:energy} and~\ref{lem:H1}. Gronwall bounds $z$, and \eqref{ut-1} then gives both integrals. Testing the initial momentum equation with $\bu_t(0)$ bounds $z(0)$ by $C_0$ using the lower density bound and $\bu_0\in H^2$.
\end{proof}

\begin{lemma}\label{lem:W2p}
For $p\in(2,\infty)$ and $T>0$, $\int_0^T\norm{\bu}_{W^{2,p}}^2\le C_0(1+T)$. The strain and density gradient satisfy the bounds in Theorem~\ref{thm:gwp}(3).
\end{lemma}

\begin{proof}
Lemma~\ref{lem:stokes} in $L^p$, together with $H^1_0\hookrightarrow L^p$ and $\norm{\bu}_{L^\infty}+\norm{\nabla\bu}_{L^p}\le C\norm{\bu}_{H^2}\le C_0$, gives $\norm{\bu}_{W^{2,p}}\le C_0(\norm{\nabla\bu_t}+1)$. Squaring and integrating yields $\int_0^T\norm{\bu}_{W^{2,p}}^2\,dt\le C_0(1+T)$ by Lemma~\ref{lem:ut}. For the strain, put $\theta=(p-2)/(2p-2)$. Interpolate $\norm{\nabla\bu}_{L^\infty}\le C_p\norm{\nabla\bu}^{\theta}\norm{\bu}_{W^{2,p}}^{1-\theta}$. For $0\le S<T$ this gives
\begin{equation}\label{strain-sublinear}
 \Lambda(T)-\Lambda(S)\le C_p (T-S)^{1/2}
  \left(\int_S^T\norm{\nabla\bu}^2\right)^{\theta/2}
  \left(\int_S^T\norm{\bu}_{W^{2,p}}^2\right)^{(1-\theta)/2}.
\end{equation}
Since $\nabla\bu\in L^2_tL^2_x$, taking $S=0$ gives $O((1+T)^{1-\theta/2})$. For the little-$o$ bound, fix $S$, divide by $T^{1-\theta/2}$ and let $T\to\infty$, then send $S\to\infty$. The exponents $1-\theta/2$ approach $3/4$ as $p\to\infty$. The density gradient bound follows from the inverse-flow chain rule and $\norm{\nabla X_t^{-1}}_{L^\infty}\le e^{\Lambda(t)}$, including at $q=1,\infty$.
\end{proof}

\begin{weightedchain}
Let $V\subset H\subset V'$ be a real Hilbert triple, and let $B(t)$ be a uniformly bounded self-adjoint operator on $H$ with a weak time derivative $B'$ bounded as a bilinear form on $V\times V$, uniformly in time. If $w\in L^2(0,T;V)$ and $m=Bw\in H^1(0,T;V')$, then $z=\langle Bw,w\rangle_H$ has an absolutely continuous representative and
\[
 z'=2\langle m',w\rangle-\langle B'w,w\rangle .
\]
\end{weightedchain}

\begin{proof}
By self-adjointness the increment of $z$ from $t$ to $t+h$ is
\[
 \langle m(t+h)-m(t),w(t+h)+w(t)\rangle
 -\langle (B(t+h)-B(t))w(t+h),w(t)\rangle .
\]
After division by $h$, the first term converges in $L^1_{\rm loc}$ by translation in $L^2(V)$ and difference quotients in $H^1(V')$. In the second, the uniform form bound permits replacing $w(t+h)$ by $w(t)$ in $L^1_{\rm loc}$, after which approximation by simple $V$-valued functions and scalar Lebesgue differentiation give $\langle B'w,w\rangle$. The two limits identify the distributional derivative in $L^1(0,T)$.
\end{proof}

\subsection{Proof of Theorem~\ref{thm:gwp}}\label{sec:proof-gwp}

\emph{Galerkin approximation and compactness.}
Let $P_N$ be the $L^2$ projection onto the first $N$ Dirichlet Stokes modes, with range $V_N$. Project the momentum equation onto $V_N$, set $\bu_N(0)=P_N\bu_0$ and transport smooth initial densities converging in every finite $L^p$ with the same bounds. The mass matrix is uniformly positive. The fixed-$N$ flow estimate gives local Lipschitz dependence of the transported density on the velocity coefficient path, so Picard iteration applies. The kinetic bound gives continuation. The estimates through Lemma~\ref{lem:ut} hold uniformly: the tests $\bu_N,\partial_t\bu_N$ belong to $V_N$, and the projected equation yields
\[
 \nu S_0\bu_N=-P_N\big[\rho_N\partial_t\bu_N+\rho_N(\bu_N\cdot\nabla)\bu_N+\rho_N\nabla f\big],
\]
where $S_0$ is the Dirichlet Stokes operator without viscosity. The $L^2$ contractivity of $P_N$ and the Stokes graph-norm estimate give the uniform elliptic bound used in those tests.

On each $[0,T]$, interpolation and compactness give $\bu_N\to\bu$ in $C(H^1)$. The transport equations give $\partial_t\rho_N=-\Div(\rho_N\bu_N)$ bounded in $L^\infty(0,T;H^{-1})$. Equicontinuity against $H^1_0$ tests and the uniform $L^2$ bound give weak $L^2$ convergence uniformly in time, along a subsequence. The strong velocity convergence identifies the limiting transport equation. Renormalization for the zero extension of the no-slip velocity \cite{DiPernaLions1989} gives $\rho\in C(L^2)$ and conservation of its $L^2$ norm. The conserved approximating norms and uniform weak convergence, tested against the compact path $\rho([0,T])$, give $\rho_N\to\rho$ in $C(L^2)$, hence in every finite $L^p$.

These convergences pass the momentum and energy equations to the limit. The full Stokes estimate gives Lemma~\ref{lem:W2p} and hence the bi-Lipschitz transport representation. Finally, $\norm{\nabla\bu}_{L^p}^p\le C_p\norm{\nabla\bu}^2\norm{\bu}_{H^2}^{p-2}$ gives the time integral in \eqref{gwp-1}.

\emph{Compactness of the kinetic time derivative.}
Write $w_N=\partial_t\bu_N$ and $m_N=P_N(\rho_Nw_N)$. The projected time-differentiated equation is
\[
 \partial_t m_N+\nu S_0w_N
 =-P_N\big[\rho_N(w_N\cdot\nabla\bu_N+\bu_N\cdot\nabla w_N)
            +(\partial_t\rho_N)(\bu_N\cdot\nabla\bu_N+\nabla f)\big],
\]
In the last term use the transport equation to place the spatial derivative on the product with the test function. In fact $\bu_N\cdot\nabla\bu_N+\nabla f$ is bounded in $H^1$ and $\bu_N$ in $L^\infty$, so the resulting pairing is bounded by $C\norm{P_N\bv}_{H^1}$. The terms involving $w_N$ are bounded by $C\norm{w_N}_{H^1}\norm{P_N\bv}_{H^1}$. Since $P_N$ is uniformly bounded on $V$, $\partial_t m_N$ is bounded in $L^2(0,T;V')$. Since $H$ embeds compactly in $V'$, $m_N$ converges strongly in $L^2(0,T;V')$. The limit is $\PP(\rho\,\partial_t\bu)$ by the strong density convergence. Pairing $m_N$ with $w_N\rightharpoonup\partial_t\bu$ in $L^2(0,T;V)$ gives
\[
 \lim_{N\to\infty}\int_0^T\norm{\sqrt{\rho_N}w_N}^2\,dt
 =\int_0^T\norm{\sqrt\rho\,\partial_t\bu}^2\,dt .
\]
This identity and the weak convergence of the weighted time derivatives give their strong convergence in $L^2$ in space and time.

Write $w=\partial_t\bu$, $m=\PP(\rho w)$ and $z=\norm{\sqrt\rho\,w}^2$. Strong local $L^1$ convergence of $z_N$ and the uniform variation estimate give $z\in BV(0,\infty)$. Since $z_N$ is uniformly bounded and $y_N=\norm{\nabla\bu_N}^2\to y$ locally uniformly, $y_N(1+z_N+z_N^2)$ converges in $L^1_{\rm loc}$. Thus the upper inequality in \eqref{ut-1} passes to the limit with weak lower semicontinuity for the time-derivative dissipation.

\emph{Strong time continuity.}
On $H$, set $B(t)v=\PP(\rho(t)v)$. It is uniformly coercive and strongly continuous, and the transport formula gives the bounded bilinear form $\langle B'v_1,v_2\rangle=\int\rho\bu\cdot\nabla(v_1\cdot v_2)$ on $V\times V$. The weighted chain rule applies to $w$ and $m$, so $z$ is locally absolutely continuous and its total variation bound gives $z'\in L^1(0,\infty)$. The limiting inequality is
\begin{equation}\label{z-measure}
 z'+\nu\norm{\nabla\partial_t\bu}^2\le C_0\norm{\nabla\bu}^2
 \quad\text{a.e.}
\end{equation}

The weakly continuous $H$ representative of $m$ follows from its $L^\infty(H)$ and $H^1(V')$ bounds. Set $w(t)=B(t)^{-1}m(t)$ at every time. Coercivity and strong continuity of $B$ imply weak continuity of $w$. Fix $t_0$ and choose Lebesgue times $t_j\to t_0$ with
$|t_j-t_0|^{1/2}\norm{w(t_j)}_{H^1}\to0$, which is possible since $w\in L^2_{\rm loc}(V)$. Using
$\norm{m(t_j)-m(t_0)}_{V'}\le |t_j-t_0|^{1/2}\norm{m'}_{L^2(V')}$ gives
\[
 z(t_j)=\langle m(t_0),w(t_j)\rangle+o(1)\longrightarrow
 \langle m(t_0),w(t_0)\rangle.
\]
The continuous representative of $z$ thus equals $\norm{\sqrt{\rho(t)}\,w(t)}^2$ at every time, including the initial time by a one-sided sequence. Coercivity, weak continuity and convergence of these weighted norms give $w\in C([0,T];H)$. The density multiplier is strongly continuous, and interpolation of $\bu\in C(H^1)\cap L^\infty(H^2)$ gives continuity of $\bu\cdot\nabla\bu$ in $L^2$. The Stokes estimate in the momentum equation now yields $\bu\in C(H^2)$ and $\nabla\Pi\in C(L^2)$. In particular the initial momentum equation holds in $L^2$, and $z(t)\to0$ by Lemma~\ref{lem:barbalat}. \qed

\subsection{Proof of Proposition~\ref{prop:euler-dependence}}

Write $\mathcal A_2=\partial_t\bu_2+\bu_2\cdot\nabla\bu_2+\nabla f$, which belongs to $L^2(0,T;H^1)$ by \eqref{gwp-1} and the two-dimensional Sobolev embeddings. The density difference satisfies
\[
 \partial_t d+\bu_1\cdot\nabla d=-\Div(\rho_2 U).
\]
Transporting a test function backwards by the flow of $\bu_1$ and using the uniform finite-time Lipschitz estimate gives
\begin{equation}\label{euler-transport}
 \norm{d(t)}_{W^{-1,4}}\le C_T
 \left(\norm{d(0)}_{W^{-1,4}}+\int_0^t\norm{\nabla U(s)}\,ds\right).
\end{equation}
Here the backward composition preserves the zero trace and has bounded $W^{1,4/3}$ norm, and $\norm{U}_{L^4}\le C\norm{\nabla U}$. Subtracting the momentum equations gives
\[
 \rho_1(\partial_tU+\bu_1\cdot\nabla U)-\nu\Delta U+\nabla(\Pi_1-\Pi_2)
 =-\rho_1 U\cdot\nabla\bu_2-d\,\mathcal A_2 .
\]
Test with $U$. The transport equation for $\rho_1$ gives the derivative of $\frac12\int\rho_1|U|^2$. This identity is justified by the chain rule along the flow, or by time regularization since the velocity difference has an $L^2$ time derivative. The convection term is bounded by $C\norm{\nabla\bu_2}\norm{U}_{L^4}^2$. For the density term, $\mathcal A_2\cdot U\in W^{1,4/3}_0$ and
\[
 \norm{\mathcal A_2\cdot U}_{W^{1,4/3}}
 \le C\norm{\mathcal A_2}_{H^1}\norm{\nabla U}
\]
by Leibniz's rule, the $H^1$ embedding into $L^4$ and approximation of $U$ by fields with zero boundary trace. After Young's inequality the derivative of the relative kinetic energy plus $c\norm{\nabla U}^2$ is at most
\[
 C\norm{\nabla\bu_2}^2\norm U^2
 +C\norm{\mathcal A_2}_{H^1}^2\norm d_{W^{-1,4}}^2.
\]
By \eqref{euler-transport},
$\norm d_{W^{-1,4}}^2\le C_T(\norm{d(0)}_{W^{-1,4}}^2+T\int_0^t\norm{\nabla U}^2)$.
Apply Gronwall's inequality to the sum of the relative kinetic energy and the integrated velocity dissipation. The coefficients are integrable on $[0,T]$; the bounds used above are uniform for the stated initial-data set. This proves \eqref{euler-dependence} and uniqueness in Theorem~\ref{thm:gwp}. \qed

\section{Relaxation to hydrostatic equilibria}\label{sec:relax}

\subsection{Proof of Theorem~\ref{thm:decay}}

The energy identity gives the limit of $E$. The functions $\norm{\nabla\bu}^2$ and $z$ belong to $W^{1,1}(0,\infty)$ by the bounds above, so they tend to zero. Interpolation with the uniform $H^2$ bound yields the velocity convergence in (2), and the kinetic energy tends to zero.

Write $R=\rho(\bu_t+\bu\cdot\nabla\bu)$. Then $R\to0$ in $L^2$ and $\Delta\bu\rightharpoonup0$ there, since $\norm{\Delta\bu}_{H^{-1}}\le\norm{\nabla\bu}\to0$. The projected momentum equations give (3), using compactness of $L^2\hookrightarrow H^{-1}$. They also give (4): $\nu\PP\Delta\bu=\bw+o_{L^2}(1)$ and $\nabla\Pi+\rho\nabla f=\nu\Delta\bu+o_{L^2}(1)$, while the Stokes estimate gives $\norm{\bu}_{H^2}\le C(\norm\bw+\norm R)$.

For (5), testing with $V$ or $H^1_0$ gives
\[
 \norm{\bw}_{V'}+\norm{\nabla\Pi+\rho\nabla f}_{H^{-1}}
 \le C(\norm{\nabla\bu}+\norm{\bu_t}).
\]
For $\psi\in H^2_0$, $\langle J,\psi\rangle=\langle\bw,\nabla^\perp\psi\rangle$. The right-hand side of the preceding norm bound is square integrable. Chebyshev's inequality and averaging on $[T,2T]$ give the two consequences of \eqref{baroclinic}.

For (6), $\norm{\bu(T+r)}\ge(m_T-C_1r)_+$. Poincar\'e's inequality and the energy identity yield
\[
 E(T)-E(T+\tau)\ge\frac{\nu}{C_P^2}\int_0^\tau(m_T-C_1r)_+^2\,dr
 =\frac{\nu}{3C_P^2C_1}[m_T^3-(m_T-C_1\tau)_+^3].
\]
On $[0,T]$, the same Lipschitz bound gives $\norm{\bu(t)}\ge(m_T-C_1(T-t))_+$. Adding the integral of its square to the future contribution gives \eqref{burst-two-sided}. Since $(m_T-C_1T)_+\le\norm{\bu_0}$, taking the supremum over $T$ gives \eqref{burst}. The weak-limit bound follows from $\int\rho_\infty f=E_\infty$. \qed

\autoref{fig:energy}(b) shows the lower envelope whose two integrals give \eqref{burst-two-sided}. When \ref{A4} holds, panel (a) places this dissipation bound in the budget $E(0)-E^\natural$, with residual energy $E_\infty-E^\natural$ above the sorted state.

\begin{figure}[t]
\centering
\begin{tikzpicture}[>=Stealth,font=\small]
\begin{scope}
\def\yE{5.0}\def\yP{3.9}\def\yI{2.2}\def\yS{0.6}
\draw[->] (0,0) -- (0,5.6) node[above]{energy};
\fill[blue!10] (0,\yI) rectangle (4.0,\yS);
\fill[red!10] (0,\yE) rectangle (4.0,\yI);
\foreach \y/\lab in {\yE/{E(0)},\yP/{\int\rho_0f},\yI/{E_\infty},\yS/{E^\natural}}
  {\draw[thin] (0,\y) -- (4.0,\y); \node[right] at (4.0,\y) {$\lab$};}
\draw[decorate,decoration={brace,amplitude=4pt,mirror}] (-0.25,\yE) -- (-0.25,\yP) node[midway,left=5pt]{$\KE_0$};
\draw[decorate,decoration={brace,amplitude=4pt,mirror}] (-0.25,\yP) -- (-0.25,\yS) node[midway,left=5pt]{$\APE_0$};
\draw[<->,thick,red!70!black] (0.5,\yE) -- (0.5,\yI) node[pos=0.68,right,align=left,font=\footnotesize]{$\nu\!\int_0^\infty\!\norm{\nabla\bu}^2$\\[-1pt]$\ge\frac{\nu[2m_T^3-(m_T-C_1T)_+^3]}{3C_P^2C_1}$};
\node[blue!60!black,align=center] at (2.0,{(\yI+\yS)/2}) {unreleased $\APE$};
\node at (2.0,-0.5) {(a) energy budget};
\end{scope}
\begin{scope}[xshift=7.6cm]
\draw[->] (0,0) -- (5.2,0) node[below]{$t$};
\draw[->] (0,0) -- (0,4.4) node[left]{$\norm{\bu(t)}$};
\def\tz{2.2}\def\m{3.2}\def\C{1.45}
\fill[orange!22] (\tz,0) -- (\tz,\m) -- ({\tz+\m/\C},0) -- cycle;
\fill[orange!22] ({max(0,\tz-\m/\C)},0) --
   (0,{max(0,\m-\C*\tz)}) -- (\tz,\m) -- (\tz,0) -- cycle;
\draw[orange!80!black,very thick] (0,{max(0,\m-\C*\tz)}) -- (\tz,\m);
\draw[orange!80!black,very thick] (\tz,\m) -- ({\tz+\m/\C},0);
\draw[very thick,blue!70!black,smooth] plot coordinates {(0,1.3) (0.6,1.8) (1.3,2.66) (\tz,\m) (2.9,2.67) (3.6,1.81) (4.4,0.95) (5,0.6)};
\fill (\tz,\m) circle (1.5pt) node[above right]{$m_T$};
\draw[densely dotted] (\tz,0) node[below left=-2pt]{$T$} -- (\tz,\m);
\draw[<->] (\tz,-0.3) -- ({\tz+\m/\C},-0.3)
  node[midway,below]{$m_T/C_1$};
\node[orange!70!black,rotate=-45,fill=white,inner sep=1pt]
  at ({\tz+1.3},{\m-1.6}) {\scriptsize slope $-C_1$};
\node at (2.6,-1.05) {(b) burst mechanism};
\end{scope}
\end{tikzpicture}
\caption{(a) The energy bounds in Theorems~\ref{thm:decay} and~\ref{thm:lorenz}. The case $E_\infty<\int\rho_0f$ is shown. For the invariant flows of Theorem~\ref{thm:nonmono}, $E_\infty=\int\rho_0f$. (b) The orange lines show the envelope $(m_T-C_1|t-T|)_+$. Integrating its square on the two sides of $T$ gives \eqref{burst-two-sided}.}
\label{fig:energy}
\end{figure}

\subsection{A quantitative transport estimate}

The transport equation gives an upper bound on the change in density over a finite time interval in terms of the energy dissipated during that interval.

\begin{proposition}\label{prop:speed}
Under the assumptions of Theorem~\ref{thm:gwp}, put
$D_0=(\operatorname*{ess\,sup}\rho_0-\operatorname*{ess\,inf}\rho_0)/2$.
For $0\le\tau\le t$,
\begin{equation}\label{speed}
\norm{\rho(t)-\rho(\tau)}_{H^{-1}}\le D_0C_P\Big(\frac{(t-\tau)\,\big(E(\tau)-E(t)\big)}{\nu}\Big)^{1/2},
\end{equation}
where $H^{-1}$ is the dual of $H^1_0(\Omega)$ equipped with the gradient norm. Consequently:
\begin{enumerate}[(1)]
\item if the density moves by a distance $d>0$ in $H^{-1}$ within a time $T$, then at least the energy $\nu d^2/(D_0^2C_P^2T)$ is dissipated;
\item let $\rho_\infty$ be a weak limit point at $H^{-1}$ distance $d>0$ from $\rho_0$. If $\norm{\rho(T)-\rho_\infty}_{H^{-1}}\le\varepsilon<d$, then
\[
T\ge\frac{\nu(d-\varepsilon)^2}{D_0^2C_P^2\,(E(0)-E_\infty)}\ \ge\ \frac{\nu(d-\varepsilon)^2}{D_0^2C_P^2\,(\KE_0+\APE_0)} ,
\]
the last inequality under \ref{A4};
\item $t\mapsto\rho(t)$ is uniformly $\frac12$-H\"older continuous into $H^{-1}$, with constant $D_0C_P(E(0)-E_\infty)^{1/2}\nu^{-1/2}$.
\end{enumerate}
\end{proposition}

\begin{proof}
Use the transport equation and $\Div\bu=0$ to write
$\langle\rho(t)-\rho(\tau),\zeta\rangle=\int_\tau^t\int(\rho-c)\bu\cdot\nabla\zeta$ for $\zeta\in H^1_0$, with $c$ the midpoint of the initial essential range. That range is preserved. Poincar\'e, Cauchy--Schwarz and the energy identity give \eqref{speed}. For (2) use $d\le\norm{\rho(T)-\rho_0}_{H^{-1}}+\varepsilon$ and $E(T)\ge E_\infty\ge E^\natural$, with the last inequality under \ref{A4}. The other assertions follow directly.
\end{proof}

\subsection{Proof of Theorem~\ref{thm:hydro}}

(1)--(2). Any sequence $t_n\to\infty$ has a weak-$*$ density subsequence with limit in $\Kr$. The projected force and potential energy converge as in Theorem~\ref{thm:decay}, giving $\PP(\rho_\infty\nabla f)=0$, $\int\rho_\infty f=E_\infty$, and $\nabla\Pi(t_n)\rightharpoonup-\rho_\infty\nabla f$ along this subsequence. The set $\mathcal S$ is weak-$*$ closed and bounded. Compactness of $L^2\hookrightarrow H^{-1}$ gives (2).

(3)--(4). The norm criterion is Lemma~\ref{lem:rearr}. If the orbit is precompact in $L^2$, every sequence of times tending to infinity has a subsequence converging in $L^2$ to $\mathcal S\cap\Rr$. It follows that $\PP(\rho(t)\nabla f)\to0$ in $L^2$, and Theorem~\ref{thm:decay}(4) applies.

(5). Taking the curl of $g\sigma\be_2=\nabla q$ gives $\partial_1\sigma=0$ in distributions. By slicing and connectedness of the horizontal sections, $\sigma$ depends only on $x_2$. The same argument holds on the periodic channel.

(6). The no-slip condition gives $\int_0^\infty\norm{\bu}_{L^\infty}<\infty$ from the assumed finite strain. Thus $X_t$ converges uniformly. The uniform Lipschitz bounds for $X_t^{\pm1}$ give uniform convergence of the inverses and a bi-Lipschitz limit. Measure preservation passes to the limit by testing against continuous functions. Approximation of $\rho_0$ in $L^p$ by continuous functions then gives $\rho(t)\to\rho_0\circ X_\infty^{-1}$ in every finite $L^p$. The limit lies in $\mathcal S\cap\Rr$, so (4) applies and a mixed weak limit is excluded. The same inverse-flow Lipschitz bound gives the gradient estimate.

(7). The closures of $\{\rho(t):t\ge T\}$ in $H^{-1}$ are nested compact connected sets by compactness and time continuity. Their intersection is the limit set and lies in $\mathcal S$. By Lemma~\ref{lem:rearr}(4), intersecting the convex set $\Kr$ with the hydrostatic and energy constraints gives a convex set. When $\mathcal S$ is a singleton, compactness yields convergence of the entire orbit in $H^{-1}$ and weakly in $L^2$. Part (3) gives the strong convergence criterion. \qed

\subsection{Proof of Theorem~\ref{thm:lorenz}}

(1)--(2). Lemmas~\ref{lem:sorted} and~\ref{lem:key} identify $\rs$ as the unique potential-energy minimizer in $\Kr$. This and the energy identity give \eqref{dissip-bound}. If $E_\infty=E^\natural$, every density weak limit minimizes the potential energy, hence equals $\rs$, and the conserved $L^2$ norm gives strong convergence. Conversely, the existence of any sequence with weak limit $\rs$ gives $E_\infty=E^\natural$ by testing against $f$. The energy identity makes this equivalent to equality in \eqref{dissip-bound}. The velocity and pressure conclusion follows from Theorem~\ref{thm:hydro}(4).

(3). In the expansion of $\norm{\rho(t)-\sigma_{\gamma,\beta}}^2$ the only time-dependent term is $2\gamma\int\rho(t) f$, which proves \eqref{lin-identity} and its limit. If $\sigma_{\gamma,\beta}\in\Rr$, it equals $\rs$ by Lemma~\ref{lem:nonexp}. Conversely a strong limit is equimeasurable with $\rho_0$. The distribution of $\sigma_{\gamma,\beta}$ is atom-free by \ref{A4}.

(4). Apply Lemma~\ref{lem:dual} with $r=\rho(t)$ and $c=f$ and integrate. The Casimir terms cancel by equimeasurability, leaving the lower bound on $\APE(t)$ and, for $-\phi'\ge\delta_0$, the upper bound. The remaining estimates follow from the energy identity, weak lower semicontinuity and \eqref{burst}. \qed

\subsection{Proof of Theorem~\ref{thm:rough-sorting} and Corollary~\ref{cor:rough-ape}}

For quantitative bathtub estimates in an optimization setting, see \cite{MazariRuizBalet2020}. Here we use a band-measure bound on $f$ at each density level and integrate over the full range of the density. First let $\sigma\in\Rr$. Put $a=\operatorname*{ess\,inf}\rho_0$ and $b=\operatorname*{ess\,sup}\rho_0$. For $r\in(a,b)$ write $A_r=\{\sigma>r\}$ and $B_r=\{\rs>r\}$. The two sets have equal measure. Up to a null set, $B_r$ has the form $\{f<c_r\}$ because $\rs$ is nonincreasing in $f$ and \ref{A4} holds. For empty or full $B_r$, the associated symmetric difference is null. Otherwise, with $\eta_r=|A_r\setminus B_r|=|B_r\setminus A_r|$, we have
\[
 \int_{A_r}f-\int_{B_r}f
 =\int_{A_r\setminus B_r}(f-c_r)+\int_{B_r\setminus A_r}(c_r-f).
\]
By layer cake, each integral on the right is at least
$\int_0^\infty(\eta_r-\omega_f(\ell))_+\,d\ell$. We also have
\[
 \mathcal D(\sigma)=\int_a^b\left(\int_{A_r}f-\int_{B_r}f\right)dr,\qquad
 \norm{\sigma-\rs}_{L^1}=2\int_a^b\eta_r\,dr .
\]
Since $\omega_f$ is nondecreasing, the sum of the two costs at level $r$ is at least $2\ell(\eta_r-\omega_f(\ell))_+$ for each $\ell>0$. Integration gives \eqref{rough-sort-modulus}. The atom-free distribution of $f$ implies $\omega_f(\ell)\to0$ as $\ell\downarrow0$.

Under the power bound the sum of the costs is at least
\[
 2\int_0^\infty(\eta_r-L_f\ell^\beta)_+\,d\ell
 =\frac{2\beta}{\beta+1}L_f^{-1/\beta}\eta_r^{q_\beta}.
\]
Jensen's inequality in $r$ and
$\norm{\sigma-\rs}_{L^p}^p\le\Delta_\rho^{p-1}\norm{\sigma-\rs}_{L^1}$ give the stated power estimate, including \eqref{rough-sort-sharp} when $\beta=1$. The bounds pass to $\Kr$ by weak-$*$ continuity of $\mathcal D$ and weak lower semicontinuity of the norms on this uniformly bounded set.

For sharpness, fix a two-layer sorted profile with density jump $\Delta_\rho$ at $x_2=z_0$ in $\T_L\times(0,h)$. Interchange the density values in the strips $z_0-\ell<x_2<z_0$ and $z_0<x_2<z_0+\ell$. The new profile is an equimeasurable rearrangement with $\norm{\sigma-\rs}_{L^p}^p=2L\ell\Delta_\rho^p$ and $\mathcal D(\sigma)=g L\Delta_\rho\ell^2$. \autoref{fig:sharp-exchange} shows that the exchanged strips abut the sorting interface. Each produces the smallest first moment for its length on its side of the interface. Hence these profiles give equality in \eqref{rough-sort-sharp} with $L_f=L/g$. They show that the power $2p$ cannot be reduced in this class.

Corollary~\ref{cor:rough-ape} follows by applying \eqref{rough-sort-sharp} to $\rho(t)\in\Rr$ and to $\rho_\infty\in\mathcal S$. For the final estimate, use $p=2$ and \eqref{wall-obstruction} to bound $\APE(t)$ below by $c e^{-2\Lambda(t)}$. Then take $t\to\infty$, using Theorem~\ref{thm:decay}(1). \qed

\begin{figure}[tbp]
\centering
\begin{tikzpicture}[>=Stealth,font=\small,x=1cm,y=.84cm]
  \begin{scope}
    \fill[blue!17] (0,0) rectangle (3.55,1.7);
    \fill[blue!46] (0,1.02) rectangle (3.55,1.7);
    \fill[orange!28] (0,1.7) rectangle (3.55,2.38);
    \draw[gray!75] (0,0) rectangle (3.55,3.4);
    \draw[dashed,thick] (0,1.7)--(3.55,1.7) node[right] {$z_0$};
    \node[above] at (1.775,3.4) {sorted layers};
    \node at (1.775,.5) {heavy};
    \node at (1.775,2.9) {light};
    \node at (1.775,1.36) {$\ell$};
    \node at (1.775,2.04) {$\ell$};
  \end{scope}
  \draw[->,thick] (4.33,1.7) -- (5.49,1.7)
      node[midway,above,font=\footnotesize] {exchange};
  \begin{scope}[xshift=5.85cm]
    \fill[blue!17] (0,0) rectangle (3.55,1.02);
    \fill[orange!28] (0,1.02) rectangle (3.55,1.7);
    \fill[blue!46] (0,1.7) rectangle (3.55,2.38);
    \draw[gray!75] (0,0) rectangle (3.55,3.4);
    \draw[dashed,thick] (0,1.7)--(3.55,1.7) node[right] {$z_0$};
    \node[above] at (1.775,3.4) {rearranged layers};
    \node at (1.775,1.36) {light strip};
    \node at (1.775,2.04) {heavy strip};
  \end{scope}
\end{tikzpicture}
\caption{Equality construction for \eqref{rough-sort-sharp}. Exchanging two
adjacent strips of thickness $\ell$ preserves the density distribution and
makes the first-moment cost on each side of $z_0$ minimal for the exchanged
measure.}
\label{fig:sharp-exchange}
\end{figure}

\subsection{Proof of Theorem~\ref{thm:mix-frontier}}

\emph{The weak closure.}
The bounds and mass are preserved under weak-$*$ limits. Conversely, partition the channel into rectangles of vanishing mesh. In each rectangle put the value $1$ on a measurable subset whose area equals the integral of $q$ over that rectangle and put $0$ on the rest. The resulting two-valued densities have exactly the initial distribution. Their normalized phases converge weak-$*$ to $q$: the error against a continuous test function is bounded by its modulus of continuity at the mesh scale, and continuous tests are dense in $L^1$. The identity for the mixing deficit follows by expanding $\norm{a+\Delta q}^2$ and using the conserved mass and $q_*^2=q_*$.

\emph{A sharp lower bound.}
For $0<y<1$, let $Z(y)=L^{-1}|\{q>y\}|$. The identity
$q(1-q)=\int_0^q(1-2y)\,dy$ and mass conservation give
\[
 \int_0^1 Z(y)\,dy=z_0,\qquad
 \frac{\mathcal M(q)}{L}=\int_0^1 (1-2y)(Z(y)-z_0)\,dy .
\]
At each level the smallest vertical first moment for a set of measure $LZ(y)$ is attained by the bottom layer of height $Z(y)$. Layer cake and the mass identity thus give
\[
 \frac{D(q)}{g\Delta L}
 \ge \frac12\int_0^1\big(Z(y)^2-z_0^2\big)\,dy
 =\frac12\int_0^1(Z(y)-z_0)^2\,dy .
\]
Since $\int_0^1(1-2y)^2\,dy=1/3$, Cauchy--Schwarz yields \eqref{mix-parabola}. Equality requires bottom layers for almost every $y$ and
$Z(y)=z_0+3(\mathcal M(q)/L)(1-2y)$. These heights lie in $[0,h]$ exactly when the stated transition fits between the walls. They determine the clipped linear profile in the equality case.

\emph{The full constrained minimum.}
For each $\ell>0$, $q_\ell$ in \eqref{mix-ramp} minimizes $D(q)/(g\Delta)-\ell\mathcal M(q)$ on the mass constraint. Indeed, subtract $(c_\ell-\ell)$ times the mass and minimize the resulting strictly convex quadratic pointwise. The relevant integrand is $\ell v^2+(x_2-c_\ell)v$. Its constrained minimizer $v_*=q_\ell(x_2)$ on $[0,1]$ satisfies
\[
 \ell v^2+(x_2-c_\ell)v-\big[\ell v_*^2+(x_2-c_\ell)v_*\big]
 =\ell(v-v_*)^2+(2\ell v_*+x_2-c_\ell)(v-v_*)\ge\ell(v-v_*)^2 .
\]
The mass and mixing terms cancel upon integration for two profiles at the same deficit, which proves \eqref{mix-stability}. The integral of the clipped profile is continuous and strictly increasing in $c_\ell$ when it is strictly between zero and the total area. Thus the mass constraint fixes $c_\ell$ uniquely and continuously in $\ell$. Comparing the two minimizers for $\ell_2>\ell_1>0$ gives $\mathcal M(q_{\ell_2})\ge\mathcal M(q_{\ell_1})$. Equality would force the two profiles to agree by strict convexity, which is impossible on the positive-length interval where one has values in $(0,1)$ because their slopes differ. The mixing deficit is therefore continuous and strictly increasing. It tends to zero as $\ell\downarrow0$, since for $\ell<\min(z_0,h-z_0)$ the mass is realized by $c_\ell=z_0+\ell$ and the level-height identity gives $\mathcal M(q_\ell)=L\ell/3$. It tends to the maximum as $\ell\uparrow\infty$ since the slope bound $(2\ell)^{-1}$ and the fixed mean imply uniform convergence to $z_0/h$. At each intermediate mixing deficit the profile with the corresponding parameter uniquely minimizes $D$. The zero endpoint follows by the bathtub principle, and the maximal endpoint by strict concavity of $q(1-q)$ and Jensen's inequality.

For \eqref{mix-exact}, reflection and exchange of the two phases preserve the two deficits and interchange $z_0$ and $h-z_0$. Take $z_0=s_0\le h/2$ and write $E_q=D(q)/(g\Delta L)$. The ramp has level heights $Z(y)=\min(h,\max(0,c_\ell-2\ell y))$. For a layered decreasing profile the inequalities for the first moment above are identities.

If $\ell\le s_0$, the heights are $s_0+\ell(1-2y)$, so $m_q=\ell/3$ and $E_q=\ell^2/6$. If $s_0\le\ell\le h^2/(4s_0)$, the height graph is a triangle of base $r=c_\ell/(2\ell)$ and height $c_\ell$. Its area is $s_0$, hence $r=\sqrt{s_0/\ell}$. Its first moment in $y$ is $s_0r/3$ and $\int_0^1 Z^2\,dy=2s_0c_\ell/3$. This gives
\[
 m_q=s_0-\frac{2s_0^{3/2}}{3\sqrt{\ell}},\qquad
 E_q=\frac{2s_0^{3/2}\sqrt{\ell}}3-\frac{s_0^2}2.
\]
If $\ell\ge h^2/(4s_0)$, the profile is $s_0/h+(h/2-z)/(2\ell)$ throughout the channel. Its mean is $s_0/h$ and $\int_0^h(z-h/2)^2\,dz=h^3/12$. Taking its quadratic and first moments gives
$m_q=m_{\max}-h^3/(48\ell^2)$ and $E_q=s_0(h-s_0)/2-h^3/(24\ell)$.
Eliminating $\ell$ gives \eqref{mix-exact}.

To see differentiability across the two contacts, denote by $\ell_i$ the parameters for $m_1<m_2$. The optimality of each ramp for its own penalized problem gives
\[
 \ell_1\le\frac{F(m_2)-F(m_1)}{m_2-m_1}\le\ell_2.
\]
The inverse parameter is continuous and strictly increasing, so $F'(m_q)=\ell$ throughout $(0,m_{\max})$ and $F$ is strictly convex. For an actual weak limit use Theorem~\ref{thm:hydro}(1) to identify $D(q)=E_\infty-E^\natural$. The exact minimum and the parabolic bound give the two assertions for weak limits. \qed

\subsection{Proof of Theorem~\ref{thm:casimir}}

Since $\rho_s\nabla f=\nabla\Phi(f)$ with $\Phi'=\varphi$, we have $\rho_s\in\Hy$. Take $\mathcal G$ from Lemma~\ref{lem:dual} with $\phi=\varphi$, so $\mathcal G'(\rho_s)=-f$. Define
\[
\mathcal H(\rho)\coloneqq\int_\Omega\big[\mathcal G(\rho)-\mathcal G(\rho_s)-\mathcal G'(\rho_s)(\rho-\rho_s)\big]\dx=\int\mathcal G(\rho)-\int\mathcal G(\rho_s)+\int f(\rho-\rho_s).
\]
The two bounds in Lemma~\ref{lem:dual} give
\[
\frac1{2\delta_1}\norm{\rho-\rho_s}^2\le\mathcal H(\rho)\le\frac1{2\delta_0}\norm{\rho-\rho_s}^2 .
\]
By Lemma~\ref{lem:transport}, $\int\mathcal G(\rho(t))=\int\mathcal G(\rho_0)$ (a Casimir), so $\mathcal H(\rho(t))-\mathcal H(\rho_0)=\int(\rho(t)-\rho_0)f$. Substituting this relation into the energy identity of Lemma~\ref{lem:energy} gives
\[
\frac12\norm{\sqrt\rho\bu(t)}^2+\mathcal H(\rho(t))+\nu\int_0^t\norm{\nabla\bu}^2=\frac12\norm{\sqrt{\rho_0}\bu_0}^2+\mathcal H(\rho_0),
\]
which implies \eqref{casimir}. For $2\le p<\infty$, stability in $L^p$ follows from $\norm{\rho-\rho_s}_{L^p}^p\le(\alpha_2+\norm{\rho_s}_{L^\infty})^{p-2}\norm{\rho-\rho_s}^2$. For $1\le p<2$, it follows from H\"older's inequality on the bounded domain. The bound for $\rho_\infty$ follows from weak lower semicontinuity of the norm. \qed

\section{Conditional rates and wall compatibility}\label{sec:rates}

\subsection{Proof of Theorem~\ref{thm:rate}}

For (a), Theorem~\ref{thm:hydro}(5) gives a layered weak limit, and the cone of nonincreasing horizontal means is weakly closed in $L^2$. For (b), extract a subsequence of $t_j$ with strong $L^2$ limit $\sigma$. Theorem~\ref{thm:hydro} gives a layered limit equimeasurable with $\rho_0$. Since horizontal averaging and sorting are $L^2$ contractions, $d(t_j)\to0$ implies $\sigma=\sigma^\natural=\rs$. Theorem~\ref{thm:lorenz}(2) now gives convergence along the entire orbit and the velocity and pressure limits.

For (c)--(d), set $q=\rho-\bar\rho$, $y=\norm{\nabla\bu}^2$ and $z=\norm{\sqrt\rho\bu_t}^2$. The slopes of $\rho_0^*(Lc/g)$ lie between $-a_1 L/g$ and $-a_0 L/g$. Theorem~\ref{thm:lorenz}(4), horizontal orthogonality and nonexpansivity of sorting give
\begin{align}
 \APE&\asymp\norm{\rho-\rs}^2, \label{r1}\\
 \norm{\rho-\rs}^2&=\norm q^2+\norm{\bar\rho-\rs}^2
 \le\norm q^2+(\norm q+d)^2 .\label{r2}
\end{align}
For $\psi\in H^2_0$ take $\chi=\partial_1^{-1}(\psi-\bar\psi)$ with zero horizontal mean. It is clamped and $\norm\chi_{H^2}\le C\norm\psi_{H^2}$. Since $\langle q,\psi\rangle=-g^{-1}\langle\bw,\nabla^\perp\chi\rangle$, the projected momentum estimate yields
\begin{equation}\label{r3}
 \norm q_{(H^2_0)'}\le C\norm\bw_{V'}\le C(y+z)^{1/2}.
\end{equation}
Cell-moment interpolation gives $\norm q^2\le C B^{2(1-\theta)}(y+z)^\theta$, where $\theta=s/(s+2)$.

Use the modified energy $\mathfrak F=E-E^\natural+\kappa z$, with $\kappa>0$ sufficiently small in \eqref{z-measure}. Poincar\'e's inequality gives
\begin{equation}\label{conditional-diss}
 -\mathfrak F'\ge c(y+\norm{\nabla\bu_t}^2)\ge c'(y+z).
\end{equation}
Since $y+z$ is bounded, \eqref{r1}--\eqref{r2} and interpolation give $\mathfrak F\le C[(y+z)^\theta+d^2]$. In (c), $d\le K\norm q$ and integration of $\mathfrak F'\le-c\mathfrak F^{1/\theta}$ yields $\mathfrak F(t)\le C(1+t)^{-s/2}$. In (d), we have
\[
 -\mathfrak F'\ge c\big(\mathfrak F-C D_*^2(1+t)^{-2r}\big)_+^{1/\theta}.
\]
With $p=\min(s/2,2r)$, the function $A(1+t)^{-p}$ is a supersolution when $A$ is sufficiently large, since $p+1\ge p/\theta$. This gives $\mathfrak F(t)\le C(1+t)^{-p}$. In both cases, $|y'|\le y+\norm{\nabla\bu_t}^2$ and \eqref{conditional-diss} give $y(t)\le C\mathfrak F(t-1)$ by averaging over $[t-1,t]$ for $t\ge1$. Finally,
\begin{equation}\label{sorted-stokes}
 \norm{\bu}_{H^2}+\norm{\nabla\Pi+\rs\nabla f}
 \le C\big(\norm{\bu_t}+\norm{\bu}_{L^\infty}\norm{\nabla\bu}+\norm{\rho-\rs}\big)
\end{equation}
by the Stokes estimate, which gives the stated velocity and pressure bounds.

For \eqref{conditional-strain-tail}, the energy bound in (c) gives
$\int_T^{2T}(\norm{\bu}_{H^1}^2+\norm{\bu_t}_{H^1}^2)\,dr\le CT^{-s/2}$ for $T\ge1$. Absorb $\bar\rho\nabla f$ into the pressure and apply the $L^p$ Stokes estimate for $2<p<\infty$. Since $\norm{q}\le CT^{-s/4}$ on $[T,2T]$ and $q$ is bounded in $H^s$, Sobolev interpolation gives $\norm{q}_{L^p}\le C_p T^{-s/4+(1-2/p)/4}$. The remaining forcing terms are bounded in $L^1(T,2T;L^p)$ by $C_p(T^{1/2-s/4}+T^{1-s/2})$, using the time-derivative estimate above and the pointwise $H^2$ bound on $\bu$. Thus
\[
 \int_T^{2T}\norm{\bu}_{W^{2,p}}\,dr\le C_p T^{1-s/4+(1-2/p)/4}.
\]
With $a=(p-2)/(2p-2)$, interpolation and H\"older's inequality yield
\[
 \int_T^{2T}\norm{\nabla\bu}_{L^\infty}\,dr
 \le C_p T^{a/2}
  \left(\int_T^{2T}\norm{\nabla\bu}^2\,dr\right)^{a/2}
  \left(\int_T^{2T}\norm{\bu}_{W^{2,p}}\,dr\right)^{1-a}
 \le C_p T^{1-s/4-a/4}.
\]
Choose $p$ so that $s/4+a/4-1>\gamma$ and sum over dyadic intervals. The finite-time estimates handle $t<1$. Since $\bu\in L^1(0,\infty;C^1)$ by no slip, the flow and its derivative converge uniformly. Uniform bounds for the inverse derivatives and volume preservation give the asserted $C^1$ diffeomorphism. The transport identity and strong density convergence identify its action on $\rho_0$. \qed

\subsection{Proof of Corollary~\ref{thm:lipred}}

The inverse flow satisfies $\|\nabla Y_t-I\|_{L^\infty}\le e^{\Lambda_\infty}-1=\eta$ by its derivative equation. The chain rule gives the two estimates of (1), hence monotonicity of the horizontal mean and uniform $H^1$ control. Inverting the strictly decreasing initial vertical profiles gives level graphs $x_2=\zeta_c(x_1)$ and
\[
 -\mu_{\rho_0}'(c)=\int_{\T_L}|\partial_2\rho_0(x_1,\zeta_c(x_1))|^{-1}\,dx_1
 \in[L/\norm{\partial_2\rho_0}_{L^\infty},L/a].
\]
Thus $\rho_0^*=\mu_{\rho_0}^{-1}$ is bi-Lipschitz. Apply Theorem~\ref{thm:hydro}(6) for the flow limit and Theorem~\ref{thm:rate} with $s=1$ to identify the sorted limit and its rate. \qed

\subsection{Proof of Theorem~\ref{thm:wall-rate}}

The no-slip flow preserves the constant density value on each wall. By Sobolev embedding and \eqref{wall-reg}, $\partial_2\rho\le-a_s/2$ for sufficiently small $\varepsilon$, and the inverse distribution calculation in the proof of Corollary~\ref{thm:lipred} gives \eqref{r1} with uniform constants. Since $\norm{\rho-\bar\rho}_{H^4}\le2K\varepsilon$, Lemma~\ref{lem:interp2} and \eqref{r2}--\eqref{r3}, with $d=0$, give
\[
 \norm{\rho-\rs}^2\le C\varepsilon^{2/3}(y+z)^{2/3}.
\]
To track the prefactor, use $\mathfrak F$ from \eqref{conditional-diss}. Initially $\APE_0\le C\norm{\rho_0-\rho_s}^2$ by sorting and \eqref{r1}. Testing the initial momentum equation with $\bu_t(0)$, after absorbing $\rho_s\nabla f$ into the pressure, gives $z(0)\le C\varepsilon^2$. Thus $\mathfrak F(0)\le C\varepsilon^2$. Equation \eqref{conditional-diss} and $|y'|\le y+\norm{\nabla\bu_t}^2$ imply $\sup(y+z)\le C\varepsilon^2$. Consequently
\[
 \mathfrak F\le C\varepsilon^{2/3}(y+z)^{2/3},\qquad
 \mathfrak F'\le-c\varepsilon^{-1}\mathfrak F^{3/2}.
\]
This yields $\mathfrak F(t)\le C\varepsilon^2(1+t)^{-2}$. The time averaging used for \eqref{rate} proves \eqref{wall-rate}, and \eqref{sorted-stokes} gives \eqref{wall-strong}.

For $T\ge1$, time integration of \eqref{conditional-diss} gives
\begin{equation}\label{wall-diss-tail}
 \int_T^\infty (\norm{\bu}_{H^1}^2+\norm{\bu_t}_{H^1}^2)\,dr
 \le C\varepsilon^2 T^{-2}.
\end{equation}
Odd extension across the no-slip walls preserves $H^2$ and $H^1$. Splitting the resulting Fourier series at frequency $N\ge2$ and using Cauchy--Schwarz gives
\[
 \norm{\bu}_{L^\infty}\le C\sqrt{\log(2+N)}\,\norm{\bu}_{H^1}
                  +C N^{-1}\norm{\bu}_{H^2}.
\]
Indeed the low and high sums of the squared reciprocal Sobolev weights are $O(\log(2+N))$ and $O(N^{-2})$, respectively. Integrate on $[T,2T]$ with $N\asymp\sqrt{2+T}$, using \eqref{wall-diss-tail} and \eqref{wall-strong}. Dyadic summation gives \eqref{wall-flow}.

To control the flow derivatives, take $2<p<\infty$ and absorb $\bar\rho\nabla f$ into the pressure. With $q=\rho-\bar\rho$, the $L^p$ Stokes estimate gives
\[
 \norm{\bu}_{W^{2,p}}
 \le C_p(\norm{\bu_t}_{H^1}+\norm{\bu}_{H^2}^2+\norm q_{L^p}).
\]
Put $\alpha_p=(1-2/p)/4$. Because $\norm q\le\norm{\rho-\rs}$ and $\norm q_{H^4}\le C\varepsilon$, Sobolev interpolation gives $\norm q_{L^p}\le C_p\varepsilon(1+t)^{-1+\alpha_p}$. Integration on $[T,2T]$ gives $\int_T^{2T}\norm{\bu}_{W^{2,p}}\le C_p\varepsilon T^{\alpha_p}$ by \eqref{wall-diss-tail} and \eqref{wall-strong}. Let $a=(p-2)/(2p-2)$. The inequality
$\norm{\nabla\bu}_{L^\infty}\le C_p\norm{\nabla\bu}^{a}\norm{\bu}_{W^{2,p}}^{1-a}$
and H\"older's inequality yield
\[
 \int_T^{2T}\norm{\nabla\bu}_{L^\infty}
 \le C_p T^{a/2}\left(\int_T^{2T}\norm{\nabla\bu}^2\right)^{a/2}
              \left(\int_T^{2T}\norm{\bu}_{W^{2,p}}\right)^{1-a}
 \le C_p\varepsilon T^{-a/4}.
\]
Here $(1-a)\alpha_p=a/4$ and $a/4$ ranges over $(0,1/8)$. On $[0,1]$, \eqref{conditional-diss} controls $\bu_t$ in $L^2(H^1)$ by $C\varepsilon$, and \eqref{wall-strong} and \eqref{wall-reg} control $\bu$ in $H^2$ and $q$ in $H^4$ pointwise. The same Stokes estimate and $W^{2,p}$ embedding give an initial strain bound of $C_p\varepsilon$. Dyadic summation gives the strain estimate. The flow derivative equation bounds $\norm{\nabla X_t-\nabla X_\infty}_{L^\infty}$ by a constant times the strain tail. The uniform bi-Lipschitz bounds and unit Jacobian give the measure-preserving $C^1$ diffeomorphism. Transport and \eqref{wall-rate} give $\rho_0=\rs\circ X_\infty$.

Finally, $L^\infty$ nonexpansivity of sorting and monotonicity of $\bar\rho$ give $\norm{\rho-\rs}_{L^\infty}\le2\norm q_{L^\infty}$. The two-dimensional interpolation
$\norm q_{L^\infty}\le C\norm q^{3/4}\norm q_{H^4}^{1/4}$
bounds this by $C\varepsilon(1+t)^{-3/4}$. The transport identity gives
$\rs(X_{\infty,2})-\rs(X_{t,2})=(\rho(t)-\rs)\circ X_t$. The lower slope bound gives \eqref{wall-flow-vertical}.

For \eqref{wall-flow-vertical-grad}, write $R=\rs$ and let $Z(x_1,c)$ be the inverse of $\rho(t,x_1,\cdot)$, at fixed $t$. Then $\bar Z(R(z))=z$, and
\[
 \frac1{R'(z)}
 =\overline{\frac1{\partial_2\rho(t,\cdot,Z(\cdot,R(z)))}}.
\]
The uniform lower and upper bounds for $|\partial_2\rho|$ give a Lipschitz bound for taking reciprocals in this identity. Moreover,
$|Z(x_1,R(z))-z|\le C\norm{\rho-R}_{L^\infty}$. Since $\partial_2^2\rho$ is uniformly bounded by the $H^4$ hypothesis, it follows that
\[
 \norm{\partial_2\rho-R'}_{L^\infty}
 \le C\big(\norm{\partial_2 q}_{L^\infty}+\norm{\rho-R}_{L^\infty}\big).
\]
The horizontal derivative is $\partial_1q$. Interpolation gives $\norm{\nabla q}_{L^\infty}\le C\norm q^{1/2}\norm q_{H^4}^{1/2}$, proving the density bound. The inverse formula also gives a uniform bound on $R''$ by one more differentiation, using $\partial_2^2\rho$ and the lower slope bound. Differentiate
$R(X_{\infty,2})-R(X_{t,2})=(\rho-R)\circ X_t$ and use the uniform Jacobian-matrix bounds and \eqref{wall-flow-vertical}. This proves the remaining estimate in \eqref{wall-flow-vertical-grad}. \qed

\subsection{Proof of Proposition~\ref{thm:wall-obstruction}}

\emph{(1) Wall value.} Let $q_b$ denote the essential limit of $q$ at the wall. By the trace mismatch and continuity there are an interval $I\subset\T_L$ of positive length and a number $d>0$ such that $|\rho_0(x_1,b)-q_b|\ge3d$ on $I$. Choose $0<r\le h$ so that $|q(x_2)-q_b|\le d$ for almost every $x_2$ within $r$ of the wall and $r\norm{\nabla\rho_0}_{L^\infty}\le d$. With $Y_t$ as above, the boundary condition and the flow estimate give
\[
 |Y_t(x_1,x_2)-(x_1,b)|\le e^{\Lambda(t)}|x_2-b|.
\]
The distance can be taken in the product metric of the channel. For $x_1\in I$ and $|x_2-b|\le r e^{-\Lambda(t)}$, the transported density differs from $\rho_0(x_1,b)$ by at most $d$. \autoref{fig:wall-transport} shows the thickness conversion: the inverse flow fixes the wall point and expands distances by at most $e^{\Lambda(t)}$, so the thin current strip maps inside the initial neighborhood of radius $r$. Thus $|\rho(t,x_1,x_2)-q(x_2)|\ge d$ on a strip of area $|I|r e^{-\Lambda(t)}$, up to a null set. This proves the first inequality in \eqref{wall-obstruction}. For the second, use the identity $\rho(t,x_1,b)=\rho_0(x_1,b)$ and the gradient bound for $\rho(t)$ across a strip of thickness $r_1/(1+\norm{\nabla\rho(t)}_{L^\infty})$, where $r_1=\min(r,d)$. The lower bounds on accumulated strain and gradient growth follow by substituting the assumed algebraic upper bound.

\emph{(2) Normal derivative.} At a wall where $\rho_0$ is constant, the finite-time $L^1(C^1)$ velocity bound gives a $C^1$ inverse flow $Y_t$. Its tangential derivative at the wall is $\be_1$ because it fixes every wall point. Volume preservation then gives $\partial_2Y_{t,2}=1$ there. The inset in \autoref{fig:wall-transport} shows how $DY_t$ changes the normal vector while preserving its normal component. Since $\partial_1\rho_0=0$ on this wall, the chain rule gives $\partial_2\rho(t)=\partial_2\rho_0$ on it. The trace of $\partial_2v$ is therefore also preserved. If $\norm{v(t)}\to0$ and $\norm{\rho(t)}_{H^s}$ stays bounded along a sequence for $s>3/2$, interpolation gives $v(t)\to0$ in $H^{s'}$ for $3/2<s'<s$ along that sequence. The trace theorem would give $\tau_b=0$.

For \eqref{wall-normal-obstruction}, rescale the normal interval of length $\ell\le h$ at the wall to $(0,1)$. The one-dimensional estimate $|w'(0)|\le C(\norm{w}_{L^2}+\norm{w^{(4)}}_{L^2})$ follows from the trace theorem and equivalence with the $H^4$ norm. Applying it on each vertical line and taking the horizontal $L^2$ norm gives
\[
 \tau_b\le C\big(\ell^{-3/2}\norm v+\ell^{5/2}\norm{\partial_2^4v}\big).
\]
Choose $\ell=\min(1,h)(\norm v/\norm v_{H^4})^{1/4}$ for $v\ne0$ and rearrange the resulting inequality to obtain the growth bound. \qed

\begin{figure}[tbp]
\centering
\begin{tikzpicture}[>=Stealth,font=\small]
  \begin{scope}
    \fill[blue!16] (1.2,0) rectangle (3.7,.42);
    \draw[very thick] (0,0) -- (4.9,0);
    \draw[blue!60!black,dashed] (1.2,0) rectangle (3.7,.42);
    \fill[blue!65!black] (2.45,.34) circle (1.35pt)
      node[above=2pt] {$x$};
    \fill (2.45,0) circle (1pt) node[below=3pt] {$x_b$};
    \draw[<->] (.93,0) -- (.93,.42)
      node[midway,left] {$r e^{-\Lambda(t)}$};
    \node[below] at (4.07,-.12) {wall};
    \node[above,blue!65!black] at (2.46,.97) {observed strip over $I$};
  \end{scope}
  \draw[->,thick] (4.86,.88) -- (6.25,.88)
    node[midway,above] {$Y_t$};
  \begin{scope}[xshift=6.45cm]
    \fill[orange!19] (2.24,0) arc (180:0:1.02) -- cycle;
    \draw[very thick] (1.00,0) -- (5.50,0);
    \draw[dashed,orange!80!black] (2.24,0) arc (180:0:1.02);
    \fill[orange!70!black] (3.60,.66) circle (1.35pt)
      node[above right=1pt] {$Y_t(x)$};
    \fill (3.26,0) circle (1pt) node[below=3pt] {$x_b$};
    \draw[dotted,thick,orange!80!black] (3.26,0) -- (3.26,1.02)
      node[midway,left] {$r$};
    \node[above,orange!70!black] at (3.26,1.32) {initial neighborhood};
  \end{scope}
  \begin{scope}[shift={(3.55,-2.12)}]
    \draw[thick] (0,0) -- (5.1,0);
    \draw[densely dotted] (1.05,.85) -- (4.10,.85);
    \draw[->,semithick,blue!65!black] (1.05,0) -- (1.05,.85)
       node[above] {$\be_2$};
    \draw[->,semithick,orange!80!black] (3.55,0) -- (4.10,.85)
       node[above] {$DY_t(x_b)\be_2$};
    \draw[->] (1.50,.40) -- (3.12,.40) node[midway,above] {$DY_t(x_b)$};
    \fill (1.05,0) circle (1pt);
    \fill (3.55,0) circle (1pt);
  \end{scope}
\end{tikzpicture}
\caption{The nonlinear wall obstruction. The inverse map fixes $x_b$, the
wall point below $x$, and has Lipschitz constant at most $e^{\Lambda(t)}$.
A point within $r e^{-\Lambda(t)}$ of the wall is mapped within distance $r$
of $x_b$. On the selected interval $I$ this preserves a fixed mismatch
between the transported density and the target profile. The inset depicts
the derivative at the fixed wall point. The inverse map may shear the normal
vector, but its normal component is preserved.}
\label{fig:wall-transport}
\end{figure}

\begin{remark}
For example, let $\rho_0\in C^1(\overline\Omega)$ be constant with value $a$ on the lower wall and suppose $\operatorname*{ess\,sup}\rho_0>a$. The lower-wall limit of the sorted profile is $\operatorname*{ess\,sup}\rho_0$, giving the mismatch in Proposition~\ref{thm:wall-obstruction}. Hence convergence to the sorted state in $L^2$ for these data would force both unbounded accumulated strain and divergence of the density Lipschitz norm.
\end{remark}

\subsection{Proof of Theorem~\ref{thm:nonmono}(5)--(6)}
Apply Theorem~\ref{thm:rate}(b) or Corollary~\ref{thm:lipred} to the state at $T_0$, respectively. The no-slip flow preserves the constant wall values and equimeasurability, so the sorted profile at $T_0$ is still $\rs$. In (6) the needed $C^1$ regularity at $T_0$ follows from $\bu\in L^1(0,T_0;C^1)$ and transport.

\section{Invariant flows and Rayleigh--Taylor perturbations}\label{sec:RT}

\subsection{Invariant shear and swirl flows}

\emph{(1)--(2).} Uniqueness preserves horizontal translations in the channel and rotations in the disk. Incompressibility and the wall condition eliminate the normal or radial velocity, so $\rho$ remains the initial profile. The remaining scalar velocity solves the respective heat equation in the theorem. The weighted heat generator and Poincar\'e's inequality give exponential $L^2$ decay of the velocity and its time derivative, since the $H^2$ data belong to the generator domain. In the channel the Dirichlet first eigenvalue is $\pi^2/h^2$, giving the stated $\lambda$. In the disk the centripetal force is radial and absorbed by pressure, so Stokes regularity gives the $H^2$ estimate in both cases.

\emph{Sharpness, (3).} In both cases $\rho(t)\equiv\rho_0$ and $\bu\to0$, so $E_\infty=\int\rho_0f$ and $\nu\int_0^\infty\norm{\nabla\bu}^2=\KE_0$. Both increasing profiles in (3) satisfy $\rho_0\neq\rs$, so $\APE_0>0$ by Lemma~\ref{lem:key}. This proves (b). For (c), $r(x_2)=\beta+\gamma x_2$ has $\rho_0^*(s)=\beta+\gamma h-\gamma s/L$, which is bi-Lipschitz. Moreover $\rho(t)=\rho_0\in H^s$ for all $t$, and $\rs=\beta+\gamma(h-x_2)\neq\rho_0$. For linear instability, the profiles satisfy $\nabla\rho_0=H_s\nabla f$, where $H_s=\gamma/g>0$ in the channel and $H_s=2\gamma/g>0$ in the disk. For a growing mode with factor $e^{\lambda t}$, eliminate the density perturbation as $-\lambda^{-1}H_s\bv\cdot\nabla f$. On $V$, the resulting quadratic eigenvalue equation has forms $\lambda^2 m+\lambda d-B$, with
\[
 m(\bv,\bv)=\int\rho_0|\bv|^2,\qquad d(\bv,\bv)=\nu\int|\nabla\bv|^2,\qquad
 B(\bv,\bv)=\int H_s|\bv\cdot\nabla f|^2.
\]
The largest Rayleigh quotient $B/(\lambda^2m+\lambda d)$ is continuous for $\lambda>0$. It tends to zero as $\lambda\to\infty$ and to infinity as $\lambda\downarrow0$, using a solenoidal test supported near a point with $\nabla f\ne0$ and $\bv\cdot\nabla f\ne0$. For each $\lambda>0$ this positive supremum is attained by compactness of $V\hookrightarrow L^2$. At the value where the supremum equals one, the variational equation gives a growing mode. Stokes regularity gives its velocity in $H^2$, completing (a). For (d), $(\mathbf 0,\rs)$ is steady because $\rs\in\Hy$ (Lemma~\ref{lem:sorted}). \qed

\emph{(4)} combines Theorem~\ref{thm:hydro}(1), the energy identity and Theorem~\ref{thm:lorenz}(1).

In \autoref{fig:frozen}(a), the horizontal velocity is tangent to the density layers and preserves their order. Panel (b) illustrates the shear decay with a heat mode of constant density. The weighted energy estimate above applies to the variable-density shear.

\begin{figure}[t]
\centering
\begin{tikzpicture}[>=Stealth,font=\small]
\begin{scope}
\shade[top color=gray!75,bottom color=gray!8] (0,0) rectangle (5.2,2.6);
\draw[very thick] (0,0) -- (5.2,0); \draw[very thick] (0,2.6) -- (5.2,2.6);
\draw[dashed] (0,0) -- (0,2.6); \draw[dashed] (5.2,0) -- (5.2,2.6);
\foreach \y in {0.33,0.65,...,2.3}{
  \pgfmathsetmacro{\L}{1.25*sin(180*\y/2.6)}
  \draw[->,thick,blue!70!black] (2.0,\y) -- ++(\L,0);}
\draw[blue!70!black] (2.0,0) .. controls (3.65,0.85) and (3.65,1.75) .. (2.0,2.6);
\draw (2.0,0) -- (2.0,2.6);
\draw[->,thick] (5.8,2.2) -- (5.8,1.2) node[midway,right]{$g$};
\node[above] at (2.6,2.65) {heavy};
\node[below] at (2.6,-0.05) {light};
\node at (2.6,-0.75) {(a) $\rho(t)\equiv\rho_0=r(x_2)$, $\;r'>0$};
\end{scope}
\begin{scope}[xshift=7.8cm]
\begin{axis}[width=5.6cm,height=4.2cm,xlabel={$U(t,x_2)$},ylabel={$x_2$},
  xmin=0,xmax=1.1,ymin=0,ymax=1,ytick={0,1},yticklabels={$0$,$h$},xtick=\empty,
  every axis plot/.append style={thick,smooth,samples=40,domain=0:1},
  legend style={font=\scriptsize,at={(1.02,1)},anchor=north west,draw=none}]
\foreach \t/\c in {0/blue!80!black,0.3/blue!60,0.8/blue!40,1.6/blue!25}{
  \edef\tmp{\noexpand\addplot[color=\c,variable=\noexpand\y] ({exp(-\t*2.0)*sin(deg(pi*\noexpand\y))},{\noexpand\y});}
  \tmp}
\legend{$t=0$,$t_1$,$t_2$,$t_3$}
\end{axis}
\node at (2.2,-1.05) {(b) shear decays like $e^{-\lambda t}$};
\end{scope}
\end{tikzpicture}
\caption{(a) An invariant shear flow with a heavy-over-light density profile (Theorem~\ref{thm:nonmono}(1)). The density remains fixed as the shear decays. (b) The curves show the decay of a sine mode for the constant-density scalar heat equation. They illustrate the shear decay in the variable-density case.}
\label{fig:frozen}
\end{figure}

\subsection{Energy release after Rayleigh--Taylor instability}

Part (1) is Theorems~\ref{thm:decay} and~\ref{thm:hydro}. For (2), the Hadamard instability theorem \cite[Theorem~1.3]{Li2025} gives a fixed velocity $L^1$ escape size $\varepsilon$ at a time $T^\delta$ for each small $\delta$. Theorem~\ref{thm:gwp} gives a uniform bound $C_1$ for $\norm{\bu_t}$, since the density bounds and initial $H^2$ velocity bounds are uniform in $\delta$. At escape, $\norm{\bu(T^\delta)}\ge m=\varepsilon|\Omega|^{-1/2}$. Reduce the range of $\delta$ so that $\norm{\bu_0^\delta}\le\eta m$ and $\KE_0^\delta\le\kappa$. Write $T=T^\delta$. On $[0,T+m/C_1]$, $\norm{\bu(t)}\ge(m-C_1|t-T|)_+$. As $(m-C_1T)_+\le\norm{\bu_0^\delta}\le\eta m$, integrating the square of this bound before and after $T$ gives
\[
 \nu\int_0^{T+m/C_1}\norm{\nabla\bu}^2\,dt
 \ge \frac{\nu}{3C_P^2C_1}\big(2m^3-(m-C_1T)_+^3\big)\ge4\kappa.
\]
The energy identity and $\KE_0^\delta\le\kappa$ give the energy bound for all later times.

For (3), $\int\rho(t)f\le E(t)$ and (2) imply $\int(\rho_s-\rho(t))f\ge\kappa$ for $t\ge t_*^\delta$ after reducing $\delta_*$, since $\int(\rho_s-\rho_0^\delta)f=O(\delta)$. The same holds at any weak limit. The $L^2$ estimate follows by Cauchy--Schwarz. For $H^{-1}$, truncate $f$ in a boundary strip of thickness $2\ell$. The truncated test has $H^1_0$ norm at most $C\ell^{-1/2}$, and the truncation error is $C\ell$ by the uniform density bound. Fix $\ell>0$ so that the truncation error is less than $\kappa/2$. Then the duality pairing with the truncated test gives an $H^{-1}$ distance at least $c\kappa\ell^{1/2}$, with constants independent of the perturbation size.

For (4), $h=(\nabla\rho_s\cdot\nabla f)/|\nabla f|^2$ is continuous near the unstable point. Hence $\nabla\rho_s\cdot\nabla f>0$ there, so $\rho_s$ increases along a short integral curve of $\nabla f$. By continuity there are two open sets on which the orderings of $f$ and $\rho_s$ agree strictly. This excludes a decreasing rearrangement and gives $\APE(\rho_s)>0$. Nonexpansivity gives $|\APE_0^\delta-\APE(\rho_s)|\le2\delta\norm{\sigma_0}\norm f$. The numerator of $\theta^\delta$ is at least $3\kappa$ and at most $\APE_0^\delta$. The convergence criterion and the distance bound follow from Theorem~\ref{thm:lorenz}(2), (4).

For (5), the layer structure is Theorem~\ref{thm:hydro}(5). The mass is conserved and $g\int\sigma x_2=E_\infty$, giving the center-of-mass bounds. Theorem~\ref{thm:hydro}(3) gives the criterion for strong convergence along the defining subsequence. Nonexpansivity gives $\norm{(\rho_0^\delta)^\natural-\rho_s^\natural}\le\delta\norm{\sigma_0}$. \qed

\autoref{fig:RT} illustrates the possibilities in Remark~\ref{rem:RT} with two-phase data. Relative to the arrangement in panel (a), panel (b) gives the energy minimum, panel (c) gives an incomplete rearrangement, and panel (d) gives a profile in the weak closure with reduced variance.

\begin{figure}[t]
\centering
\begin{tikzpicture}[font=\small]
\pgfplotsset{rtax/.style={width=3.7cm,height=4.6cm,xmin=0.7,xmax=2.3,ymin=0,ymax=1,
  xtick={1,2},xticklabels={$\rho_-$,$\rho_+$},ytick={0,1},yticklabels={$0$,$h$},
  xlabel={$\rho$},ylabel={$x_2$},ylabel style={yshift=-8pt},title style={font=\small,align=center}}}
\newcommand{\cm}[1]{\draw[red!70!black,thick,dashed] (axis cs:0.7,#1) -- (axis cs:2.3,#1);
  \node[red!70!black,anchor=south east,font=\scriptsize] at (axis cs:2.3,#1){$z$};}
\begin{axis}[rtax,title={(a) initial profile\\ (heavy on top)}]
\addplot[very thick,black] coordinates {(1,0) (1,0.5) (2,0.5) (2,1)}; \cm{0.583}
\end{axis}
\begin{axis}[rtax,xshift=3.9cm,title={(b) complete\\ overturning},ylabel={},yticklabels={}]
\addplot[very thick,blue!70!black] coordinates {(2,0) (2,0.5) (1,0.5) (1,1)}; \cm{0.417}
\end{axis}
\begin{axis}[rtax,xshift=7.8cm,title={(c) partial,\\ no mixing},ylabel={},yticklabels={}]
\addplot[very thick,blue!70!black] coordinates {(2,0) (2,0.25) (1,0.25) (1,0.75) (2,0.75) (2,1)}; \cm{0.5}
\end{axis}
\begin{axis}[rtax,xshift=11.7cm,title={(d) partial,\\ with mixing},ylabel={},yticklabels={}]
\addplot[very thick,blue!70!black] coordinates {(2,0) (2,0.25) (1.5,0.25) (1.5,0.75) (1,0.75) (1,1)}; \cm{0.4375}
\end{axis}
\end{tikzpicture}
\caption{Layered profiles illustrating Remark~\ref{rem:RT}. For $h=1$, $\rho_-=1$ and $\rho_+=2$, the initial center of mass in (a) is $7/12$. The profiles in (b)--(d) have centers of mass $5/12$, $1/2$ and $7/16$, respectively. Profile (b) is sorted, (c) is a different rearrangement, and (d) contains an intermediate density. The profiles illustrate the energy constraints for two-layer data. Theorem~\ref{thm:RT} assumes a smooth initial equilibrium.}
\label{fig:RT}
\end{figure}

\section{The linearized problem}\label{sec:linear}

\subsection{Proof of Theorem~\ref{thm:lin}(1)--(4)}

\emph{Functional setting.} We work over $\mathbb C$. Let $X\coloneqq H\times L^2$ with the energy inner product
\[
\langle(\bu,\varrho),(\bv,\sigma)\rangle_X\coloneqq\int\rho_s\bu\cdot\bar\bv+\int\delta^{-1}\varrho\bar\sigma .
\]
Let $\mathcal Q$ be the orthogonal projection of $L^2(\rho_s d\bx)$ onto $H$, so that $\int\rho_s\mathcal Q[\mathbf w]\cdot\bar\bv=\int\rho_s\mathbf w\cdot\bar\bv$ for $\bv\in H$. Write \eqref{lin-gen} as $\dot x=Ax$ with $A=A_0+B$, where
\[
A_0(\bu,\varrho)\coloneqq(\nu\mathcal Q[\rho_s^{-1}\Delta\bu],0),\qquad
B(\bu,\varrho)\coloneqq(-\mathcal Q[\rho_s^{-1}\varrho\nabla f],\ \delta\,\bu\cdot\nabla f),
\]
and $D(A)=D(A_0)=(H^2\cap H^1_0\cap H)\times L^2$.
\begin{itemize}
\item $A_0$ is the self-adjoint nonpositive operator associated with the closed coercive form $\nu\int\nabla\bu:\nabla\bar\bv$ on $V=H^1_0\cap H$. Its domain is identified by Lemma~\ref{lem:stokes}.
\item $B$ is bounded on $X$, and a direct computation gives $\langle B x,y\rangle_X=-\int\varrho\nabla f\cdot\bar\bv+\int(\bu\cdot\nabla f)\bar\sigma=-\langle x,By\rangle_X$, so $B$ is skew-adjoint.
\end{itemize}
Hence $A$ generates a $C_0$ contraction semigroup, $A^*=A_0-B$ with $D(A^*)=D(A)$, and
\begin{equation}\label{diss}
\mathrm{Re}\langle Ax,x\rangle_X=\mathrm{Re}\langle A^*x,x\rangle_X=-\nu\norm{\nabla\bu}^2 .
\end{equation}

\emph{Reducing subspaces.} Let $X_0\coloneqq\{0\}\times\Hy_L$ and $X_1\coloneqq X_0^\perp$. If $\sigma\nabla f=\nabla q$, then $\mathcal Q[\rho_s^{-1}\nabla q]=0$ and $B(0,\sigma)=0$, so $X_0\subset\ker A\cap\ker A^*$. Therefore $X_0$ and $X_1$ reduce $A$, and the semigroup acts as the identity on $X_0$. The $X_0$-component of $(\bu_0,\varrho_0)$ is $(0,\Pi_{\Hy}\varrho_0)$, because the orthogonality is taken in the $\delta^{-1}$-weighted inner product. Let $A_1$ be the part of $A$ in $X_1$.

\emph{Sectorial resolvent and proof of (1).}
Let $S$ be the positive weighted Stokes operator on $H$, so that $A_0=(-S,0)$ and $S\ge s_0 I$ for a constant $s_0>0$ by Poincar\'e's inequality. Write $C\bu=\delta\bu\cdot\nabla f$. The adjoint is taken between the weighted velocity and density spaces in $X$, and $A$ has block form
\[
 A=\begin{pmatrix}-S&-C^*\\ C&0\end{pmatrix}.
\]
For $z=a+i\omega$, $a\ge0$, $z\ne0$, elimination of the velocity in $(z-A)(\bu,\varrho)=(F,G)$ gives
\begin{equation}\label{resolvent}
 \big[z+C(z+S)^{-1}C^*\big]\varrho=G+C(z+S)^{-1}F,\qquad
 \bu=(z+S)^{-1}(F-C^*\varrho).
\end{equation}
Let $Q_z$ denote the square bracket. For $D_z=((S+a)^2+\omega^2)^{-1}$, the real and imaginary parts of $\langle Q_z\varrho,\varrho\rangle$ are
\begin{align*}
 a\norm\varrho^2+\langle (S+a)D_zC^*\varrho,C^*\varrho\rangle,\qquad
 \omega\big(\norm\varrho^2-\langle D_zC^*\varrho,C^*\varrho\rangle\big),
\end{align*}
where the norms and pairings here are weighted. Because $S+a\ge s_0 I$, the form $\langle D_z C^*\varrho,C^*\varrho\rangle$ is bounded by $s_0^{-1}$ times the real part. Thus, on every bounded subset of the closed right half-plane,
\[
 |z|\norm\varrho^2
 \le C\big|\langle Q_z\varrho,\varrho\rangle\big|.
\]
The estimate also holds for $Q_z^*=Q_{\bar z}$, so $Q_z$ is bijective with inverse bounded by $C/|z|$. For $|z|$ large, use $(z-A_0)^{-1}$, bounded by $1/|z|$, and the bounded perturbation $B$. We obtain
\begin{equation}\label{res-bound}
 \norm{(z-A)^{-1}}_{\mathcal L(X)}\le C|z|^{-1},\qquad
 \Re z\ge0,\quad z\ne0.
\end{equation}
The resolvent identity extends this bound to $|\arg z|<\pi/2+\eta$ for some $\eta>0$: at $i\omega$ a disk of radius $c|\omega|$ is in the resolvent set by a Neumann series. The sectorial generation criterion \cite[Theorem~II.4.6]{EngelNagel2000} yields a bounded analytic semigroup and the estimates
\begin{equation}\label{analytic-decay}
 \norm{A^j e^{tA}}_{\mathcal L(X)}\le C_j t^{-j},\qquad j\ge1,\ t>0.
\end{equation}
To see the density of $\operatorname{Ran}A_1$ in $X_1$, observe that $A^*(\bv,\sigma)=0$ gives $\bv=0$ by \eqref{diss}, and then $\sigma\nabla f\in G$. Thus $\ker A^*=X_0$ and $\ker A_1^*=\{0\}$. The estimate \eqref{analytic-decay} gives decay on $\operatorname{Ran}A_1$, and contractivity extends it to all of $X_1$. The $X_0$ component is preserved.

For the stronger convergence in (1), \eqref{analytic-decay} gives $\norm{\bu_t}\to0$. The Stokes estimate after absorbing the conserved hydrostatic force gives
\[
 \norm{\bu}_{H^2}+\norm{\nabla P+(\Pi_{\Hy}\varrho_0)\nabla f}
 \le C(\norm{\bu_t}+\norm{\varrho-\Pi_{\Hy}\varrho_0}),
\]
whose right-hand side tends to zero.

\emph{Proof of (2).} The eigenvalues $-a_n\to0$ in Theorem~\ref{thm:planar-spectrum} give $\norm{e^{tA_1}}=1$ by contractivity. The dissipation identity and \eqref{analytic-decay} imply $\nu\norm{\nabla\bu}^2\le Ct^{-1}\norm{x_0}_X^2$ and $\norm{\partial_t\varrho}\le Ct^{-1}\norm{x_0}_X$. For a normalized eigenvector, $\nu\norm{\nabla\bu_0}^2=a_n$ and its density norm is bounded below by Poincar\'e's inequality. Evaluation at $t=a_n^{-1}$ proves sharpness of both powers. For $t\ge1$ the Stokes estimate above, contractivity and \eqref{analytic-decay} give a uniform $H^2$ bound in the transient initial norm. Interpolating with \eqref{linear-uniform-velocity} proves \eqref{linear-velocity-interpolated}.

To verify its sharpness in the channel, take $a(k)=a_{k,1}$ from Theorem~\ref{thm:channel-spectrum}, with $a(k)\asymp |k|^{-2}$ as $|k|\to\infty$. Normalize its eigenvector to have unit energy norm. The dissipation identity gives $\nu\norm{\nabla\bu_0}^2=a(k)$, so $\norm{\varrho_0}$ is bounded away from zero by Poincar\'e's inequality. Density reconstruction gives $\norm{\bu_0}\ge c a(k)\norm{\varrho_0}$. Since the velocity is supported at horizontal frequency $k$, $\norm{\bu_0}_{H^r}\ge c|k|^r a(k)$. Its norm at time $t=a(k)^{-1}$ is thus at least $c t^{-(2-r)/2}$. Real initial data follow by taking the paired modes $k$ and $-k$.

\emph{Proof of (3).} The displacement data are $x_0=A(\bv,\sigma)\in X_1$. Contractivity and \eqref{analytic-decay} give the state estimate; applying \eqref{linear-velocity-interpolated} on the second half of the time interval gives \eqref{displacement-velocity}. For sharpness, take $(\bv,\sigma)$ to be the unit-energy eigenvectors from (2). Their generating norm is bounded by the Stokes estimate. Acting with $A$ contributes a factor $a(k)$, so the same test at $t=a(k)^{-1}$ gives the claimed exponent.

\emph{Proof of (4).} Since $\rho_s$ is hydrostatic, $\nabla\rho_s\parallel\be_2$ and $\rho_s=\rho_s(x_2)$, so $\delta=-\rho_s'/g$ depends only on $x_2$. By the argument of Theorem~\ref{thm:hydro}(5), $\Hy_L$ consists of the functions of $x_2$. Orthogonal projection onto this space with a weight depending only on $x_2$ is the average over horizontal sections. \qed

To justify Remark~\ref{rem:lin}, let $\sigma\nabla f=\nabla q$. In local regular-level coordinates the tangential derivative of $q$ vanishes, so $q$ depends locally only on the level coordinate $f$. Its derivative in that coordinate is $\sigma$. By slicing in a countable collection of such charts and using connectedness of the regular levels, these local functions agree along almost every level. The coarea formula and $\nabla f\ne0$ a.e. give $\sigma=\psi(f)$ a.e. Conversely, if $\sigma=\psi(f)\in L^2$, smooth functions of $f$ approximate $\sigma$ in $L^2$ for the pushforward measure under $f$. Their products with $\nabla f$ are gradients, and closedness of $G$ gives $\sigma\in\Hy_L$. The identity $\nabla\rho_s=-\delta\nabla f$ gives $\delta=-\varphi'(f)$ a.e., so the weight is constant along levels and the weighted projection equals the ordinary conditional expectation.

\subsection{Proof of Theorem~\ref{thm:lin}(5)}

Write $x(t)=A_1^2 e^{tA_1}(I-A_1)^{-2}z$. By \eqref{analytic-decay} for large time and graph-norm contractivity for bounded time,
$\norm{A^jx(t)}_X\le C(1+t)^{-2-j}\norm z_{D(A_1^2)}$ for $j=0,1,2$. Dissipation gives $\norm{\bu(t)}_{H^1}\lesssim(1+t)^{-5/2}$ and $\norm{\bu_t(t)}_{H^1}\lesssim(1+t)^{-7/2}$. Integrating $\varrho_t=\delta\bu\cdot\nabla f$ from $t$ to infinity, with zero transient limit, bounds $\norm{\varrho(t)}_{L^p}$ by $C_p(1+t)^{-3/2}\norm{z}_{D(A_1^2)}$. The $L^p$ Stokes estimate gives the same bound on $\bu$ in $W^{2,p}$. Taking $p>2$ and integrating proves the strain and displacement conclusions. \qed

\subsection{Proof of Theorem~\ref{thm:lin}(6)}

Integrate $\frac{d}{dt}\langle\varrho,\varphi\rangle_{\delta^{-1}}
=\langle\varphi\nabla f,\bu\rangle$ and apply Cauchy--Schwarz and the dissipation identity. This gives \eqref{slowmode}, first on $D(A)$ and then by density.

In the channel take $\varphi_n=a(x_1)b(x_2)\cos(nx_2)$ with smooth zero-mean $a$ and compactly supported $b$, both nonzero. These tests have zero hydrostatic projection, weighted $L^2$ norms bounded above and below, and $H^s$ norms at most $C_s n^s$. Write
\[
 \varphi_n\nabla f
  =\nabla\big(gab\sin(nx_2)/n\big)-g\sin(nx_2)\nabla(ab)/n.
\]
The gradient pairs to zero on $V$. Integrating the remaining term by parts in $x_2$ gives $\norm{\varphi_n\nabla f}_{V'}\le C n^{-2}$. With initial density $\varphi_n$ and zero velocity, \eqref{slowmode} thus bounds the change in its pairing with $\varphi_n$ by $C\sqrt t/n^2$. It retains at least half its initial value up to $t=c n^4$. Normalization in $H^s$ yields the asserted lower bound and excludes $o(t^{-s/4})$. \qed

\subsection{Proof of Theorem~\ref{thm:planar-spectrum}}

\emph{A compact pencil.}
Use the block representation of $A$ in the proof of Theorem~\ref{thm:lin}, and write $s_0>0$ for a lower bound of $S$. The Schur complement for $0<|\lambda|<s_0$ is $\lambda I+C(\lambda+S)^{-1}C^*$ on the density space orthogonal to $\Hy_L$. Since $(\lambda+S)^{-1}$ is compact, the Fredholm alternative implies that every spectral point in this punctured disk is an eigenvalue. The analytic Fredholm theorem gives finite algebraic multiplicity at these nonzero spectral points, since the resolvent exists for positive real $\lambda$. Eliminating the density, $\lambda\varrho=C\bu$, gives the Hermitian forms on $V$
\[
 d(\bu,\bv)=\nu\int\nabla\bu:\nabla\bar\bv,\quad
 m(\bu,\bv)=\int\rho_s\bu\cdot\bar\bv,\quad
 b(\bu,\bv)=\int\delta(\bu\cdot\nabla f)(\bar\bv\cdot\nabla f)
\]
and the pencil $\lambda^2m+\lambda d+b=0$. Conversely, a nonzero pencil eigenvector gives an eigenvector of $A_1$ by the Stokes estimate and the reconstruction of the density. If $\lambda$ is nonreal, the imaginary part on the eigenvector gives $2\Re\lambda=-d(\bu,\bu)/m(\bu,\bu)\le-s_0$. Dissipativity excludes positive real eigenvalues. Thus close to zero all the nonzero eigenvalues are negative real.

For $\lambda=-a<0$ let $K_a$ be the compact operator associated with $b+a^2m$ in the $d$ inner product. Let $\kappa_n(a)$ be its positive eigenvalues counted with multiplicities in decreasing order, and set $\mu_n=\kappa_n(0)$. The positive spectrum of $K_0$ is infinite, as shown below. With $M$ the operator associated with $m$, the min--max principle gives
\begin{equation}\label{general-minmax}
 \mu_n\le\kappa_n(a)\le\mu_n+a^2\norm M,\qquad
 |\kappa_n(a)-\kappa_n(a')|\le\norm M\,|a^2-(a')^2|.
\end{equation}
On $0<a<r_0$ with $2r_0\norm M<1$, each function $\kappa_n(a)-a$ is strictly decreasing. For large $n$ it has one root $a_n\in[\mu_n,2\mu_n]$. Each eigenvalue with $0<a<r_0$ is a root for some $n$. Compactness of $K_{r_0}$ bounds the number of roots in any subinterval separated from zero. The root multiplicity in the ordered list agrees with the dimension of the pencil kernel. In addition, let $J(\bu,\varrho)=(\bu,-\varrho)$ so that $A_1^*J=JA_1$. On a nonzero eigenvector $x=(\bu,\varrho)$ at $-a$,
\[
 a\norm{x}_X^2=d(\bu,\bu)\ge s_0\norm{\bu}_{L^2(\rho_s\,d\bx)}^2 .
\]
For $a<s_0/2$ this gives $\langle x,Jx\rangle_X<0$. A generalized eigenvector $y$ solving $(A_1+a)y=x$ would give $\langle x,Jx\rangle_X=\langle (A_1+a)y,Jx\rangle_X=0$. Hence these eigenvalues are semisimple. For the upper bound in \eqref{pencil-root-bounds}, use the first positive root of $\mu_n+a^2\norm M=a$. At this root $\kappa_n(a)-a\le0$ by \eqref{general-minmax}, hence $a_n$ is no larger. \autoref{fig:pencil-root} shows the normalized bracketing: the scalar envelopes restrict the crossing to the interval between $\mu_n$ and this root. For the lower bound use $b\le K m$ with $K=\delta_1\norm{\nabla f}_{L^\infty}^2/\inf\rho_s>0$. Min--max gives $\kappa_n(a)\ge(1+a^2/K)\mu_n$, hence $a_n-\mu_n\ge K^{-1}\mu_n a_n^2\ge K^{-1}\mu_n^3$.

\begin{figure}[tbp]
\centering
\begin{tikzpicture}
\begin{axis}[
  width=.57\textwidth,height=5.7cm,
  xmin=.70,xmax=1.93,ymin=.78,ymax=1.85,
  xlabel={$r=a/\mu_n$},ylabel={normalized level},
  xtick={1,1.333333},xticklabels={$1$,$r_*$},
  ytick={1},yticklabels={$1$},
  tick label style={font=\small},
  label style={font=\small},
  axis on top,
]
\fill[blue!9] (axis cs:1,.78) rectangle (axis cs:1.333333,1.85);
\addplot[gray!67,very thin,dotted] coordinates {(1,.78) (1,1)};
\addplot[gray!67,very thin,dotted] coordinates {(1.333333,.78) (1.333333,1.333333)};
\addplot[gray!80!black,thick,domain=.78:1.85,samples=2] {x};
\addplot[teal!65!black,dashed,thick,domain=.70:1.91,samples=2] {1};
\addplot[blue!70!black,thick,domain=.70:1.91,samples=46] {1+3*x*x/16};
\addplot[only marks,mark=*,mark size=1.3pt,black] coordinates {(1,1) (1.333333,1.333333)};
\node[font=\footnotesize,gray!75!black,anchor=south west] at (axis cs:1.64,1.65) {$r$};
\node[font=\footnotesize,blue!70!black,anchor=south] at (axis cs:.99,1.22) {$1+\mu_n\norm M,r^2$};
\node[font=\footnotesize,teal!65!black,anchor=south] at (axis cs:1.74,1.01) {$1$};
\end{axis}
\end{tikzpicture}
\caption{Normalized bounds on the pencil crossing in \eqref{pencil-root-bounds}.
The picture takes $\mu_n\norm M=3/16$ to show the two scalar envelopes.
Their crossings with the diagonal enclose the true crossing in the shaded
interval, whose upper endpoint $r_*$ is the smaller root of
$r=1+\mu_n\norm M\,r^2$.}
\label{fig:pencil-root}
\end{figure}

\emph{Order of the count.}
For the upper bound, extend $V$ by zero to a square containing $\Omega$, and impose orthogonality of both velocity components to the first $j$ Dirichlet modes of the square. The codimension is at most $2j$. Counting the square's sine modes gives $\norm{\bu}^2\le C j^{-1}d(\bu,\bu)$ on this subspace, and $b\le C\norm{\bu}^2$. The same proof in a periodic channel uses horizontal Fourier modes and vertical sine modes. For the lower bound, choose a ball where $\nabla f$ is close to one nonzero vector, and choose a smooth divergence-free bump whose component along that vector has nonzero $L^2$ norm. Pack $n$ disjoint rescaled copies of this bump in the ball, at scale comparable to $n^{-1/2}$. The quotient $b/d$ on each copy is bounded below by $c/n$ by scaling and $\delta\ge\delta_0$. Disjoint supports give the same lower bound on their span. The min--max principle yields $\mu_n\asymp n^{-1}$.

\emph{Leading coefficient.}
In a bounded planar domain, extend $\bu\in V$ by zero to the plane and use a stream function $\bu=\nabla^\perp\psi$. It is constant on each component of the complement, with the outer constant set to zero. Imposing zero constants in any holes costs at most one linear constraint per hole. In the periodic channel a periodic stream function exists because $\bar u_2=0$. Set its bottom-wall value to zero and impose zero horizontal flux to make its top-wall value zero, at a cost of at most one constraint. In either case the resulting velocity subspace is represented by $\psi\in H^2_0(\Omega)$, with the horizontal variable periodic in the channel. Such a finite codimension changes the eigenvalue count by a bounded number. On this space the forms are
\[
 d=\nu\int|\Delta\psi|^2,\qquad
 b=\int (\nabla\psi)^*\mathsf B(x)\nabla\psi,\qquad
 \mathsf B=\delta\,\nabla^\perp f\otimes\nabla^\perp f.
\]
Here the identity for $d$ follows by integration by parts with clamped traces. We prove the count directly by local Fourier comparison.

For a square $Q$ and a constant real symmetric matrix $B_0$, denote by $N_Q(\tau;B_0)$ the maximal dimension of a subspace of $H^2_0(Q)$ on which
$\int(\nabla\psi)^*B_0\nabla\psi>\tau\nu\int|\Delta\psi|^2$ for nonzero $\psi$. Embed the clamped functions into the periodic square. On that square the positive directions of the form difference are exactly the Fourier modes with
$\xi^TB_0\xi>\tau\nu|\xi|^4$. They all have $|\xi|\le C\tau^{-1/2}$. Within their span, imposing vanishing value and normal derivative on the four edges costs $O(\tau^{-1/2})$ constraints, since each edge trace is a one-dimensional trigonometric polynomial of degree $O(\tau^{-1/2})$. This bounds the clamped count from below by the periodic count minus $O(\tau^{-1/2})$. The embedding gives the upper bound. The rank-one region in \autoref{fig:spectral-count} is the case $B_0=\delta\nabla^\perp f\otimes\nabla^\perp f$; its area gives the leading local term. Lattice Riemann sums on the rescaled region yield
\begin{equation}\label{local-weyl-count}
 \lim_{\tau\downarrow0}\tau N_Q(\tau;B_0)
 =\frac{|Q|}{8\pi^2\nu}
     \int_{\mathbb S^1}(\theta^TB_0\theta)_+\,d\theta .
\end{equation}
This formula also holds when $B_0$ has a negative eigenvalue. The radial extent of the rescaled region is $\sqrt{(\theta^TB_0\theta)_+/\nu}$, which gives the angular integral.

For the lower bound in $\Omega$, use disjoint grid squares contained in $\Omega$, extend their clamped functions by zero, and on each square replace $\mathsf B$ by its value at a sample point minus its oscillation bound times the identity. Apply \eqref{local-weyl-count} to this subspace, using coordinate squares on the horizontal cylinder in the periodic case. For the upper bound, enlarge the grid squares meeting $\Omega$ by a fixed relative factor $1+\epsilon$ and use a smooth partition with $\sum_i\chi_i^2=1$ on $\overline\Omega$, subordinate to these enlarged squares. Extend $\psi$ by zero. The localized functions $\chi_i\psi$ are clamped in the corresponding squares and their numerator sum is at least $b(\psi,\psi)$, since the cross terms cancel and $\mathsf B$ is nonnegative. The denominator sum is at most
\[
 (1+\eta)d(\psi,\psi)+C_{\eta,\chi}\norm{\psi}_{H^1}^2
\]
by the product rule. By compactness we may impose finitely many linear constraints so the second term is bounded by $\eta d(\psi,\psi)$. Replace each local matrix by its sampled value plus its oscillation bound times the identity. The injection $\psi\mapsto(\chi_i\psi)_i$ then bounds the count from above by the sum of the local counts at $\tau/(1+2\eta)$, up to the finite codimension. This argument uses continuous samples only where the original domain intersects the square, since the localized functions vanish elsewhere.

Take $\tau\downarrow0$ for each fixed grid, then let the mesh tend to zero and $\epsilon,\eta\downarrow0$. Uniform continuity of $\mathsf B$ and the vanishing area of the boundary strip give the same limit for the upper and lower bounds. On the cylinder the comparison is local in periodic coordinate charts, so no extra boundary strip occurs at a choice of horizontal seam. Since $\int_{\mathbb S^1}\theta^T\mathsf B\theta\,d\theta=\pi\operatorname{tr}\mathsf B$, we obtain
\[
 \lim_{\tau\downarrow0}\tau\#\{n:\mu_n>\tau\}
 =\frac1{8\pi\nu}\int_\Omega\operatorname{tr}\mathsf B\,\dx .
\]
This constant is positive by the nondegeneracy assumption on $\nabla f$. Inverting the count and using $a_n/\mu_n\to1$ proves \eqref{planar-weyl}. \qed

\subsection{Proof of Theorem~\ref{thm:channel-spectrum}}\label{sec:channel-spectrum}

The clamped biharmonic operator also governs the linear Stokes--transport spectrum for an affine background \cite[Lemma~B.2]{DGL2025}. Here the inertial term adds a quadratic term to the eigenvalue equation. Fix $k\ne0$. A solenoidal no-slip velocity in this horizontal mode has the form $e^{ikx_1}(\psi',-ik\psi)$ with $\psi\in H^2_0(0,h)$. Define the Hermitian forms
\begin{align*}
 d(\psi,\chi)&=\nu\int_0^h(\psi''-k^2\psi)\overline{(\chi''-k^2\chi)}\,dx_2,\\
 b(\psi,\chi)&=g^2k^2\int_0^h\delta\,\psi\bar\chi\,dx_2,\\
 m(\psi,\chi)&=\int_0^h\rho_s\big(\psi'\bar\chi'+k^2\psi\bar\chi\big)\,dx_2 .
\end{align*}
For fixed $k$, $d$ is equivalent to the squared $H^2$ norm. Density reconstruction and the Stokes estimate identify the eigenvectors with the pencil $\lambda^2m+\lambda d+b=0$, as in Theorem~\ref{thm:planar-spectrum}. For $\lambda=-a$, let $\kappa_n(a)$ be the ordered eigenvalues of $b+a^2m$ relative to $d$, so $\kappa_n(0)=\mu_{k,n}$.

To estimate $\mu_{k,n}$, first replace the weight $\delta$ by a constant on an interval of length $\ell$. On $H^2\cap H^1_0$ of that interval, the sine basis diagonalizes the quotient $b/d$, with eigenvalues
\[
 \frac{g^2 k^2\delta}{\nu\big(k^2+(\pi j/\ell)^2\big)^2},
 \qquad j\ge1.
\]
The subspace $H^2_0$ has codimension two in $H^2\cap H^1_0$. The min--max principle therefore bounds its $j$-th eigenvalue between the $(j+2)$-nd and the $j$-th eigenvalues in the sine basis. Comparison with $\delta_0\le\delta\le\delta_1$ gives the stated two-sided bound for $\mu_{k,n}$, uniformly in both indices.

The bounds on $\rho_s$ and $\delta$ give $m\ge c b$. Integration by parts and Cauchy--Schwarz give
$m\le C\norm\psi\norm{\psi''-k^2\psi}\le C|k|^{-1}\sqrt{bd}$ on the scalar quadratic forms. For the upper bound restrict to the $d$-orthogonal complement of the first $n-1$ buoyancy eigenvectors, where $b\le\mu_{k,n}d$. Min--max then gives
\[
 (1+c a^2)\mu_{k,n}\le\kappa_n(a)\le
 \mu_{k,n}+C a^2 |k|^{-1}\sqrt{\mu_{k,n}}.
\]
For sufficiently large $k^2+n^2$ this upper bound is less than $2\mu_{k,n}$ at $a=2\mu_{k,n}$. Continuity gives a root in $[\mu_{k,n},2\mu_{k,n}]$. For uniqueness near zero, use $m\le C k^{-2}d$ from the expansion of $d$ and the min--max Lipschitz bound in \eqref{general-minmax}. Since $|k|\ge2\pi/L$, $\kappa_n(a)-a$ is strictly decreasing on a fixed small interval uniformly in $k,n$. At the root, $a_{k,n}\asymp\mu_{k,n}$, so the bounds on $\kappa_n(a)$ give \eqref{channel-uniform-correction}.

For continuous $\delta$, both counting limits follow from the same one-dimensional comparison. Let $N_k(\tau)=\#\{n:\mu_{k,n}>\tau\}$. Partition $(0,h)$ into $J$ intervals $I_i$ of lengths $\ell_i$. Write $\delta_i^\pm$ for the extrema of $\delta$ on each interval and set
\[
 p_k(d_0,\tau)=\left(\sqrt{\frac{g^2k^2d_0}{\nu\tau}}-k^2\right)_+^{1/2}.
\]
For constant weight $d_0$, the integer indices in the sine count are those with $\pi j/\ell_i<p_k(d_0,\tau)$, so rounding changes the count by at most one. For the lower comparison, extend the clamped functions on each $I_i$ by zero and use $\delta_i^-$. The two clamped trace constraints cost at most two modes per interval. For the upper comparison, impose vanishing at the internal partition points and use the direct sum of the sine spaces with weights $\delta_i^+$. Min--max gives
\[
 \frac1\pi\sum_i\ell_i p_k(\delta_i^-,\tau)-3J
 \ \le\ N_k(\tau)\ \le\
 \frac1\pi\sum_i\ell_i p_k(\delta_i^+,\tau)+J .
\]
For fixed $k$, multiply by $\tau^{1/4}$ and let $\tau\downarrow0$, then refine the partition. This gives the coefficient $\pi^{-1}(g^2k^2/\nu)^{1/4}\int_0^h\delta^{1/4}$ for $\tau^{1/4}N_k(\tau)$. Inverting the count and using $a_{k,n}/\mu_{k,n}\to1$ gives \eqref{channel-weyl}.

For the joint limit, set $\tau=g^2q/(\nu k^2)$ and divide the same inequalities by $|k|$. Now $p_k(d_0,\tau)/|k|=(\sqrt{d_0/q}-1)_+^{1/2}$. First sending $|k|\to\infty$ and then refining the partition gives $\mathcal N(q)$ for the buoyancy count. This integrand is the half-width of the vertical frequency slice in \autoref{fig:spectral-count}, in units of $|k|$. The normalized inertial correction in \eqref{channel-uniform-correction} is uniformly $O(|k|^{-4})$, so continuity in $q$ transfers the limit to the count of $a_{k,n}$. Finally, $\mathcal N$ strictly decreases on $(0,D)$ since the integrand strictly decreases on $\{\delta>q\}$, which has positive measure for $q<D$. Its endpoint limits are infinity and zero. Comparing the counts on either side of $\mathcal N^{-1}(\gamma)$ gives the indexed limit in \eqref{channel-joint-count}.

For the large horizontal frequencies, put $\epsilon=|k|^{-1}$. Integration by parts in $d$ identifies $(\nu/g^2) k^2\mu_{k,n}$ with the $n$-th min--max value of
\[
 Q_\epsilon(\psi)=\frac{\int_0^h\delta|\psi|^2}
 {\norm\psi^2+2\epsilon^2\norm{\psi'}^2+\epsilon^4\norm{\psi''}^2},
 \qquad \psi\in H^2_0(0,h)\setminus\{0\}.
\]
It is at most $D$. For fixed $n$ and $\eta>0$, choose $n$ orthonormal $L^2$ functions supported on the positive-measure set $\{\delta>D-\eta\}$. Approximate them in $L^2$ by smooth compactly supported functions. On their span the quotient with denominator $\norm\psi^2$ is then at least $D-2\eta$, by boundedness of $\delta$ and continuity on the finite-dimensional unit sphere. On that fixed span the extra denominator terms tend uniformly to zero. Min--max and $\eta\downarrow0$ give the limit in \eqref{large-horizontal-branch} since the inertial correction in \eqref{channel-uniform-correction} is $O(|k|^{-6})$.

For the second term, write $V=D-\delta$, $b_i=-\delta''(z_i)/2$, and denote the denominator of $Q_\epsilon$ by $B_\epsilon(\psi)$. The deficit
$\Lambda_{\epsilon,n}=D-(\nu/g^2)k^2\mu_{k,n}$ is the $n$-th increasing min--max value of $F_\epsilon/B_\epsilon$, where
\[
 F_\epsilon(\psi)=\int_0^h V|\psi|^2
 +2D\epsilon^2\norm{\psi'}^2+D\epsilon^4\norm{\psi''}^2 .
\]
At $z_i$, the scale $y=(z-z_i)/\sqrt\epsilon$ gives the oscillator $-2D\partial_y^2+b_i y^2$. Dilation reduces it to $\sqrt{2Db_i}(-\partial_x^2+x^2)$. Write $\mathfrak a=\partial_x+x$. Then $-\partial_x^2+x^2=\mathfrak a^*\mathfrak a+1$, $[\mathfrak a,\mathfrak a^*]=2$, and $\mathfrak a e^{-x^2/2}=0$. Applying $\mathfrak a^*$ successively gives eigenvalues $2j-1$ and eigenfunctions that are polynomials of degree $j-1$ times $e^{-x^2/2}$. These functions are complete: if $v\in L^2(\R)$ is orthogonal to all of them, the entire function $\int_{\R}v(x)e^{-x^2/2}e^{\zeta x}\,dx$ has every derivative zero at $\zeta=0$. Restricting to the imaginary axis and using Fourier uniqueness gives $v=0$. For the normalized eigenfunctions $h_j$, the ladder identity gives
$\langle\mathfrak a v,h_{j-1}\rangle=\sqrt{2j-2}\langle v,h_j\rangle$ for $j\ge2$ by integration by parts in the form domain. Parseval then gives
$\norm{\mathfrak a v}^2+\norm v^2=\sum_{j\ge1}(2j-1)|\langle v,h_j\rangle|^2$.
Thus $e_n$ is the $n$-th eigenvalue of the direct sum of the oscillators, with the corresponding lower bound on the form restricted to the orthogonal complement of its first $n-1$ eigenfunctions.

For the upper bound, take the first $n$ normalized eigenfunctions of this direct sum. For a component $w$ at $z_i$, use the trial function
$\epsilon^{-1/4}\chi_i(z)w((z-z_i)/\sqrt\epsilon)$, with smooth cutoffs equal to one near the respective maxima and with disjoint supports in $(0,h)$. In this basis the matrices of $B_\epsilon$ and $F_\epsilon$ are
$I+O(\epsilon)$ and $\epsilon\,\operatorname{diag}(e_1,\ldots,e_n)+o(\epsilon)$. Indeed, Taylor's formula and the Gaussian decay of the Hermite functions give convergence of the rescaled potential terms, and the term with two derivatives contributes $O(\epsilon^2)$. Min--max gives $\limsup_{\epsilon\downarrow0}\Lambda_{\epsilon,n}/\epsilon\le e_n$.

For the lower bound, fix $0<\eta<\min_i b_i$ and let $e_n^\eta$ be the corresponding value with $b_i$ replaced by $b_i-\eta$. Choose a smooth partition $\sum_{i=0}^M\chi_i^2=1$ with each $\chi_i$ for $i\ge1$ supported near $z_i$ where $V\ge(b_i-\eta)(z-z_i)^2$. Choose $\chi_0$ supported away from $Z$, where $V\ge c_\eta>0$. With $G_\epsilon(\psi)=\int_0^h V|\psi|^2+2D\epsilon^2\norm{\psi'}^2$, expansion of the localized first derivatives gives
\[
 \sum_{i=0}^M G_\epsilon(\chi_i\psi)
 =G_\epsilon(\psi)+2D\epsilon^2\sum_{i=0}^M\norm{\chi_i'\psi}^2 .
\]
In the span of the first $n$ buoyancy eigenfunctions choose a nonzero $\psi$ such that the rescaled $\chi_i\psi$, extended by zero, are collectively orthogonal to the first $n-1$ eigenfunctions of the direct sum with coefficients $b_i-\eta$. These are at most $n-1$ linear conditions. The local oscillator forms for $i\ge1$ and the gap on $\operatorname{supp}\chi_0$ then give
\[
 F_\epsilon(\psi)\ge G_\epsilon(\psi)
 \ge(\epsilon e_n^\eta-C_\eta\epsilon^2)\norm\psi^2
\]
for sufficiently small $\epsilon$. On the chosen span, $F_\epsilon\le\Lambda_{\epsilon,n}B_\epsilon\le C_n\epsilon B_\epsilon$. Since $D(B_\epsilon-\norm\psi^2)\le F_\epsilon$, the last display gives $\liminf_{\epsilon\downarrow0}\Lambda_{\epsilon,n}/\epsilon\ge e_n^\eta$. Letting $\eta\downarrow0$ gives the limit $e_n$, including coincident levels. The inertial correction changes $k^2\mu_{k,n}$ by $O(|k|^{-4})$, which proves \eqref{large-horizontal-second}.

Finally, normalize a true pencil eigenfunction by $B_\epsilon(\psi_{k,n})=1$. The identity $b/d=a_{k,n}-a_{k,n}^2m/d$ and $m/d=O(k^{-2})$ give
$Q_\epsilon(\psi_{k,n})=(\nu/g^2)k^2a_{k,n}+O(|k|^{-4})$.
Consequently $F_\epsilon=O(\epsilon)$ and $\norm{\psi_{k,n}}^2=1-O(\epsilon)$. Since $V(z)\ge c\operatorname{dist}(z,Z)^2$ by nondegeneracy and compactness, renormalizing in $L^2$ proves the concentration estimate.

\emph{A maximum at a wall.}
By reflection it suffices to take $z_*=0$. Set $r_\epsilon=\epsilon^{2/3}$ and consider on the half-line the Dirichlet form
\[
 \mathcal A_b(w)=2D\int_0^\infty |w'|^2\,dy+b\int_0^\infty y|w|^2\,dy,
 \qquad w\in H^1_0(0,\infty),\quad \sqrt y\,w\in L^2(0,\infty).
\]
Its embedding in $L^2$ is compact by local Rellich compactness and the tail bound from the linear potential. The eigenvalue equation reduces to the Airy equation with argument $(b/(2D))^{1/3}y-\lambda/((2D)^{1/3}b^{2/3})$. The square-integrable solution is proportional to $\operatorname{Ai}$, and the Dirichlet condition selects $\lambda_n(b)=(2D)^{1/3}b^{2/3}\zeta_n$ \cite[Sections~9.2, 9.7 and 9.9]{DLMFAiry}.

For the upper bound, $C_c^\infty(0,\infty)$ is dense in the form domain by truncation at infinity and $H^1_0$ approximation on bounded intervals. For any $\vartheta>0$, approximate the first $n$ eigenfunctions there to obtain an $n$-dimensional space on which $\mathcal A_b/\norm w^2\le\lambda_n(b)+\vartheta$. For $\psi(z)=r_\epsilon^{-1/2}w(z/r_\epsilon)$ on this trial space, both clamped traces hold for small $\epsilon$. On the fixed trial space the forms satisfy
\[
 r_\epsilon^{-1}F_\epsilon(\psi)=\mathcal A_b(w)+o(1)\norm w^2,\qquad
 B_\epsilon(\psi)=(1+O(r_\epsilon))\norm w^2.
\]
Indeed the potential divided by $r_\epsilon$ converges uniformly to $by$ on the common compact support, the first derivative contributes at order $r_\epsilon$, and the second derivative at order $r_\epsilon^2$. Min--max and $\vartheta\downarrow0$ give the upper bound on $\Lambda_{\epsilon,n}/r_\epsilon$.

For the lower bound, fix $0<\eta<b$ and choose $\chi_0,\chi_1$ with $\chi_0^2+\chi_1^2=1$, where $\chi_1$ is supported close to the wall and equals one near it. Take the support small enough that $V(z)\ge(b-\eta)z$ there. On the support of $\chi_0$, $V$ is bounded below by a positive constant. In the span of the first $n$ buoyancy eigenfunctions choose a nonzero $\psi$ so that $r_\epsilon^{1/2}(\chi_1\psi)(r_\epsilon y)$, extended by zero to the half-line, is orthogonal to the first $n-1$ eigenfunctions of $\mathcal A_{b-\eta}$. The extension has Dirichlet trace. Localizing $G_\epsilon$ as above and dropping the nonnegative fourth-order term gives, for small $\epsilon$,
\[
 F_\epsilon(\psi)\ge G_\epsilon(\psi)
 \ge \big(r_\epsilon\lambda_n(b-\eta)-C_\eta\epsilon^2\big)\norm\psi^2 .
\]
Here the local Dirichlet contribution is at least $r_\epsilon\lambda_n(b-\eta)\norm{\chi_1\psi}^2$, and the positive gap away from the wall gives the same bound with $\chi_0$. On the chosen span $F_\epsilon\le C_n r_\epsilon B_\epsilon$, hence $\norm\psi^2/B_\epsilon(\psi)=1-O(r_\epsilon)$. Letting $\epsilon\downarrow0$ and then $\eta\downarrow0$ proves $\Lambda_{\epsilon,n}/r_\epsilon\to\lambda_n(b)$. The normalized inertia error is $O(|k|^{-4})$, giving \eqref{large-horizontal-airy}. A true pencil eigenfunction normalized by $B_\epsilon=1$ has $F_\epsilon=O(r_\epsilon)$ by the same identity for $Q_\epsilon$ as in the interior case. Since $V(z)\ge c z$ on $[0,h]$, and $\norm\psi^2=1-O(r_\epsilon)$, this gives \eqref{wall-peak-localization}.

\emph{Inertia at a unique maximum.}
In either case with a unique maximum $z_*$, set $a=a_{k,n}$. Let $U_n$ and $W_n$ be the spans of the first $n$ eigenvectors of $b$ and $b+a^2m$ relative to $d$, respectively. On $U_n$, $b/d\ge\mu_{k,n}$. On $W_n$, $b/d\ge a-a^2m/d\ge\mu_{k,n}-O(|k|^{-6})$. Thus $F_\epsilon=o(1)$ uniformly on both spans with $B_\epsilon(\psi)=1$. In particular $\norm\psi^2=1-o(1)$ and $\epsilon^2\norm{\psi'}^2=o(1)$. Since $\int V|\psi|^2=o(1)$ and $V$ is bounded below away from every neighborhood of its unique zero $z_*$, $|\psi|^2\,dz$ concentrates at $z_*$ uniformly on both spans. By continuity of $\rho_s$,
\[
 |k|^2\frac{m}{d}
 =\frac1\nu\int_0^h\rho_s\big(|\psi|^2+\epsilon^2|\psi'|^2\big)
 =\frac{\rho_s(z_*)}{\nu}+o(1)
\]
uniformly on the two normalized spans. Using $U_n$ in the max--min formula for $b+a^2m$ gives a lower bound on $a-\mu_{k,n}$. For the upper bound use a nonzero vector of $W_n$ that is $d$-orthogonal to the first $n-1$ buoyancy eigenvectors, so that $b/d\le\mu_{k,n}$. Together the bounds give
\[
 a-\mu_{k,n}= a^2 |k|^{-2}\big(\rho_s(z_*)/\nu+o(1)\big).
\]
Equation~\eqref{large-horizontal-branch} yields \eqref{channel-inertial-limit}. \qed

\subsection{Proof of Theorem~\ref{thm:channel-traces}}

\emph{Compatible data.}
The horizontal mean of the density is stationary and the mean velocity is a decaying shear that solves $\rho_s\partial_t\bar u_1=\nu\partial_2^2\bar u_1$. The one-dimensional weighted heat semigroup yields exponential decay of $\bar u_1$ in $H^2$ for $t\ge1$. Subtract these means and suppose the data have zero horizontal mean. Write $\partial_1^{-1}$ for the zero-mean horizontal antiderivative. If $w=w_0$ is clamped, the field
\[
 v_2=w,\qquad v_1=-\partial_1^{-1}\partial_2w
\]
is solenoidal and no-slip. Set $F=\nu\Delta\bv-\rho_s\bu_0$, $q=\partial_1^{-1}F_1$ and $\sigma=(F_2-\partial_2q)/g$. Then
\[
 F=\nabla q+g\sigma\be_2,\qquad
 \norm{\bv}_{H^2}+\norm\sigma\le C B_0 .
\]
In particular $x_0=A(\bv,\sigma)$. For $t\ge1$, \eqref{analytic-decay} implies $\norm{x(t)}_X\le CB_0t^{-1}$ and $\norm{Ax(t)}_X\le CB_0t^{-2}$. Formula \eqref{diss} gives $\norm{\bu(t)}_{H^1}\le CB_0 t^{-3/2}$, and the Stokes estimate applied to \eqref{lin-gen} gives the $H^2$ bound. For optimality select the fixed nonzero horizontal frequency eigenvectors of Theorem~\ref{thm:channel-spectrum} with eigenvalues $-a_n$ and normalize the density component to have $L^2$ norm one. Write $\psi_n$ for their stream function. Density reconstruction gives $\norm{\psi_n}\asymp a_n$. The pencil at $-a_n$ gives
$d(\psi_n,\psi_n)=a_n m(\psi_n,\psi_n)+a_n^{-1}b(\psi_n,\psi_n)$, hence $\norm{\psi_n}_{H^2}\le C a_n^{-1/2}\norm{\psi_n}$. Its one-dimensional equation is
\[
 \nu(\partial_2^2-k^2)^2\psi_n
 =a_n\big(-\partial_2(\rho_s\psi_n')+k^2\rho_s\psi_n\big)
  +a_n^{-1} g^2k^2\delta\,\psi_n .
\]
It gives $\norm{\psi_n}_{H^4}\le C a_n^{-1}\norm{\psi_n}$. Thus $w_0=u_{0,2}/(-a_n)$ has the clamped traces and $H^4$ norm at most $C/a_n$. The dissipation identity gives $\norm{\bu_0}_{H^1}\asymp a_n^{1/2}$, since the density norm is one and the velocity norm is bounded. Scale the entire eigenvector by $a_n$, so that $B_0$ stays bounded. At $t=a_n^{-1}$ the $L^2$ density norm and the $H^1$ velocity norm are bounded below by positive constant multiples of $a_n$ and $a_n^{3/2}$, respectively. This proves sharpness of both decay exponents.

\emph{Upper bounds for incompatible traces.}
Let $j=0$ if the trace of $w_0$ is nonzero and $j=1$ if it vanishes on both walls. Let $w_\ell$ be $w_0$ times a smooth vertical cutoff that vanishes within distance $\ell$ of either wall and is one at distance at least $2\ell$. The datum $g\delta w_\ell$ has zero horizontal mean and is compatible. By the $H^4$ bound, a strip of width $2\ell$ gives
\[
 \norm{w_0-w_\ell}\le C B_0 \ell^{j+1/2},\qquad
 \norm{w_\ell}_{H^4}\le C B_0\ell^{j-7/2}.
\]
For $j=1$, these estimates use $|w_0|\le C B_0\ell$ and bounded first derivatives in the strip. For $j=0$ it suffices to bound $w_0$ there. The remaining derivative terms follow by Leibniz's rule and the $H^4$ bound. Applying the compatible $X$-norm estimate to $(\bu_0,g\delta w_\ell)$ and contractivity to the difference, and choosing $\ell$ proportional to $(1+t)^{-1/4}$, gives
\begin{equation}\label{linear-trace-upper}
 \norm{x(t)}_X\le C B_0(1+t)^{-\alpha_j},\qquad
 \alpha_j=\frac{2j+1}{8},
\end{equation}
on the zero-horizontal-mean component.

\emph{Lower bounds for incompatible traces.}
Choose a wall where the first nonvanishing trace has order $j$ and write $r$ for distance from that wall. Choose a real horizontal harmonic $b$ of nonzero frequency $k$ such that the pairing of $b$ with $\partial_r^j w_0$ on the chosen wall is nonzero. Take $\chi(r)=r^{3-j}$ near $0$, smoothly cut off in $[0,1)$. Define
\begin{align*}
 \psi_\ell(r)&=\ell^4\chi(r/\ell),&
 \Psi_\ell(x_1,x_2)&=b'(x_1)\psi_\ell(r),\\
 \bv_\ell&=(\partial_2\Psi_\ell,-\partial_1\Psi_\ell),&
 \phi_\ell&=\frac{\nu}{g}b(x_1)(\partial_r^2-k^2)^2\psi_\ell(r).
\end{align*}
The velocity test is clamped in stream-function coordinates. The field $F=-\nu\Delta\bv_\ell-g\phi_\ell\be_2$ has zero curl by the definition of $\phi_\ell$. Its horizontal mean is zero, so $\partial_1^{-1}F_1$ is a periodic potential with gradient $F$. Consequently $S\bv_\ell=C^*\phi_\ell$, and for $Y_\ell=(\bv_\ell,\phi_\ell)$ we have $A^*Y_\ell=(0,-C\bv_\ell)$. Direct scaling gives
\[
 \norm{\bv_\ell}\le C\ell^{7/2},\qquad
 \norm{\phi_\ell}\le C\ell^{1/2},\qquad
 \norm{C\bv_\ell}\le C\ell^{9/2}.
\]
Write $W_b(r)=\int_{\T_L} b(x_1)w_0(x_1,x_2)\,dx_1$ in the wall coordinate. In the energy pairing with $\phi_\ell$ the weight $\delta^{-1}$ cancels the factor $\delta$ in the initial density fluctuation. Integration by parts, using the two clamped traces of $\psi_\ell$, gives
\[
 \langle \varrho_0-\bar\varrho_0,\phi_\ell\rangle_{\delta^{-1}}
 = \nu\left[-W_b(0)\psi_\ell^{(3)}(0)+W_b'(0)\psi_\ell''(0)
       +\int_0^\ell (\partial_r^2-k^2)^2 W_b\,\psi_\ell\,dr\right].
\]
The last term is $O(\ell^{9/2})$ by $W_b\in H^4$ and $\norm{\psi_\ell}_{L^2}=O(\ell^{9/2})$. Near $0$, $\psi_\ell(r)=\ell^{j+1}r^{3-j}$, so the only nonzero endpoint contribution is a nonzero multiple of $\ell^{j+1}W_b^{(j)}(0)$. The initial velocity pairing is $O(\ell^{7/2})$. Thus $|\langle x_0,Y_\ell\rangle_X|\ge c_0\ell^{j+1}$ for small $\ell$.

For every $t\ge0$, duality and \eqref{linear-trace-upper} give
\[
 |\langle x(t)-x_0,Y_\ell\rangle_X|
 \le C B_0\ell^{9/2}\int_0^t(1+\tau)^{-\alpha_j}\,d\tau
 \le C B_0\ell^{9/2}(1+t)^{1-\alpha_j}.
\]
Take $\ell=\eta(1+t)^{-1/4}$, with $\eta>0$ sufficiently small depending on this datum. The upper bound is at most $(c_0/2)\ell^{j+1}$. At the final time the velocity pairing is still $O(\ell^{7/2})$ by contractivity. For $t$ large the density pairing is therefore at least $(c_0/4)\ell^{j+1}$ in absolute value. Because $\phi_\ell$ is supported within distance $\ell$ of the selected wall, dividing by $C\ell^{1/2}$ gives the lower bound on that strip itself. The upper bound \eqref{linear-trace-upper} gives the fixed-fraction assertion on squaring the two bounds.

For \eqref{incompatible-velocity}, apply the analytic and velocity estimates over the second half of the time interval, using \eqref{linear-trace-upper} at time $t/2$. This gives the first two terms and the $H^2$ bound via Stokes regularity. The mean shear decays exponentially. \qed

\subsection{Proof of Theorem~\ref{thm:model}}\label{sec:model}

For $k\ne0$, write $R=|k|$, $\beta=gk_1/R$, and let $a,r$ be the amplitudes of velocity along $(-k_2,k_1)/R$ and of density. They solve
\[
 \bar\rho\dot a=-\nu R^2 a-\beta r,\qquad \dot r=\delta\beta a .
\]
If $k_1=0$, $r$ is stationary and $a$ decays as a heat mode. For $k_1\ne0$ and sufficiently large $R$, the quadratic roots satisfy
\[
 -\lambda_+\asymp k_1^2/R^4,\qquad
 -\lambda_-\asymp \lambda_+-\lambda_-\asymp R^2 .
\]
These comparisons follow from the quadratic formula, rationalizing the slow root. In $r=d_+e^{\lambda_+t}+d_-e^{\lambda_-t}$ one has
$d_+=(\delta\beta a(0)-\lambda_-r(0))/(\lambda_+-\lambda_-)$ and
$d_-=r(0)-d_+$. Writing $A=|a(0)|+|r(0)|$ and $w=|k_1|/R^2$, this gives $|r(t)|\le C Ae^{-ctw^2}$ and, by $a=\dot r/(\delta\beta)$,
\begin{equation}\label{sharp-velocity-mode}
 |a(t)|\le C e^{-cR^2t}|a(0)|+C\frac{|k_1|}{R^3}e^{-ctw^2}A.
\end{equation}
The finitely many remaining modes with $k_1\ne0$ have strictly stable roots, including in the repeated-root case, and obey these bounds after adjusting $c,C$.

Use $w\ge R^{-2}$ for (1). The weighted slow factor in (2) is $R^{m-s-1}w e^{-ctw^2}$. If $m\le s+1$, use $R\ge w^{-1/2}$; if $m\ge s+1$, use $R\le w^{-1}$. The resulting powers of $w$ are $(s+3-m)/2$ and $s+2-m$, respectively. Maximizing after multiplication by $e^{-ctw^2}$ gives \eqref{model-smoothing}. The fast term has arbitrary fixed-order smoothing for $t\ge1$ and is bounded up to $t=0$ when $m\le s$.

For sharpness set $a(0)=0$. On the distinct real-root modes, $|r(t)|\ge |r(0)|e^{\lambda_+t}$ and
\[
 |a(t)|=\frac{|\beta r(0)|}{\bar\rho(\lambda_+-\lambda_-)}
           (e^{\lambda_+t}-e^{\lambda_-t}).
\]
Modes $(1,n)$ at $t=n^4$ give the lower velocity rate for $m\le s+1$ and give (4) on the paired cosine modes. Modes $(n,0)$ at $t=n^2$ give the rate for $s+1\le m\le s+2$. At fixed $t>0$ these latter modes have $|a(t)|\asymp_t n^{-2}|r(0)|$, excluding boundedness when $m>s+2$.

For (3), integrate \eqref{sharp-velocity-mode} and sum with weight $R$. The fast terms are summable by Cauchy--Schwarz, and the slow terms are bounded by
\[
 C\left(\sum_{k_1\ne0}\frac{|k|^{4-2s}}{k_1^2}\right)^{1/2}
   \norm{x_0}_{H^s}.
\]
The series converges for $s>5/2$ by summing in $k_2$ first. The flow estimate gives uniformly bi-Lipschitz maps. At $s=5/2$, take zero velocity and the density $\sum_{n\ge n_0} d_n\sin(x_1+nx_2)$ with $d_n=n^{-3}/\log(2+n)$ and $n_0$ in the real-root regime. This datum is in $H^{5/2}$. The velocity sine amplitudes $a_n$ are nonpositive and $\int_0^\infty(-a_n)=d_n/(\delta\beta)$. Uniform absolute convergence of the gradient series follows from $|a_n(t)|\le Cn^{-3}d_n$. At the origin all contributions to $\partial_2u_1$ have the same sign, so
\[
 \int_0^\infty\norm{\nabla\bu(t)}_{L^\infty}\,dt
 \ge \frac1{\delta g}\sum_{n\ge n_0}n^2d_n=\infty .
\]
\qed

\subsection{Comparison with stress-free Navier slip}
With stress-free Navier slip, the density admits a uniform $H^1$ estimate for constant wall traces. Let $\Omega=\T_L\times(0,h)$, $f=gx_2$ with $g>0$, and $\rho_s(x_2)=\beta-\delta_0gx_2$ with $\delta_0>0$ and $\beta>\delta_0gh$. Consider
\begin{align}\label{lin}
\begin{cases}
\rho_s\partial_t\bu=\nu\Delta\bu-\nabla P-g\varrho\,\be_2,\qquad \partial_t\varrho=\delta_0g\,u_2,\qquad \Div\bu=0,\\
u_2=\partial_2u_1=0\quad\text{on }x_2\in\{0,h\},\qquad (\bu,\varrho)|_{t=0}=(\bu_0,\varrho_0),
\end{cases}
\end{align}

\begin{proposition}\label{prop:slip}
Let $\bu_0\in H^2$ with $\Div\bu_0=0$ satisfy the boundary conditions in \eqref{lin}, let $\varrho_0\in L^2$, and set $c_\infty\coloneqq\int\rho_su_{0,1}\big/\int\rho_s$. Then:
\begin{enumerate}[(1)]
\item $\norm{\bu(t)-c_\infty\be_1}_{H^1}+\norm{\partial_t\bu(t)}\to0$, the horizontal average $\bar\varrho(t)=\bar\varrho_0$ is conserved, and
\[
 \norm{\varrho(t)-\bar\varrho_0}+\norm{\bu(t)-c_\infty\be_1}_{H^2}
 +\norm{\nabla P(t)+g\bar\varrho_0\be_2}\to0 .
\]
\item If in addition $\varrho_0\in H^1$ with a constant trace on each wall, then $\sup_t\norm{\varrho(t)}_{H^1}<\infty$, and
\[
\norm{\varrho(t)-\bar\varrho_0}_{H^s}+\norm{\bu(t)-c_\infty\be_1}_{H^2}+\norm{\nabla P(t)+g\bar\varrho_0\be_2}\to0\qquad(0\le s<1).
\]
\end{enumerate}
In particular, $\rho_s+\varrho(t)$ converges to the layered profile $\rho_s+\bar\varrho_0(x_2)$, and this profile is affine in $f$ if and only if $\bar\varrho_0$ is affine.
\end{proposition}

\begin{proof}
Horizontal momentum $\int\rho_su_1$ is conserved by periodicity and the slip conditions, and $\bar\varrho=\bar\varrho_0$ since $\bar u_2=0$. Set $\bv=\bu-c_\infty\be_1$. On the velocity space of zero weighted horizontal momentum, the slip Stokes form is coercive by Poincar\'e's inequality. Coupled to the zero-horizontal-mean density, the generator has the block form of \eqref{resolvent}; its adjoint has trivial kernel by dissipation and the characterization of layered densities. Thus the same Schur-complement and density-of-range argument gives strong $L^2$ relaxation. Analyticity gives $\partial_t\bv\to0$ in $L^2$, so the slip Stokes estimate and momentum equation give the $H^2$ and pressure limits. The energy identity applied to the state and its time derivative also gives $\nabla\bv,\nabla\bv_t\in L^2_tL^2_x$, since the initial state is in the generator domain.

For the $H^1$ bound on density, $\omega=\partial_1v_2-\partial_2v_1$ has zero wall trace. Test the curl equation with $\omega$ and the differentiated density equation with $\delta_0^{-1}\nabla\varrho$. Their buoyancy contribution is
$g\int(\partial_2v_1\partial_1\varrho-\partial_1v_1\partial_2\varrho)=0$ by integration by parts, using the constant density traces preserved by $v_2=0$. Since $\rho_s'=-\delta_0 g$, the identity is
\[
 \frac12\frac{d}{dt}\big(\norm{\sqrt{\rho_s}\omega}^2+\delta_0^{-1}\norm{\nabla\varrho}^2\big)
    +\nu\norm{\nabla\omega}^2=-\delta_0g\int\partial_tv_1\,\omega .
\]
The tests are justified by $\bv\in C(H^2)$, $\bv_t\in L^2_{\rm loc}(H^1)$, and the density equation. The right-hand side is integrable by weighted Poincar\'e and the two energy bounds. Interpolation with the strong $L^2$ limit proves (2).
\end{proof}

\begin{remark}
In general $\bu(t)$ does not converge to $0$: the uniform horizontal flow $c\be_1$ is a steady state of \eqref{lin}, and $\int\rho_su_1$ is conserved. The trace condition in Proposition~\ref{prop:slip}(2) removes the boundary term in the $H^1$ density estimate.
\end{remark}

\section{Genericity, conjecture and open problems}\label{sec:open}

\subsection{Genericity as a Baire-category question}

Fix a domain and potential satisfying \ref{A1} (or the periodic-channel version), the bounds $0<\alpha_1\le\alpha_2$, and $R>0$. Let $\mathfrak D_R$ be the set of data $(\bu_0,\rho_0)$ satisfying \ref{A2}--\ref{A3} with $\norm{\bu_0}_{H^2}\le R$. Equip it with the metric $\norm{\bu_0-\bu_0'}_{H^1}+\norm{\rho_0-\rho_0'}_{L^1}$. It is a complete metric space: the $H^2$ ball is closed in $H^1$ by weak compactness, and the density bounds are closed in $L^1$.

\begin{proposition}\label{prop:baire}
Assume \ref{A4}. The relaxing set
\[
\mathcal R_R\coloneqq\{x_0\in\mathfrak D_R:\ \rho(t)\to(\rho_0)^\natural\text{ in }L^2\}
\]
is a $G_\delta$ subset of $\mathfrak D_R$. Hence $\mathcal R_R$ is residual (relaxation is generic in the sense of Baire) if and only if it is dense. In the channel and disk configurations of Theorem~\ref{thm:nonmono}, its complement contains the frozen unsorted data from that theorem whenever the density bounds permit a nonconstant profile.
\end{proposition}

\begin{proof}
\emph{Continuity of $E(t)$.} Let $x_0^n\to x_0$ in $\mathfrak D_R$. The uniform density bounds upgrade initial $L^1$ convergence to $L^{4/3}$ convergence, which embeds in $W^{-1,4}$. Proposition~\ref{prop:euler-dependence} gives velocity convergence in $L^2$ and density convergence in $W^{-1,4}$ at each finite time. This density convergence and the uniform bounds imply weak $L^2$ convergence, and the conserved norms imply strong $L^2$ convergence. Hence $x_0\mapsto E(t;x_0)$ is continuous for every $t$.

\emph{Semicontinuity of the deficit.} Then $E_\infty=\inf_tE(t)$ is upper semicontinuous. The map $x_0\mapsto E^\natural=\int\rho_0^\natural f$ is continuous by Lemma~\ref{lem:nonexp}. So the deficit $\mathcal D\coloneqq E_\infty-E^\natural\ge0$ is upper semicontinuous, and each set $\{\mathcal D<1/n\}$ is open.

\emph{Conclusion.} By Theorem~\ref{thm:lorenz}(2), $\mathcal R_R=\{\mathcal D=0\}=\bigcap_n\{\mathcal D<1/n\}$, which is a $G_\delta$. The last statement follows from Theorem~\ref{thm:nonmono}(3).
\end{proof}

\subsection{Small perturbations of stable stratifications}

The invariant shear and swirl solutions in Theorem~\ref{thm:nonmono} show that the density distribution alone does not determine the limiting profile. In the periodic channel, eventual monotonicity of the horizontal mean and uniform $H^s$ control for some $s>0$ select the sorted state. Proposition~\ref{prop:baire} identifies density of the relaxing set as the condition for generic relaxation. Finding sufficient conditions for precompactness of the density orbit and determining whether RT perturbations can produce mixed weak limits remain open.

The linearized problem converges to the hydrostatic projection. The first variation of sorting agrees with the hydrostatic projection under the hypotheses of Lemma~\ref{lem:firstvar}. In the gravitational channel, Proposition~\ref{prop:secondvar} identifies the quadratic term of the sorting map and shows that the leading available potential energy is exactly the density energy of the nonhydrostatic linear perturbation. The algebraic stability results for the Boussinesq and Stokes--transport systems \cite{Park2024,DGL2025} motivate the following conjecture for the full system.

\begin{conjecture}
Let $\Omega=\T\times(0,1)$, $f=gx_2$, and $\rho_s=\varphi(f)$ be a positive $H^4$ profile with $-\varphi'\ge\delta_0>0$. If $\rho_0-\rho_s\in H^2_0\cap H^4$ and $\varepsilon\coloneqq\norm{\rho_0-\rho_s}_{H^4}+\norm{\bu_0}_{H^4}\ll1$, then $\norm{\rho(t)-\rs}+\norm{\bu(t)}\lesssim\varepsilon(1+t)^{-1}$.
\end{conjecture}

Theorem~\ref{thm:wall-rate} proves the conjectured decay rate and its small-data prefactor under the additional assumption \eqref{wall-reg}. Thus a time-uniform small $H^4$ bound for the density perturbation would imply the conjecture for the stated compatible data. The conditional estimate requires the density trace to be constant on each wall. Lemma~\ref{lem:interp2} handles the dual estimate for the density fluctuation without a normal-derivative condition. Proposition~\ref{thm:wall-obstruction} shows why convergence to a target with incompatible wall values cannot follow from a uniform regularity argument.

A separate sufficient condition for sorting is a small accumulated strain bound of the form
\begin{equation}\label{lip-conj}
\int_0^\infty\norm{\nabla\bu(t)}_{L^\infty}\,dt
  \le C\big(\norm{\rho_0-\rho_s}_{H^N}+\norm{\bu_0}_{H^N}\big)
\end{equation}
for suitable $N$ and data with constant wall traces. Corollary~\ref{thm:lipred} then yields convergence with the rates of Theorem~\ref{thm:rate}. Theorem~\ref{thm:lin}(5) gives a finite integrated Lipschitz norm for displacement-generated data in the no-slip linearization. In the periodic Boussinesq model, Theorem~\ref{thm:model}(3) identifies the exact Sobolev threshold. For the nonlinear no-slip problem, the remaining estimate is the uniform propagation of $H^4$ regularity with the variable-inertia terms. Closing that estimate requires control of derivatives of the variable-inertia terms in the nonlinear equations.

 \section*{Acknowledgement}
The authors acknowledge the use of generative AI tools in developing and checking mathematical arguments and in drafting and revising the manuscript. Q. Wang was supported by the Natural Science Foundation of
Sichuan Province (No. 2025ZNSFSC0072).

\begingroup
\small
\let\standardbibliography\thebibliography
\renewcommand{\thebibliography}[1]{%
  \standardbibliography{#1}\setlength{\itemsep}{0.25pt plus 0.4pt}\setlength{\parsep}{0pt}}

\endgroup

\end{document}